\documentclass[11pt]{amsart}
\usepackage{graphicx} % Required for inserting images
\usepackage{amsfonts,epsfig}
\usepackage{latexsym}
\usepackage{amssymb}
\usepackage{amsmath}
\usepackage{amsthm}
\usepackage{graphics}
\usepackage[all]{xy}
\usepackage[T2A]{fontenc}
\usepackage{multirow}
\usepackage{hhline}
\usepackage{tikz-cd}
\usepackage{array, booktabs}
\usepackage[backref]{hyperref}
\usepackage{enumerate}
\newtheorem{defn}{Definition}[section]
\newtheorem{definition}[defn]{Definition}

\newtheorem{corollary}[defn]{Corollary}
\newtheorem{lemma}[defn]{Lemma}

\newtheorem{thm}[defn]{Theorem}
\newtheorem{theorem}[defn]{Theorem}
\newtheorem{cor}[defn]{Corollary}
\newtheorem{prop}[defn]{Proposition}
\newtheorem{proposition}[defn]{Proposition}

\newtheorem{algorithm}[defn]{Algorithm}

\theoremstyle{definition}

\newtheorem*{ack}{Acknowledgements}
\newtheorem{remark}[defn]{Remark}

\newtheorem{example}[defn]{Example}

\newcommand{\Q}{\mathbb{Q}}
\newcommand{\Z}{\mathbb{Z}}
\newcommand{\N}{\mathbb{N}}

\newcommand{\OK}{\mathcal{O}_K}
\newcommand{\Of}{\mathcal{O}_{K,f}}
\newcommand{\cC}{\mathcal{C}}
\newcommand{\OO}{\mathcal{O}}

\newcommand{\Rats}{\mathbb{Q}}

\newcommand{\SL}{\operatorname{SL}}

\newcommand{\arrow}{\longrightarrow}

\newcommand{\Gal}{\operatorname{Gal}}
\newcommand{\Aut}{\operatorname{Aut}}
\newcommand{\GQ}{\Gal(\overline{\Rats}/\Rats)}

\newcommand{\GL}{\operatorname{GL}}

\newcommand{\lcm}{\operatorname{lcm}}
\newcommand{\End}{\operatorname{End}}
\DeclareMathOperator{\disc}{disc} %discriminant
\DeclareMathOperator{\im}{im}  %for image
\DeclareMathOperator{\Id}{Id}  %for identity

\newcommand{\eclabel}[1]{\href{https://www.lmfdb.org/EllipticCurve/Q/#1}{\texttt{#1}}}

\title[The adelic Galois image of an elliptic curve with complex multiplication]{The image of the adelic Galois representation of an elliptic curve with complex multiplication}

\author{\'Alvaro Lozano-Robledo and Benjamin York}

\begin{document}

\begin{abstract}
Let $E/\mathbb{Q}$ be an elliptic curve and let $\rho_E \colon \operatorname{Gal}(\overline{\mathbb{Q}}/\mathbb{Q}) \to \operatorname{GL}(2, \widehat{\mathbb{Z}})$ be the adelic Galois representation attached to $E$. Much work has been done in recent years to study the image of $\rho_E$ (up to conjugation) as part of Mazur's so called ``Program B.'' In this paper, we describe and implement an efficient algorithm to compute the image of $\rho_E$ in $\operatorname{GL}(2, \widehat{\mathbb{Z}})$ (up to conjugation) for an elliptic curve $E/\mathbb{Q}$ with complex multiplication (CM) and $j$-invariant not $0$ or $1728$.
\end{abstract}

\maketitle

\section{Introduction}

Let $F$ be a number field, let $E/F$ be an elliptic curve, let $\ell$ be a prime number, and let $N\geq 2$ be an integer. Let $T_\ell(E)=\varprojlim E[\ell^n]$ and $T(E)=\varprojlim E[N]$ be, respectively, the $\ell$-adic and the adelic Tate modules of $E/F$. Then, $T_\ell(E)$ and $T(E)$ are $\Gal(\overline{F}/F)$-modules, and the associated Galois representations
\begin{align*}
   \rho_{E,\ell^{\infty}} &\colon \Gal(\overline{F}/F)\to \Aut(T_\ell(E))\cong \GL(2,\Z_\ell)\\
      \rho_{E} &\colon \Gal(\overline{F}/F)\to \Aut(T(E))\cong \GL(2,\widehat{\Z})
\end{align*}
are objects of great interest in modern algebraic number theory. There is indeed much literature (see, for example, \cite{2adic}, \cite{elladic} and \cite{zywina}) on the problem of classifying the possible images of $\rho_{E,\ell^\infty}$ and $\rho_E$ (up to conjugation as a subgroup of $\GL(2,\Z_\ell)$ and $\GL(2,\widehat{\Z})$, respectively) of such Galois representations, which is usually referred to as Mazur's ``Program B'' after \cite{mazurprogramb}. In particular, Zywina has recently described an algorithm to compute the adelic image of $\rho_E$ when $E/\Q$ in \cite{zywinaadelic1}, and more generally for $E/F$ over a number field in \cite{zywinaadelic2}, provided that $E$ does not have CM. 

For an elliptic curve $E/F$ with complex multiplication by an order $\OO$, the first author (\cite{lozano-galoiscm}; see Section \ref{sec-alvaroCMresults}) described all the possible $\ell$-adic images of Galois representations attached to $E$ when $F=\Q(j(E))$, and  Rouse, Sutherland, and Zureick-Brown have described an efficient algorithm to describe such images for a fixed curve $E/\Q(j(E))$ (see \cite[Section 12]{elladic}). The authors of the current paper together with Gonz\'alez-Jim\'enez (\cite{modelspaper}) have characterized the Weierstrass models of curves with each possible $\ell$-adic image. Bourdon and Clark (e.g., \cite{bourdon2},
\cite{bourdon}, \cite{bourdon4}) have extensively analyzed the properties of division fields and isogenies of CM elliptic curves and, in particular in  \cite[Cor. 1.5]{bourdon2}, they proved a uniform bound for the  index of the adelic Galois representation $\rho_E$ as a subgroup of the maximal subgroup $\mathcal{N}_{\delta,\phi}$ of $\Aut(T(E))$ allowed by the CM order $\OO$ (see Section \ref{sec-notation} for an explicit matrix description of $\mathcal{N}_{\delta,\phi}$). Lombardo (\cite{lombardo}) has studied division fields and shown a similar bound for adelic Galois representations of abelian varieties of CM type. Campagna and Pengo in \cite{CampagnaPengo2}  have determined the index of the adelic Galois representation $\rho_E$ as a subgroup  $\mathcal{N}_{\delta,\phi}$   and in \cite{CampagnaPengo1}  they study the possible entanglements between the $\ell$-adic division fields of $E$. 

The goal of this paper is to describe an efficient computational method to explicitly describe the adelic image of an elliptic curve $E/\Q$ with complex multiplication, when $j(E)\neq 0,1728$, as a subgroup of $\GL(2,\widehat{\Z})$ up to conjugation. The description of the adelic image in the cases of $j(E)=0,1728$ and $j(E)\not\in\Q$ is much more delicate, and it will be described by the authors in upcoming work. 

The main result of this paper is as follows.

\begin{thm}\label{main-result}
    Let $E/\Q$ be an elliptic curve with complex multiplication by an order $\Of$ of an imaginary quadratic field $K$, such that $j(E)\neq 0,1728$. Let $G_E$ be the image of the Galois representation $\rho_E\colon \GQ\to \GL(2,\widehat{\Z})$. Then,
    \begin{enumerate}
        \item There is a group $\mathcal{N}_{\delta,\phi}$ of $\GL(2,\widehat{\Z})$ such that $G_E\subseteq\mathcal{N}_{\delta,\phi}$ with index $[\mathcal{N}_{\delta,\phi}:G_E]=2$. 
        \item There is an explicitly computable integer $M\geq 2$ such that $G_E=\pi_M^{-1}(\pi_M(G_E))$, where $\pi_M\colon \mathcal{N}_{\delta,\phi}\to \mathcal{N}_{\delta,\phi}(\Z/M\Z)$ is the natural reduction mod-$M$ map. In other words, $G_E$ is defined modulo $M$.
        \item The image $\pi_M(G_E)$ is explicitly computable.
    \end{enumerate}
\end{thm}

The auxiliary constants $\delta$ and $\phi$, and the groups $\mathcal{N}_{\delta,\phi}$ will be defined in Section \ref{sec-notation}. For now, we provide two examples:
\begin{example}\label{ex-intro1}
    Let $E/\Q$ be the elliptic curve given by $y^2+xy=x^3-x^2-107x+552$ (with LMFDB label \href{https://www.lmfdb.org/EllipticCurve/Q/49/a/2}{49.a2}; see \cite{lmfdb} for more information on the LMFDB). This elliptic curve has CM by $\Z[(1+\sqrt{-7})/2]$, the maximal order of $\Q(\sqrt{-7})$. The mod-$7$ image of $\rho_{E,7}$ is of index $2$
 in the maximal possible image $\mathcal{N}_{-2,1}(\Z/7\Z)$ allowed by the CM order (this follows from, for example, \cite{modelspaper}, Table 4). In particular, $E/\Q$ is what we call a simplest CM curve (see Section \ref{sec-simplestCMcurves}), and therefore, the adelic level of definition of $\rho_E$ is $7$ and the adelic image is the full inverse image of $G_{E,7}$ in $\mathcal{N}_{-2,1}(\widehat{\Z})$ via the natural projection map  $\mathcal{N}_{-2,1}(\widehat{\Z})\to \mathcal{N}_{-2,1}(\Z/7\Z)$ given by reduction modulo $7$. Hence, the adelic image is given by
 $$G_{E}=\left\langle \left\{ \left(\begin{array}{cc} a+b & b\\ -2b & a \\\end{array}\right) : a,b \in \widehat{\Z}, a^2+ab+2b^2 \in \widehat{\Z}^\times, a+b/2 \in ((\Z/7\Z)^\times)^2\right\}, \left(\begin{array}{cc} 1 & 0\\ -1 & -1 \\\end{array}\right)\right\rangle$$
 in $\GL(2,\widehat{\Z})$. We will come back to this curve in Example \ref{ex-intro1-details} and show that the adelic image is as described here.
 \end{example}

 \begin{example}\label{ex-intro2}
    Let $E'/\Q$ be the elliptic curve given by $y^2+xy+y=x^3-x^2-965x-13940$ 
 (\href{https://www.lmfdb.org/EllipticCurve/Q/441/c/2}{441.c2}), which is a quadratic twist by $-3$ of $E$ in Example \ref{ex-intro1}. Thus, the elliptic curve $E'$ also has CM by $\Z[(1+\sqrt{-7})/2]$, the maximal order of $\Q(\sqrt{-7})$. The image of $\rho_{E',7^\infty}$ is the maximal possible image allowed by the CM order, namely $\mathcal{N}_{-2,1}(7^\infty)$. However, the image modulo $21$ is conjugate to
    $$G_{E',21}=\left\langle \left(\begin{array}{cc} -1 & 0\\ 1 & 1 \\\end{array}\right), \left(\begin{array}{cc} 1 & 10\\ 1 & 12 \\\end{array}\right)\right\rangle$$
    that has index $2$ in $\mathcal{N}_{-2,1}(\Z/21\Z)$. In this case, the level of definition of the adelic image is $21$ and image of $\rho_E$ is the full inverse image of $G_{E',21}$ via the natural reduction map $\mathcal{N}_{-2,1}(\widehat{\Z})\to \mathcal{N}_{-2,1}(\Z/21\Z)$.
\end{example} 

While this paper concerns mostly those elliptic curves $E/\Q$ with CM and $j(E)\neq 0,1728$, we are also able to compute the adelic images of simplest curves with $j(E)=0,1728$ defined over $\Q$. Here is one example of a curve with $j(E)=1728$. 

 \begin{example} Let $E/\Q$ be given by $y^2=x^3+x$ (\href{https://www.lmfdb.org/EllipticCurve/Q/64/a/4}{\texttt{64.a4}}). This elliptic curve has CM by $\Z[i]$, the maximal order of $\Q(i)$, and it is a $2$-simplest elliptic curve. According to \cite[Table 4]{modelspaper}, the $2$-adic image $G_{E,2^\infty}$ of $E/\Q$ was computed in  \cite[Table 4]{modelspaper} and it is defined modulo $2^{n_{E,2}}=16$. Therefore, by Theorem \ref{thm-adelicimageofsimplestcurves}, the adelic image $G_E$ of $E$ is given by
   $$\left\langle \left\{ c_{-1,0}(a,b)=\left(\begin{array}{cc} a & b\\ -b & a \\\end{array}\right) : a,b \in \widehat{\Z}, a^2+b^2 \in \widehat{\Z}^\times, (c_{-1,0}(a,b) \bmod 16) \in G_{E,16}\right\}, \left(\begin{array}{cc} 0 & -1\\ -1 & 0 \\\end{array}\right)\right\rangle$$
   as a subgroup of $\GL(2,\widehat{\Z})$. We will come back to this example in Example \ref{ex-intro3-details}.
 \end{example}

The paper is organized as follows. In Section \ref{sec-notation} we introduce basic notation and recall previous results from the literature on the classification of Galois representations in the CM case, and we recall some facts about division fields. In Section \ref{sec-level} we define the concept of ($\ell$-adic and adelic) level of definition of a Galois representation attached to CM elliptic curves. In Section \ref{sec-simplestCMcurves} we define  simplest CM elliptic curves whose adelic images are simplest to compute, and are used to compute the image of any other CM curve via twisting. In Section \ref{sec-determinelevelofdef} we prove results about the level of definition of an adelic image. In Section  \ref{sec-complexconj} we determine the image of complex conjugation under $\rho_E$, and in Section  \ref{sec-imageofcartan} we determine the image of the Cartan subgroup. In Section \ref{sec-proofofmaintheorem} we prove Theorem \ref{main-result} and we describe Algorithm \ref{alg-main} which computes a level of definition $M$ and the image modulo $M$. Finally, in Section \ref{sec-levelsofdiff} we show that the adelic level of definition $M$ computed in Algorithm \ref{alg-main} differentiates images up to conjugation (see Definition \ref{def-level-of-dif}).

\subsection{Code}

Algorithm \ref{alg-main} has been implemented in \verb|Magma| \cite{magma}. The implementation of this algorithm can be found in the GitHub repository \cite{githubrepo}
\begin{center}
    \url{https://github.com/benjamin-york/CM_adelic_Galois_images}
\end{center}

\begin{ack}
    The authors would like to thank Francesco Campagna, Riccardo Pengo, Drew Sutherland, and David Zywina for their thoughts and comments on previous versions of this paper. The second author would also like to thank Shiva Chidambaram, Asimina Hamakiotes, and David Roe for many valuable discussions during the completion of this project.
\end{ack}

\section{Notation and previous results}\label{sec-notation}

In this section, we will set the notation that we will follow in the rest of the paper, define the matrix groups that are the maximal possible images in the CM case, and recall results from previous articles that we will use in later sections. Our main references are \cite{lozano-galoiscm} and \cite{modelspaper}.

Let $K$ be an imaginary quadratic field, and let $\OK$ be the ring of integers of $K$ with discriminant $\Delta_K$. Let $f\geq 1$ be an integer and let $\OO_{K,f}$ be the order of $K$ of conductor $f$. Let $j_{K,f}$ be a CM $j$-invariant corresponding to the order $\OO_{K,f}$ (thus, $j_{K,f}$ is an algebraic integer and every CM $j$-invariant with CM by $\OO_{K,f}$ is a Galois conjugate of $j_{K,f}$) and let $N\geq 2$. We define associated constants $\delta$ and $\phi$ as follows:
\begin{itemize}
    \item If $\Delta_Kf^2\equiv 0\bmod 4$, let $\delta=\Delta_K f^2/4$, and $\phi=0$.
	\item If $\Delta_Kf^2\equiv 1 \bmod 4$, let $\delta=\frac{(\Delta_K-1)}{4}f^2$, let $\phi=f$.
\end{itemize}
We define matrices in $\GL(2,\Z/N\Z)$ by
$$c_\varepsilon=\left(\begin{array}{cc} \varepsilon & 0\\ -\varepsilon\phi & -\varepsilon\\\end{array}\right),\quad c'_\varepsilon = \left(\begin{array}{cc} 0 & \varepsilon \\ \varepsilon & 0 \\ \end{array}\right), \quad \text{ and } \quad c_{\delta,\phi}(a,b)=\left(\begin{array}{cc}a+b\phi & b\\ \delta b & a\\ \end{array}\right)$$ where $\varepsilon \in \{\pm 1\}$ and $a,b \in \Z/N\Z$ such that $\det(c_{\delta,\phi}(a,b))\in (\Z/N\Z)^\times$. 
We define the Cartan subgroup $\cC_{\delta,\phi}(N)$ of $\GL(2,\Z/N\Z)$ by
$$\cC_{\delta,\phi}(N)=\left\{c_{\delta,\phi}(a,b): a,b\in\Z/N\Z,\  \det(c_{\delta,\phi}(a,b)) \in (\Z/N\Z)^\times \right\},$$
and $\mathcal{N}_{\delta,\phi}(N) = \left\langle \cC_{\delta,\phi}(N),c_1\right\rangle$. We will sometimes call $\mathcal{N}_{\delta,\phi}(N)$ the ``normalizer'' of $\cC_{\delta,\phi}(N)$ in $\GL(2,\Z/N\Z)$. In the case of a prime number $N$, then $\mathcal{N}_{\delta,\phi}(N)$ is the true group-theoretic normalizer, but this is often not the case for composite values of $N$ (see \cite{lozano-galoiscm}, Section 5). Finally, we write $\mathcal{N}_{\delta,\phi} = \varprojlim \mathcal{N}_{\delta,\phi}(N)$, where the projective limit is taken over the natural projection maps $\mathcal{N}_{\delta,\phi}(N)\to \mathcal{N}_{\delta,\phi}(M)$ given by reduction mod-$N$ when $M|N$, and regard the inverse limit as a subgroup $\mathcal{N}_{\delta,\phi}\subseteq \GL(2,\widehat{\Z})$. If $\ell$ is a prime, we write  $\cC_{\delta,\phi}(\ell^{\infty}) = \varprojlim \cC_{\delta,\phi}(\ell^n)$ and $\mathcal{N}_{\delta,\phi}(\ell^{\infty}) = \varprojlim \mathcal{N}_{\delta,\phi}(\ell^n)$.

\begin{remark}
    If $\Delta_Kf^2\equiv 1 \bmod 4$ and $N$ is odd, one can find a basis of $E[N]$ such that the image of $\rho_{E,N}$ is a subgroup of $\mathcal{N}_{\delta,0}(N)$ with $\delta=\Delta_Kf^2/4$, which allows for simpler (and somewhat more familiar) descriptions of the images (see \cite[Thm. 1.1]{lozano-galoiscm}, and  Remark \ref{rem-changebasis}). However, under the same assumptions of $\Delta_Kf^2\equiv 1 \bmod 4$, we need to use the choices $\delta=\frac{(\Delta_K-1)}{4}f^2$ and $\phi=f$ in order to build a  $\widehat{\Z}$-basis that is compatible modulo $N$, for every $N>1$.
\end{remark}

\subsection{Notation for Galois representations}\label{sec-galoisrepnotation}
Let $E/\Q(j_{K,f})$ be an elliptic curve with CM by $\OO_{K,f}$. We define several Galois representations attached to $E$, as follows. Let  $G_{\Q(j_{K,f})}$  be the absolute Galois group $\Gal(\overline{\Q(j_{K,f})}/\Q(j_{K,f}))$ of the field of definition $\Q(j_{K,f})$, let $N\geq 2$ be an integer, and let $\ell$ be a prime. The natural action of $G_{\Q(j_{K,f})}$ on $E[N]$, on the $\ell$-adic  Tate module $T_\ell(E)=\varprojlim E[\ell^n]$, and on the adelic Tate module $T(E)=\varprojlim E[N]$ induce Galois representations
\begin{align*}
    \rho_{E,N} &\colon G_{\Q(j_{K,f})} \to \Aut(E[N]),\\
    \rho_{E,\ell^\infty} &\colon G_{\Q(j_{K,f})} \to \Aut(T_\ell(E)),\\
    \rho_{E} &\colon G_{\Q(j_{K,f})} \to \Aut(T(E)).
\end{align*}
Moreover, once we choose $\Z/N\Z$, $\Z_\ell$, and $\widehat{\Z}$-bases of $E[N]$, $T_\ell(E)$, and $T(E)$, respectively, we obtain non-canonical isomorphisms of the automorphism groups with the groups of matrices $\GL(2,\Z/N\Z)$, $\GL(2,\Z_\ell)$, and $\GL(2,\widehat{\Z})$. Thus, we will usually consider the images of the Galois representations as subgroups of matrices in $\GL_2$.  

We will typically denote the images of $\rho_{E,N}$, $\rho_{E,\ell^\infty}$, and $\rho_E$ by $G_{E,N}$, $G_{E,\ell^\infty}$, and $G_E$, respectively. The results stated below show that there exists an appropriate choice of basis such that $G_{E,\ell^\infty} \subseteq \mathcal{N}_{\delta,\phi}(\ell^{\infty})$, and $G_E\subseteq \mathcal{N}_{\delta,\phi}\subseteq \GL(2,\widehat{\Z})$.

Finally, we will often need to restrict the Galois representations to the index-$2$ subgroup $G_{K(j_{K,f})}=\Gal(\overline{\Q(j_{K,f})}/K(j_{K,f}))$ of the absolute Galois group $G_{\Q(j_{K,f})}$, and we will denote the restrictions by $\rho_{E/K,N}$, $\rho_{E/K,\ell^\infty}$, and $\rho_{E/K}$, and their images will be denoted by  $G_{E/K,N}$, $G_{E/K,\ell^\infty}$, and $G_{E/K}$, respectively. Our choice of $\widehat{\Z}$-basis will be such that $G_{E/K}\subseteq \mathcal{C}_{\delta,\phi}\subseteq \GL(2,\widehat{\Z})$.

\subsection{Classification of \texorpdfstring{$\ell$}{ell}-adic images} \label{sec-alvaroCMresults}

For a prime $\ell$, the first author (in \cite{lozano-galoiscm}) has given a complete classification of the possible $\ell$-adic images $G_{E,\ell^\infty} = \rho_{E,\ell^{\infty}}(G_{\Q(j_{K,f})})$ attached to elliptic curves $E/\Q(j_{K,f})$ with CM by $\OO_{K,f}$, which we summarize next.

\begin{theorem}[\cite{lozano-galoiscm}, Theorems 1.1, 1.2, and 4.1]\label{thm-cmrep-intro-alvaro}\label{thm-firstofsec1} Let $E/\Q(j_{K,f})$ be an elliptic curve with CM by $\OO_{K,f}$, and let $N\geq 2$.
	 Then:
	\begin{enumerate}
            \item There is a $\Z/N\Z$-basis of $E[N]$ such that the image $G_{E,N}$ of $\rho_{E,N}$
	is contained in $\mathcal{N}_{\delta,\phi}(N)$, and	 the index of the image of $\rho_{E,N}$ in $\mathcal{N}_{\delta,\phi}(N)$ is a divisor of the order of $\Of^\times/\mathcal{O}_{K,f,N}^\times$, where $\mathcal{O}_{K,f,N}^\times=\{u\in\Of^\times: u\equiv 1 \bmod N\Of\}$.
    \item With the basis chosen as in (1), and if the image of $\rho_{E,N}$ is $G_{E,N}\subseteq \mathcal{N}_{\delta,\phi}(N)$, then the image of $\rho_{E/K,N}$ is $G_{E,N}\cap \mathcal{C}_{\delta,\phi}(N)$. 
		\item There is a compatible system of bases of $E[N]$ such that the image of $\rho_E$ is contained in $\mathcal{N}_{\delta,\phi}$, and the index of the image of $\rho_{E}$ in $\mathcal{N}_{\delta,\phi}$ is a divisor of $|\Of^\times|$. In particular, the index is a divisor of $4$ or $6$. 
	\end{enumerate} 
\end{theorem}

The following theorem shows that, for a prime $\ell>3$, the image of $\rho_{E,\ell^\infty}$ is, in fact, defined modulo $\ell$. Moreover, if the prime $\ell$ does not divide $2\Delta_K f$, then the $\ell$-adic image is as large as possible.

\begin{thm}[\cite{lozano-galoiscm}, Theorem 1.2]\label{thm-goodredn}
Let $E/\Q(j_{K,f})$ be an elliptic curve with CM by $\OO_{K,f}$, and let $\ell$ be prime. If $\ell > 3$, or $\ell > 2$ and $j_{K,f} \neq 0$, then $G_{E,\ell^\infty}$ is the full inverse image via the natural reduction mod-$\ell$ map $\mathcal{N}_{\delta,\phi}(\ell^\infty)\to \mathcal{N}_{\delta,\phi}(\ell)$ of the image $G_{E,\ell}$ of $\rho_{E,\ell}\equiv \rho_{E,\ell^\infty}\bmod \ell$. Further, if $\ell \nmid 2 \Delta_K f$, then $G_{E,\ell} = \mathcal{N}_{\delta, \phi}(\ell)$, and so $G_{E,\ell^{\infty}} = \mathcal{N}_{\delta, \phi}(\ell^{\infty})$.
\end{thm}

% In the following theorems, we describe certain explicit subgroups of $\GL(2,\Z_\ell)$ that correspond to the possible images (up to conjugation) of $\rho_{E,\ell^{\infty}}$. The group $G_{\delta,\phi}^{\, i,t}=G_{\delta,\phi}^{\, i,t}(\ell^\infty)$
% will denote a subgroup of $\mathcal{C}_{\delta,\phi}(\ell^{\infty})$ of index $i$, so $[\mathcal{C}_{\delta,\phi}(\ell^{\infty}) : G_{\delta,\phi}^{\, i,t}(\ell^\infty)] = i$. When there are several possible subgroups of $\mathcal{C}_{\delta,\phi}(\ell^{\infty})$ with index $i$, we use a parameter $t \in \{1,2,3,4\}$ to differentiate them (where the choice of numbering was arbitrary).

The following result describes the $\ell$-adic image in the case of a prime $\ell>2$ that divides $\Delta_K f$.

\begin{thm}[\cite{lozano-galoiscm}, Theorem 1.5]\label{thm-j0ell3alvaro}\label{thm-oddprimedividingdisc}
	Let $E/\Q(j_{K,f})$ be an elliptic curve with CM by $\Of$, and let $\ell$ be an odd prime dividing $f\Delta_K$ (so, in particular, $j_{K,f} \neq 1728$). Then $G_{E,\ell^\infty}$ is precisely one of the following groups:
	\begin{enumerate}
		\item[(a)] If $j_{K,f}\neq 0,1728$, then either $G_{E,\ell^\infty}=\mathcal{N}_{\delta,0}(\ell^\infty)$, or $G_{E,\ell^\infty}$ is generated by $c_{\varepsilon}$ and the group
		 \[ J_{\delta, 0} = \left\{ \left(\begin{array}{cc} a^2 & b\\ \delta b & a^2 \\ \end{array}\right): a \in \Z_\ell^\times, b \in \Z_\ell \right\}.\] 
		\item[(b)] If $j_{K,f}=0$, then $\ell=3$, and there are twelve possibilities for $G_{E,3^\infty}$. More concretely, either $G_{E,3^\infty}=\mathcal{N}_{-3/4,0}(3^\infty)$, or $[\mathcal{N}_{-3/4,0}(3^\infty):G_{E,3^\infty}]=2,3,$ or $6$, and one of the following holds:
		\begin{enumerate} \item[(i)] If $[\mathcal{N}_{-3/4,0}(3^\infty):G_{E,3^\infty}]=2$, then $G_{E,3^\infty}$ is generated by $c_{\varepsilon}$ and 
		\[ \left\{ \left(\begin{array}{cc}  a & b\\ -3b/4 & a\\ \end{array}\right): a,b\in  \Z_3,\ a\equiv 1 \bmod 3\right\}. \]
		
		\item[(ii)]  If $[\mathcal{N}_{-3/4,0}(3^\infty):G_{E,3^\infty}]=3$, then $G_{E,3^\infty}$ is generated by $c_{\varepsilon}$ and 
		\[ \left\{ \left(\begin{array}{cc} a & b\\ -3b/4 & a\\ \end{array}\right): a\in \Z_3^\times,\ b\equiv 0 \bmod 3 \right\}, \]
		\[ \text{or }\ \left\langle \left(\begin{array}{cc}  2 & 0\\ 0 & 2\\ \end{array}\right) ,\left(\begin{array}{cc}  1 & 1\\ -3/4 & 1\\ \end{array}\right)\right\rangle,\ \text{ or } \left\langle \left(\begin{array}{cc}  2 & 0\\ 0 & 2\\ \end{array}\right) ,\left(\begin{array}{cc}  -5/4 & 1/2\\ -3/8 & -5/4\\ \end{array}\right)\right\rangle. \] 
		\item[(iii)]  If $[\mathcal{N}_{-3/4,0}(3^\infty):G_{E,3^\infty}] = 6$, then $G_{E,3^\infty}$ is generated by $c_{\varepsilon}$ and one of
		\[  \text{or }\ \left\{ \left(\begin{array}{cc}  a & b\\ -3b/4 & a\\ \end{array}\right) : a\equiv 1,\ b\equiv 0 \bmod 3\Z_3 \right\}, \]
		\[ \text{or }\ \left\langle \left(\begin{array}{cc}  4 & 0\\ 0 & 4\\ \end{array}\right) ,\left(\begin{array}{cc}  1 & 1\\ -3/4 & 1\\ \end{array}\right)\right\rangle,\ \text{ or } \ \left\langle \left(\begin{array}{cc}  4 & 0\\ 0 & 4\\ \end{array}\right) ,\left(\begin{array}{cc}  -5/4 & 1/2\\ -3/8 & -5/4\\ \end{array}\right)\right\rangle. \]
		\end{enumerate}
	\end{enumerate}
\end{thm}

\begin{remark}\label{rem-changebasis}
    Let $\ell$ be an odd prime, and let $\delta,\phi$ be the constants defined above for an order $\Of$. The matrix equalities 
    \begin{align*} \left(\begin{array}{cc}  1 & 0\\ \phi/2 & 1\\ \end{array}\right)\left(\begin{array}{cc}  a+b\phi & b\\ \delta b & a\\ \end{array}\right) \left(\begin{array}{cc}  1 & 0\\ -\phi/2 & 1\\ \end{array}\right) &= \left(\begin{array}{cc}  a+b\phi/2 & b\\ (\delta + \phi^2/4)b & a+b\phi/2\\ \end{array}\right), \\
    \left(\begin{array}{cc}  1 & 0\\ -\phi/2 & 1\\ \end{array}\right)\left(\begin{array}{cc}  a & b\\ \delta b & a\\ \end{array}\right)\left(\begin{array}{cc}  1 & 0\\ \phi/2 & 1\\ \end{array}\right) &= \left(\begin{array}{cc}  (a-b\phi/2)+b\phi & b\\ (\delta-\phi^2/4) b & a-b\phi/2\\ \end{array}\right),\\
    \left(\begin{array}{cc}  1 & 0\\ \phi/2 & 1\\ \end{array}\right)\left(\begin{array}{cc}  \varepsilon & 0\\ -\varepsilon\phi & -\varepsilon\\ \end{array}\right) \left(\begin{array}{cc}  1 & 0\\ -\phi/2 & 1\\ \end{array}\right) &= \left(\begin{array}{cc}  \varepsilon & 0\\ 0 & -\varepsilon\\ \end{array}\right), \text{ and}\\
    \left(\begin{array}{cc}  1 & 0\\ -\phi/2 & 1\\ \end{array}\right)\left(\begin{array}{cc}  \varepsilon & 0\\ 0 & -\varepsilon\\ \end{array}\right) \left(\begin{array}{cc}  1 & 0\\ \phi/2 & 1\\ \end{array}\right) &= \left(\begin{array}{cc}  \varepsilon & 0\\ -\varepsilon\phi  & -\varepsilon \\ \end{array}\right)
    \end{align*}
    show that $\mathcal{N}_{\delta,\phi}(\ell^\infty)$ is conjugate to $\mathcal{N}_{\delta+\phi^2/4,0}(\ell^\infty)$ when $\delta = (\Delta_K -1)f^2/4$ and $\phi = f$, and that $\mathcal{N}_{\delta,0}(\ell^\infty)$ is conjugate to $\mathcal{N}_{\delta-\phi^2/4,\phi}(\ell^\infty)$ when $\delta = \Delta_K f^2/4$ and $\phi = f$. In particular, the subgroup $\langle J_{\delta, 0},c_\varepsilon\rangle\subseteq \mathcal{N}_{\delta,0}(\ell^\infty)$ of Theorem \ref{thm-oddprimedividingdisc} is conjugate to
\[\left\langle \left\{ \left(\begin{array}{cc} a^2+b\phi/2 & b\\ (\delta-\phi^2/4) b & a^2-b\phi/2 \\ \end{array}\right): a \in \Z_\ell^\times, b \in \Z_\ell \right\}, \left(\begin{array}{cc}  \varepsilon & 0\\ -\varepsilon\phi  & -\varepsilon\\ \end{array}\right) \right\rangle\]  
    for some $\varepsilon\in \{\pm 1\}$, as a subgroup of $\mathcal{N}_{\delta-\phi^2/4,\phi}(\ell^\infty)$.
\end{remark}

The next result describes the $2$-adic image in the case of $j_{K,f}\neq 0,1728$.

\begin{thm}[\cite{lozano-galoiscm}, Theorem 1.6]\label{thm-m8and16alvaro}
	Let $E/\Q(j_{K,f})$ be an elliptic curve with CM by $\OO_{K,f}$ with $j_{K,f}\neq 0, 1728$. Then, either $G_{E,2^{\infty}} = \mathcal{N}_{\delta, \phi}(2^{\infty})$, or $[\mathcal{N}_{\delta, \phi}(2^{\infty}) :  G_{E,2^{\infty}}] = 2$ and one of the following two possibilities hold:
	\begin{enumerate} 
		\item $\Delta_K f^2 \equiv 0 \bmod 16$, and in particular
        \begin{itemize}
         \item $\Delta_K \equiv 1 \bmod 4$ and $f \equiv 0 \bmod 4$ or 
         \item $\Delta_K \equiv 0 \bmod 4$ and $f \equiv 0 \bmod 2$. \end{itemize}
      In this case, $G_{E,2^{\infty}}$ is generated by $c_{\varepsilon}$ and one of the groups
			\[ \left\langle \left(\begin{array}{cc} 5 & 0\\ 0 & 5\\\end{array}\right),\left(\begin{array}{cc} 1 & 1\\ \delta & 1\\\end{array}\right)\right\rangle \text{ or }\left\langle \left(\begin{array}{cc} 5 & 0\\ 0 & 5\\\end{array}\right),\left(\begin{array}{cc} -1 & -1\\ -\delta & -1\\\end{array}\right)\right\rangle.\]
		\item $\Delta_K\equiv 0 \bmod 8$, or $\Delta_K\equiv 4 \bmod 8$ and $f\equiv 0 \bmod 4$, or $\Delta_K\equiv 1 \bmod 4$ and $f\equiv 0 \bmod 8$, and in this case $G_{E,2^{\infty}}$ is generated by $c_{\varepsilon}$ and one of the groups
			\[ \left\langle \left(\begin{array}{cc} 3 & 0\\ 0 & 3\\\end{array}\right),\left(\begin{array}{cc} 1 & 1\\ \delta & 1\\\end{array}\right)\right\rangle \text{ or } \left\langle \left(\begin{array}{cc} 3 & 0\\ 0 & 3\\\end{array}\right),\left(\begin{array}{cc} -1 & -1\\ -\delta & -1\\\end{array}\right)\right\rangle.\]
	\end{enumerate}
	%$$\left\langle \left(\begin{array}{cc} \varepsilon & 0\\ 0 & -\varepsilon\\\end{array}\right),\left(\begin{array}{cc} \alpha & 0\\ 0  & \alpha\\\end{array}\right),\left(\begin{array}{cc} 1 & 1\\ \delta  & 1\\\end{array}\right)\right\rangle \text{ or } \left\langle \left(\begin{array}{cc} \varepsilon & 0\\ 0 & -\varepsilon\\\end{array}\right),\left(\begin{array}{cc} \alpha & 0\\ 0  & \alpha\\\end{array}\right),\left(\begin{array}{cc} -1 & -1\\ -\delta  & -1\\\end{array}\right)\right\rangle \subseteq \GL(2,\Z_2).$$
\end{thm}

The last result describes the $2$-adic image when $j_{K,f}=1728$.

\begin{thm}[\cite{lozano-galoiscm}, Theorem 1.7]\label{thm-j1728alvaro}\label{thm-lastofsec1}
	Let $E/\Q$ be an elliptic curve with $j(E)=1728$, and let $\gamma \in \{ c_{1}, c_{-1}, c'_{1}, c'_{-1} \}$. Then $G_{E,2^\infty} = \mathcal{N}_{-1,0}(2^\infty)$, or $[\mathcal{N}_{-1,0}(2^\infty) : G_{E,2^\infty}] = 2$ or $4$ and $G_{E,2^\infty}$ is one of the following groups:
	\begin{itemize} 
		\item If $[\mathcal{N}_{-1,0}(2^\infty):G_{E,2^\infty}]=2$, then $G_{E,2^\infty}$ is generated by $\gamma$ and one of the groups
		\[ \left\langle -\operatorname{Id}, 3\cdot \operatorname{Id},\left(\begin{array}{cc} 1 & 2\\ -2 & 1\\\end{array}\right) \right\rangle, \text{ or } \left\langle -\operatorname{Id}, 3\cdot \operatorname{Id},\left(\begin{array}{cc} 2 & 1\\ -1 & 2\\\end{array}\right) \right\rangle. \]
		\item If $[\mathcal{N}_{-1,0}(2^\infty) : G_{E,2^\infty}]=4$, then $G_{E,2^\infty}$ is generated by $\gamma$ and one of the group
		\[ \left\langle 5\cdot \operatorname{Id},\left(\begin{array}{cc} 1 & 2\\ -2 & 1\\\end{array}\right) \right\rangle, \text{ or } \left\langle 5\cdot \operatorname{Id},\left(\begin{array}{cc} -1 & -2\\ 2 & -1\\\end{array}\right) \right\rangle, \text{ or } \]
		\[  \left\langle -3\cdot \operatorname{Id},\left(\begin{array}{cc} 2 & -1\\ 1 & 2\\\end{array}\right) \right\rangle, \text{ or } \left\langle -3\cdot \operatorname{Id},\left(\begin{array}{cc} -2 & 1\\ -1 & -2\\\end{array}\right) \right\rangle. \]
\end{itemize}
\end{thm}

\subsection{Field Theory}\label{sec-fieldtheory}

 In this section, we collect several results on division fields of elliptic curves, as well as cyclotomic extensions of $\Q$, in one theorem. Note: if $N\geq 2$ is an integer, then $N=n\cdot m^2$ for some integers $m,n$ with $n$ square-free, which we call the square-free part of $N$, and in this paper we denote by $n=N^\text{sf}$.

\begin{theorem}\label{cyc-division-thm}\label{kronecker-weber}\label{sqrt-general}\label{lem-clark}
    Let $E/F$ be an elliptic curve defined over a number field $F$. Let $M, N \geq 2$ be integers, where $N$ is non-zero and square-free. We define
    \[ N^{\dagger} : = \operatorname{disc}(\mathcal{O}_{\Q(\sqrt{N})}) =  
\begin{cases}
|N| & \text{ if } N \equiv 1 \bmod 4, \\
|4N| & \text{ if } N \equiv 2,3 \bmod 4, \\
\end{cases} \]
    Then:
    \begin{enumerate}
        \item  $F(\zeta_M) \subseteq F(E[M])$.
        \item $F(\sqrt{N}) \subseteq F(\zeta_{N^\dagger})\subseteq F(E[N^{\dagger}])$. 
        \item Let $K$ be an imaginary quadratic field, let $E/F$ be an elliptic curve with CM by $\OO_{K,f}$ and $j(E)=j_{K,f}$, and assume $M \geq 3$. Then,  $K\subseteq K(j_{K,f})\subseteq \Q(j_{K,f},E[M]) \subseteq F(E[M])$.
    \end{enumerate} 
\end{theorem}

\begin{proof}
    Part (1) follows from \cite[Ch. III, Corollary 8.1.1]{silverman1}. Part (2) follows from \cite[Ch. 2, Exercise 8]{marcus2018number} and part (1). For part (3), let $E/F$ be an elliptic curve with CM by $\mathcal{O}_{K,f}$ and $j(E)=j_{K,f}$. Then \cite[Lemma 3.15]{bourdon4} shows that $K\subseteq F(E[M])$ for $M \geq 3$ and, in fact, their proof shows something a bit stronger, which is that $K\subseteq \Q(j(E),E[M])\subseteq F(E[M])$. Therefore $K\subseteq K(j_{K,f})\subseteq \Q(j_{K,f},E[M]) \subseteq F(E[M])$, as desired.
\end{proof}

We note here that the definition of $N^\dagger$ is best possible. For instance, $\Q(\sqrt{-3})\subseteq \Q(\zeta_3)$ and $\Q(\sqrt{3})\subseteq \Q(\zeta_{12})$ but $\sqrt{3}\not\in \Q(\zeta_3)$.

% \begin{corollary}\label{sqrt-primes}
%     For $p$ an odd prime, we have $\Q(\sqrt{p^{\ast}}) \subseteq \Q(\zeta_p)$, where 
%     \[ p^{\ast} := \begin{cases}
%         p & \text{ if } p \equiv 1 \bmod 4, \\
%         -p & \text{ if } p \equiv 3 \bmod 4.
%     \end{cases}\]
% Further, $\Q(\sqrt{2},\sqrt{-2}) \subseteq F(E[8])$.
% \end{corollary}

\section{Level of Definition}\label{sec-level}

Let $E/\Q(j_{K,f})$ be an elliptic curve with CM by $\Of$ and let $\delta$ and $\phi$ be the associated constants defined in Section \ref{sec-notation}. For $M,N \geq 2$ where $M \mid N$, let
\[ \pi_{N,M} \colon \mathcal{N}_{\delta, \phi}(N) \to \mathcal{N}_{\delta, \phi}(M) \]
be the natural projection map given by reduction modulo $M$.
We also let
\[ \pi_{N} \colon  \mathcal{N}
_{\delta, \phi} \to \mathcal{N}_{\delta, \phi}(N) \]
be the natural projection map from the adelic normalizer of Cartan $\mathcal{N}
_{\delta, \phi}=\mathcal{N}
_{\delta, \phi}(\widehat{\Z})$, and for $\ell$ prime and $n \geq 1$, we let
\[ \pi_{\ell^{\infty}, \ell^n} \colon \mathcal{N}_{\delta, \phi}(\ell^{\infty}) \to \mathcal{N}_{\delta, \phi}(\ell^{n}) \]
be the natural projection map from the $\ell$-adic normalizer of Cartan. In terms of these projection maps, we make the following definitions about the level of an adelic or $\ell$-adic image.

\begin{definition}\label{def-level-of-definition}
    Let $E/\Q(j_{K, f})$ be an elliptic curve with CM by $\Of$, and fix a compatible adelic basis for $T(E)$ such that  $G_{E} := \im \rho_{E} \subseteq \mathcal{N}_{\delta, \phi}$. We say that the adelic image of $E$ is \emph{defined at level $N$} for some $N \geq 1$ or, equivalently, that $N$ is an \emph{adelic level of definition for $E$}, if \[ G_{E} = \pi_N^{-1}(\pi_N(G_{E})). \] We will further say that the \emph{minimal adelic level of definition for $E$} is the smallest level $N$ for which the adelic image is defined. Likewise, we define the \emph{minimal $\ell$-adic level of definition for $E$} to be the least power of $\ell$, denoted $\ell^n$, for which \[ G_{E, \ell^{\infty}} = \pi_{\ell^{\infty}, \ell^n}^{-1}(\pi_{\ell^{\infty}, \ell^n}(G_{E, \ell^{\infty}})). \]
\end{definition}

To begin studying the adelic image of CM elliptic curves defined over $\Q$, we first cite a result on the size of the adelic image.

\begin{theorem}[\cite{CampagnaPengo2}, Theorem 1.1]\label{thm-adelic-indexQ}
    Let $E/\Q(j_{K,f})$ be an elliptic curve with CM by $\Of$, and fix a choice of adelic basis so that $G_{E} := \im \rho_{E} \subseteq \mathcal{N}_{\delta, \phi}$. Then, the index of $\im \rho_E$ in $\mathcal{N}_{\delta, \phi}$ is given by the formula
    $$\frac{|\Of^{\times}|}{[\Q(j_{K,f},E_\text{tors}):K^\text{ab}]}.$$
   In particular, if $\Q(j_{K,f})=\Q$, then the index is exactly $|\Of^{\times}|$. More precisely, 
    \[ [\mathcal{N}_{\delta, \phi} : G_E] = \begin{cases}
2 & \text{ if } j(E) \neq 0, 1728, \\
4 & \text{ if } j(E) = 1728, \\
6 & \text{ if } j(E) = 0.
\end{cases}  \]
\end{theorem}

\begin{proof} Let $E/F$ be an elliptic curve with CM by $\Of$, and let $H_{\Of}$ be the ring class field of $K$ relative to the order $\Of$. Then, \cite[Theorem 1.1]{CampagnaPengo2}, shows that the index of $G_E$ in $\mathcal{N}_{\delta, \phi}$ is given by
$$\frac{[(FK)\cap K^\text{ab} : H_{\Of}]\cdot |\Of^{\times}|}{[F(E_\text{tors}):FK^\text{ab}]}.$$
In our case, $F=\Q(j_{K,f})$ and \cite[Theorem 3.2]{lozano-galoiscm} shows that $H_{\Of}=K(j_{K,f})$. Since $K(j_{K,f})/K$ is abelian, we have that $FK^{\text{ab}} = K^\text{ab}$ in this case. Therefore, $[(FK)\cap K^\text{ab} : H_{\Of}]=1$, and the index formula in our statement follows.

Finally, if $F=\Q(j_{K,f})=\Q$, then \cite[Example 5.8]{silverman2} shows that $K^\text{ab} = K(E_\text{tors})$. Further, Theorem \ref{lem-clark} shows that $K\subseteq \Q(E_{\text{tors}})$, and therefore $K^\text{ab} = \Q(E_{\text{tors}})$ and so $[F(E_\text{tors}):FK^\text{ab}]=1$ and the index of the image is precisely given by $|\Of^{\times}|$.  
\end{proof}

\begin{remark}\label{rem-level1}
    Let $E/\Q(j_{K,f})$ be an elliptic curve with CM by $\Of$. If $N=1$ is a level of definition, then $G_E=\mathcal{N}_{\delta, \phi}$, i.e., the adelic index is $1$. While Theorem \ref{thm-adelic-indexQ} says that the adelic index is never one for $E/\Q$, the index can be one when $\Q(j_{K,f})\neq \Q$ (see Section 5 of \cite{CampagnaPengo2} for examples). Indeed, for $\pi_N$ as in Def. \ref{def-level-of-definition} with $N=1$ we have that $\pi_1(G_E)$ is the trivial group, and therefore $\pi_1^{-1}(\pi_1(G_E))=\mathcal{N}_{\delta, \phi}$. However, the adelic index is at least $2$ when $E$ is defined over $\Q$, so the minimal level of definition must be at least $N \geq 2$ in that case.
\end{remark}

The rest of this section is dedicated to justify Definition \ref{def-level-of-definition}, and prove results about levels of definition of an adelic image. First, we show that if an adelic image is defined modulo $N$, then it is also defined at level $nN$ for any $n\geq 1$.

\begin{lemma}\label{lem-levelup}
	Let $E/\Q(j_{K, f})$ be an elliptic curve with CM by $\Of$, suppose that the adelic image of $E$ is defined modulo $N\geq 1$, and let $n\geq 1$. Then, the adelic image is also defined modulo $nN$.
	\end{lemma}
	\begin{proof}
		Suppose that $N$ is a level of definition of the adelic image of $E$. Then, by definition, we have 
		\[ G_{E} = \pi_N^{-1}(\pi_N(G_{E})). \]
		Let $n\geq 1$ and let $N'=nN$. Since $N$ is a level of definition, it follows that 
		$$\pi_{N'}(G_E) = \pi_{N',N}^{-1}(\pi_N(G_E)).$$
		Thus,
		$$\pi_{N'}^{-1}(\pi_{N'}(G_E))= \pi_{N'}^{-1}(\pi_{N',N}^{-1}(\pi_N(G_E)) = \pi_{N}^{-1}(\pi_N(G_E))=G_E,$$
		as desired.
	\end{proof}

\begin{remark}
Under Definition \ref{def-level-of-definition}, the adelic image of a CM elliptic curve $E$ is defined at any level $N$ for which the index of the mod-$N$ image is equal to the index of the adelic image. We explain why this is the case through the following results.
\end{remark}

\begin{proposition}\label{prop-gen-surj-map}
    Let $G$ and $H$ be groups and let $\varphi \colon G \to H$ be a surjective group homomorphism. Let $H' \subseteq H$ be a subgroup of $H$, and define $G'= \varphi^{-1}(H')$ to be the preimage of $H'$.
    Then 
    \[ [G : G'] \geq [H : H'].\]
\end{proposition}

\begin{proof}
    It is sufficient to show that for every distinct coset representative of $H'$ in $H$, there is a distinct coset representative of $G'$ in $G$. Let $h_1, h_2 \in H$ be distinct coset representatives of $H'$, and let $g_1, g_2 \in G$ be such that $\varphi(g_i) = h_i$ for $i = 1, 2$. Then, since $\varphi(G') = H'$ by definition and $h_1 H' \cap h_2 H' = \emptyset$ by hypothesis, we have 
    \begin{align*}
        \varphi(g_1 G' \cap g_2 G') 
        &\subseteq \varphi(g_1 G') \cap \varphi(g_2 G') = h_1 H' \cap h_2 H' = \emptyset,
    \end{align*}    
    so $g_1 G' \cap g_2 G' = \emptyset$ as desired.
\end{proof}

\begin{theorem}\label{thm-levelofdef}
    Let $E/\Q(j_{K,f})$ be an elliptic curve with CM by $\Of$, and suppose a compatible choice of bases is made so that $G_{E} = \im \rho_{E} \subseteq \mathcal{N}_{\delta, \phi}$, that $G_{E, N} = \im \rho_{E, N} \subseteq \mathcal{N}_{\delta, \phi}(N)$, and that $G_{E,N} = \pi_{N}(G_{E})$ for $N \geq 2$. Moreover, assume that $[\mathcal{N}_{\delta, \phi}(N) : G_{E, N}] = [\mathcal{N}_{\delta, \phi} : G_{E}]$. Then, the adelic image of $E$ is defined at level $N$.
\end{theorem}

\begin{proof}
   Since $\pi_N(G_E)=G_{E,N}$, it follows that $G_{E}$ is a subgroup of $\pi_{N}^{-1}(G_{E, N})$. By Proposition \ref{prop-gen-surj-map} and our hypotheses, it follows that 
    \[ [\mathcal{N}_{\delta, \phi}(N) : G_{E, N}] \leq [\mathcal{N}_{\delta, \phi} : \pi_{N}^{-1}(G_{E, N})] \leq [\mathcal{N}_{\delta, \phi} : G_{E}] = [\mathcal{N}_{\delta, \phi}(N) : G_{E, N}],  \]
    so $[\mathcal{N}_{\delta, \phi} : \pi_{N}^{-1}(G_{E, N})] = [\mathcal{N}_{\delta, \phi} : G_{E}]$, and therefore $[\pi_{N}^{-1}(G_{E, N}) : G_{E}] = 1$.
\end{proof}

Note that a similar proof shows that if $[\mathcal{N}_{\delta, \phi}(\ell^n) : G_{E, \ell^n}] = [\mathcal{N}_{\delta, \phi}(\ell^{\infty}) : G_{E, \ell^{\infty}}]$, then the $\ell$-adic image of $E$ is defined at level $\ell^n$.

\begin{remark}
    It would be reasonable at this stage to ask whether the minimal adelic level of definition of a CM elliptic curve, as stated, does not miss some arithmetic information about the corresponding elliptic curve. Namely, it is not clear \textit{a priori} why the adelic image of a CM elliptic curve could not be defined at two positive integers $N$ and $M$ where $N$ and $M$ are relatively prime, in which case there could be said to be two minimal levels of definition capturing two different kinds of arithmetic information. In the following propositions, we show that this is not an issue.
\end{remark}

\begin{figure}[ht!]\label{fig-prop3.4}
% https://tikzcd.yichuanshen.de/#N4Igdg9gJgpgziAXAbVABwnAlgFyxMJZABgBoBGAXVJADcBDAGwFcYkQAdDgW3pwAsAxk2ABhAL4B9YF1iMc9UlzT8s4gBQA5AJQhxpdJlz5CKAEwVqdJq3ZdeA4YzFSZHOQqUcVa9QFldfUNsPAIiclJiKwYWNkROHj4hEQlpWRh5RWVVDQARQIMQDBCTcNIzaJs4hIdk51S3DyzvHPUoQKsYKABzeCJQADMAJwhuJDIQHAgkAGYaGNt45SxpTVIocRAaRnoAIwyABSNQ0xAhrG7+HD1C4dGkC0npxAmF6uXpXNJNTe29w+OpXi50u1yCIDuY0QjymSAiIB2+0YRxKYXijBgA2u8yqdm8K2AXz8m3BkLhNFhiDm1lieLQBL86xJlHEQA
\begin{tikzcd}
                                                          & {\mathcal{C}_{\delta,\phi}(D)} \arrow[ld, "{\pi_{D,N}}"'] \arrow[rd, "{\pi_{D,M}}"] &                                                          \\
{\mathcal{C}_{\delta,\phi}(N)} \arrow[rd, "{\pi_{N,d}}"'] &                                                                                     & {\mathcal{C}_{\delta,\phi}(M)} \arrow[ld, "{\pi_{M,d}}"] \\
                                                          & {\mathcal{C}_{\delta,\phi}(d)}                                                      &                                                         
\end{tikzcd}
\label{blah}
\caption{Diagram of groups and maps that appear in Prop. \ref{prop-normalizerCRT}}
\end{figure}
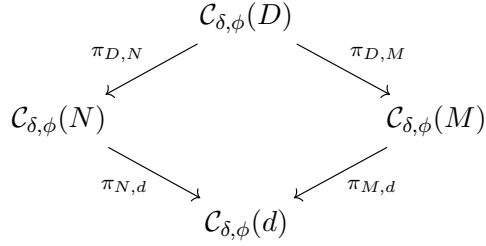 

\begin{proposition}\label{prop-normalizerCRT}
    Let $\Delta(\Of)=\Delta_K f^2$ be the discriminant of $\Of$ and define the constants $\delta$ and $\phi$ as in Section \ref{sec-notation}. For integers $M, N \geq 2$, let $D = \lcm(M,N)$ and $d = \gcd(M,N)$. Choose $g_{M} \in \mathcal{C}_{\delta, \phi}(M)$ and $g_{N} \in \mathcal{C}_{\delta, \phi}(N)$ so that there is some $g_d \in \mathcal{C}_{\delta, \phi}(d)$ such that $\pi_{M,d}(g_M) = g_d = \pi_{N,d}(g_N)$.    Then, there exists a unique $g_D \in \mathcal{C}_{\delta, \phi}(D)$ such that $\pi_{D,M}(g_D) = g_M$ and $\pi_{D,N}(g_D) = g_N$.
\end{proposition}
\begin{proof} First, we note that if $u,v$ are integers such that $u\equiv v \bmod d$, then the system of congruences 
	\begin{align}\label{eq-1} \begin{cases}
		x\equiv u \bmod M,\\
		x\equiv v \bmod N,
	\end{cases}\end{align} 
	has a unique solution modulo $x_0=x_0(u,v)\bmod D$. Indeed, if we write $M=M'd$ and $N=N'd$, then the system (\ref{eq-1}) is equivalent to 
		\begin{align}\label{eq-2} \begin{cases}
			x\equiv u \bmod M',\\
			x\equiv u \bmod d,\\
			x\equiv v \bmod N',\\
			x\equiv v \bmod d
	\end{cases}\end{align}
which, by the Chinese remainder theorem, has a unique solution modulo $M'N'd = D$ if and only if $u\equiv v \bmod d$. 

In particular, the uniqueness of solutions of the system (\ref{eq-1}) modulo $D$ shows that for any $g_N \in \GL(2, \Z/N\Z)$ and $g_M\in \GL(2, \Z/M\Z)$, there exists a unique $g_D \in \GL(2,\Z/D\Z)$ which reduces to $g_N$ and $g_M$ modulo $N$ and $M$, respectively, if and only if $g_N\equiv g_M \bmod d$.

Now, if  $g_{M} \in \mathcal{C}_{\delta, \phi}(M)$ and $g_{N} \in \mathcal{C}_{\delta, \phi}(N)$ so that there is some $g_d \in \mathcal{C}_{\delta, \phi}(d)$ such that $\pi_{M,d}(g_M) = g_d = \pi_{N,d}(g_N)$, as in the statement of the lemma, then our previous remarks show that there is a unique $g_D\in \GL(2,\Z/D\Z)$ such that $g_D\equiv g_M\bmod M$ and $g_D\equiv g_N\bmod N$. Moreover, since $g_M$ is an element of the Cartan subgroup, it must be of the form $c_{\delta,\phi}(a_M,b_M)$ for some integers $a_M,b_M$ (as in the notation established in Section \ref{sec-notation}), and similarly $g_N= c_{\delta,\phi}(a_N,b_N)$ for some $a_N,b_N$. Let $x_0(a_M,a_N)$ and $x_0(b_M,b_N)$ be the unique solutions of the corresponding systems as in Eq. (\ref{eq-1}).  Then, the element
$$g=c_{\delta,\phi}(x_0(a_M,a_N),x_0(b_M,b_N))\in  \mathcal{C}_{\delta, \phi}(D)\subseteq \GL(2,\Z/D\Z)$$
satisfies $g\equiv g_N \bmod N$ and $g\equiv g_M\bmod M$, and therefore, by the uniqueness of $g_D$, we must have $g=g_D$ belongs to $\mathcal{C}_{\delta, \phi}(D)$, as desired. 
 \end{proof}

An important fact for us will be that the index of the image inside the Cartan is the same as the index of the image in the normalizer of Cartan, a fact that is stated precisely in the following proposition.

\begin{prop}\label{prop-sameindex}
	Let $N\geq 3$, let $E/\Q(j_{K,f})$ be an elliptic curve with CM by $\Of$, and let $G_{E,N}\subseteq \mathcal{N}_{\delta, \phi}(N)$ be the image of $\rho_{E,N}$. Then,
	$$[\mathcal{N}_{\delta, \phi}(N):G_{E,N}] = [\mathcal{C}_{\delta, \phi}(N):G_{E,N}\cap \mathcal{C}_{\delta, \phi}(N)]=[\mathcal{C}_{\delta, \phi}(N):G_{E/K,N}].$$
\end{prop}
\begin{proof}
	This is precisely the content of part (3) of \cite[Theorem 6.3]{lozano-galoiscm}, where we remind the reader of the notation from Section \ref{sec-galoisrepnotation} where we defined $G_{E/K,N}$ as the image of $\rho_{E,N}$ when restricted to $\Gal(\overline{\Q(j_{K,f})}/K(j_{K,f}))$ and, by Theorem \ref{thm-firstofsec1}, we have $G_{E,N}\cap \mathcal{C}_{\delta,\phi}(N)=G_{E/K,N}$.
\end{proof}

We record the following corollary for later use.

\begin{corollary}\label{cor-galoistheoryofcompositum}
    Let $E/\Q(j_{K,f})$ be an elliptic curve with CM by $\Of$, let $F=\Q(j_{K,f})$, and let $N>M\geq 2$ be relatively prime integers. Then,
    \begin{enumerate}
        \item $\mathcal{C}_{\delta, \phi}(MN)\cong \mathcal{C}_{\delta, \phi}(M)\times \mathcal{C}_{\delta, \phi}(N)$ via the natural isomorphism induced by the Chinese remainder theorem,
        \item $G_{E/K,MN}$ is isomorphic to a subgroup  of $ G_{E/K,M}\times G_{E/K,N} \subseteq \mathcal{C}_{\delta, \phi}(M)\times \mathcal{C}_{\delta, \phi}(N)$, 
        \item the index $[\mathcal{C}_{\delta, \phi}(MN): G_{E/K,MN}]$ is divisible by the degree of $F(E[M])\cap F(E[N])/FK$, and
        \item if $j(E)\neq 0,1728$ and $[F(E[M])\cap F(E[N]):FK]=2$, then $[\mathcal{N}_{\delta, \phi}(M):G_{E,M}]=[\mathcal{N}_{\delta, \phi}(N):G_{E,N}]=1$.
    \end{enumerate}
\end{corollary}
\begin{proof}
    Part (1) follows from Prop. \ref{prop-normalizerCRT} with $d=\gcd(M,N)=1$ and $D=\lcm(M,N)=MN$. For an integer $S>2$, let $G_{E/K,S}$ be the image of $\rho_{E/K,S}$. Then, $K(j_{K,f})\subseteq F(E[S])$ by Theorem \ref{cyc-division-thm}, and $\Gal(F(E[S])/FK) \cong G_{E/K,S}$ from the definition and properties of the Galois representation $\rho_{E/K,S}$. Since $F(E[MN])$ is the compositum of the Galois extensions $F(E[M])$ and $F(E[N])$ of $FK=K(j_{K,f})$, parts (2) and (3) now follow from Galois theory, because the index $[G_{E/K,M}\times G_{E/K,N}:G_{E/K,MN}]$ is equal to the degree of the extension $F(E[M])\cap F(E[N])/FK$, and
    $$[\mathcal{C}_{\delta, \phi}(MN): G_{E/K,MN}] = [\mathcal{C}_{\delta, \phi}(MN):G_{E/K,M}\times G_{E/K,N}]\cdot [G_{E/K,M}\times G_{E/K,N}:G_{E/K,MN}],$$
    as claimed.

    Now suppose that $[F(E[M])\cap F(E[N]):FK]=2$. Then, $[G_{E/K,M}\times G_{E/K,N}:G_{E/K,MN}]=2$. Moreover, if either $[\mathcal{N}_{\delta, \phi}(M):G_{E,M}]$ or $[\mathcal{N}_{\delta, \phi}(N):G_{E,N}]$ is bigger than $1$, then the above displayed equation of indices shows that $[\mathcal{C}_{\delta, \phi}(MN): G_{E/K,MN}]>2$. Hence, the adelic index of $E$ is greater than $2$, which is impossible since $j(E)\neq 0,1728$ and so the index must be $1$ or $2$. Thus, $[\mathcal{N}_{\delta, \phi}(M):G_{E,M}]=[\mathcal{N}_{\delta, \phi}(N):G_{E,N}]=1$ as desired. This shows (4).
\end{proof}

We are now ready to justify our way of defining the minimal level of definition with the following three results.

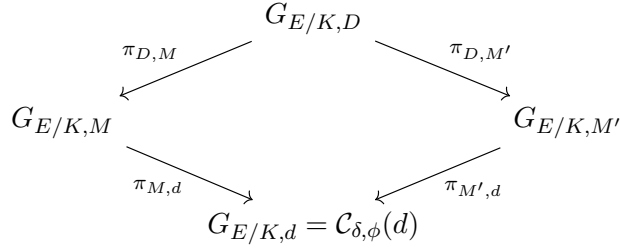
\begin{figure}[ht!]
    \centering 
\begin{tikzcd}
&  {G_{E/K, D}} \arrow[ld, "{\pi_{D,M}}"'] \arrow[rd, "{\pi_{D,M'}}"] & \\
{G_{E/K, M}} \arrow[rd, "{\pi_{M,d}}"'] & &  {G_{E/K, M'}} \arrow[ld, "{\pi_{M',d}}"] \\
& {G_{E/K, d} = \mathcal{C}_{\delta, \phi}(d)} & 
\end{tikzcd}\caption{A diagram for the groups and maps for the proof of Prop. \ref{prop-levelofdefngcd}.}
\end{figure}

\begin{prop}\label{prop-levelofdefngcd}
    Let $E/\Q(j_{K,f})$ be an elliptic curve with CM by $\Of$. We make a compatible choice of bases so that $G_{E} = \im \rho_{E} \subseteq \mathcal{N}_{\delta, \phi}$, that $G_{E,N}=\im \rho_{E,N} \subseteq \mathcal{N}_{\delta, \phi}(N)$, and that $G_{E, N} = \pi_{N}(G_E)$ for all $N \geq 2$. Suppose that for integers $M, M' \geq 2$, we have $G_{E} = \pi_{M}^{-1}(G_{E, M})$ and $G_{E} = \pi_{M'}^{-1}(G_{E, M'})$. Then, the adelic image of $E$ is defined at level $d = \gcd(M, M')$ if $d>2$, or at level $4$ if $d=2$.
\end{prop}

\begin{proof} Since $M$ and $M'$ are both levels of definition for $E$, it follows by Lemma \ref{lem-levelup} that $D := \lcm(M, M')$ is also a level of definition. Also by Lemma \ref{lem-levelup}, if $\gcd(M,M')=2$, then we may replace $M$ by $2M$ and $M'$ by $2M'$, in which case $\gcd(M,M')=4$. Thus, we may assume that $d \geq 3$.

    Choose an element $g_{M'} \in \pi_{M', d}^{-1}(G_{E/K, d}) \subseteq \mathcal{C}_{\delta, \phi}(M')$ of the pullback from $G_{E/K, d}$ to $G_{E/K, M'}$, and let $g_d := \pi_{M',d}(g_{M'}) \in G_{E/K,d}$. Since we have that $\pi_{M,d}(G_{E/K,M}) = G_{E/K,d}$, there exists some $g_M \in G_{E/K,M}$ such that $\pi_{M,d}(g_{M}) = g_d$.

   Since $\pi_{M,d}(g_{M}) = g_d = \pi_{M',d}(g_{M'})$, it follows by Proposition \ref{prop-normalizerCRT} that there exists some $g_D \in \mathcal{C}_{\delta, \phi}(D)$ such that $\pi_{D,M}(g_D) = g_M$ and $\pi_{D,M'}(g_D) = g_{M'}$. However, since $M$ is a level of definition, we must have $G_{E/K,D} = \pi_{D,M}^{-1}(G_{E/K,M})$, so in fact $g_D \in G_{E/K,D}$. Thus, since $\pi_{D,M'}(G_{E/K,D}) = G_{E/K,M'}$, we have $g_{M'} \in G_{E/K,M'}$. Since $g_{M'} \in \pi_{M', d}^{-1}(G_{E/K, d})$ was an arbitrary element, if follows that $G_{E/K,M'} = \pi_{M', d}^{-1}(G_{E/K, d})$. Therefore, since $M'$ is a level of definition, it follows that $G_{E/K} = \pi_{M'}^{-1}(G_{E/K, M'}) = \pi_{M'}^{-1}(\pi_{M', d}^{-1}(G_{E/K, d})) = \pi_{d}^{-1}(G_{E/K, d})$. This shows that $d$ is a level of definition and it completes the proof.
\end{proof}

\begin{prop}\label{prop-gcdoneimpliesindex1}
    Let $E/\Q(j_{K,f})$ be an elliptic curve with CM by $\Of$, and assume $j_{K,f} \neq 0, 1728$. We make a compatible choice of bases as in Prop. \ref{prop-levelofdefngcd}. Suppose that for integers $M, M' \geq 2$, we have $G_{E} = \pi_{M}^{-1}(G_{E, M})$ and $G_{E} = \pi_{M'}^{-1}(G_{E, M'})$, and $\gcd(M,M')=1$. Then, the adelic index is $1$, i.e., $[\mathcal{N}_{\delta, \phi} : G_{E}]= 1$ or equivalently $G_E=\mathcal{N}_{\delta, \phi}$.
\end{prop}
\begin{proof}
    Suppose that the adelic image $G_E$ is defined at relatively prime levels $M$ and $M'\geq 2$, and suppose for a contradiction that $[\mathcal{N}_{\delta, \phi} : G_{E}]= 2$. Thus, $[\mathcal{N}_{\delta, \phi}(M) : G_{E, M}] = [\mathcal{N}_{\delta, \phi}(M') : G_{E, M'}] = 2$ and further that $[\mathcal{N}_{\delta, \phi}(D) : G_{E, D}] = 2$ for $D := \lcm(M, M')=MM'$, since the image is defined at level $M$ and $M\mid D$. By Prop. \ref{prop-sameindex}, we also have
    $$[\mathcal{C}_{\delta, \phi} : G_{E/K}] = [\mathcal{C}_{\delta, \phi}(D) : G_{E/K, D}]= [\mathcal{C}_{\delta, \phi}(M) : G_{E/K, M}] = [\mathcal{C}_{\delta, \phi}(M') : G_{E/K, M'}]= 2,$$
    However, since $\gcd(M,M')=1$, it follows that 
    $\mathcal{C}_{\delta, \phi}(MN)\cong \mathcal{C}_{\delta, \phi}(M)\times \mathcal{C}_{\delta, \phi}(N)$ by Cor. \ref{cor-galoistheoryofcompositum}. Moreover, $G_{E/K,MN}$ is isomorphic to a subgroup  of $ G_{E/K,M}\times G_{E/K,N} \subseteq \mathcal{C}_{\delta, \phi}(M)\times \mathcal{C}_{\delta, \phi}(N)$, and therefore
    $$2=[\mathcal{C}_{\delta, \phi}(D) : G_{E/K, D}]\geq [\mathcal{C}_{\delta, \phi}(M) : G_{E/K, M}] \cdot  [\mathcal{C}_{\delta, \phi}(M') : G_{E/K, M'}]\geq 2\cdot 2 = 4,$$
    which is the contradiction we were looking for. Hence $[\mathcal{N}_{\delta, \phi} : G_{E}]= 1$, as claimed.
\end{proof}

\begin{prop}\label{prop-minimalleveldivideslevelsofdef}
    Let $E/\Q(j_{K,f})$ be an elliptic curve with CM by $\Of$, and assume $j_{K,f} \neq 0, 1728$. We make a compatible choice of bases as in Prop. \ref{prop-levelofdefngcd}. Suppose that $M, M' \geq 1$ are levels of definition, i.e., we have $G_{E} = \pi_{M}^{-1}(G_{E, M})$ and $G_{E} = \pi_{M'}^{-1}(G_{E, M'})$. Moreover, assume that $M=M_E$ is the minimal level of definition for the adelic image of $E$. Then, $M_E$ is a divisor of $M'$.
\end{prop}
\begin{proof}
    If $M_E=1$ (and therefore, by Remark \ref{rem-level1}, the adelic index of the image is $1$), then there is nothing to prove, so let us assume that $2\leq M_E\leq M'$. If $\gcd(M,M')=1$, then Prop. \ref{prop-gcdoneimpliesindex1} implies that the adelic index is $1$ and so $M_E=1$, a contradiction. Thus, assume $d=\gcd(M,M')\geq 2$. Prop. \ref{prop-levelofdefngcd} now shows that $G_E$ is also defined at level $d$ (or level $4$ if $d=2$). If $d>2$ and since $d$ is a gcd of $M$ and $M'$, it follows that $d\leq M_E$, and since $M_E$ is minimal, it must be that $d=M_E$. Hence, $M_E$ divides $M'$ as claimed.

    If $d=2$ and $M_E=2$, then $M_E$ is a divisor of $M'$. If $d=2$ and $M_E=4$, suppose for a contradiction that $M'\equiv 2 \bmod 4$. Then, $M'=2\cdot M''$ with $M''\geq 3$ an odd integer. Now consider $D=4M''$. Since $M_E=4$, we must have $[\mathcal{C}_{\delta, \phi}(4) : G_{E/K, 4}] = 2$, and since $2M''$ is a level of definition, we must have that $G_{E/K,2M''}$ is a subgroup of index $2$ of $G_{E/K,2}\times G_{E/K,M''}$ by Cor. \ref{cor-galoistheoryofcompositum}. 
         \begin{align*} 2 &=[\mathcal{C}_{\delta, \phi}(4M''): G_{E/K,4M''}] \\
         &= [\mathcal{C}_{\delta, \phi}(4M''):G_{E/K,4}\times G_{E/K,M''}]\cdot [G_{E/K,4}\times G_{E/K,M''}:G_{E/K,4M''}],\\
          &\geq [\mathcal{C}_{\delta, \phi}(4):G_{E/K,4}] \cdot [\mathcal{C}_{\delta, \phi}(M''):G_{E/K,M''}]\cdot [G_{E/K,2}\times G_{E/K,M''}:G_{E/K,2M''}] \\
          &\geq 2\cdot 1 \cdot 2 = 4,
         \end{align*}
         which leads to a contradiction. Hence, we must have $M'\equiv 0 \bmod 4$, and $M_E=4$ is a divisor of $M'$, as desired.
\end{proof}

\begin{remark}

The convention used in this paper is a departure from the way in which other authors have considered levels of definition. For example, in \cite[Proposition 12.1.4]{elladic}, the level of definition $\ell^n$ for the $\ell$-adic image of a CM elliptic curve $E$ agrees with ours, except the assumption is made that if $E'$ is another CM elliptic curve, where $G_{E', \ell^{\infty}}$ is not conjugate to $G_{E, \ell^{\infty}}$ in $\GL(2, \Z_\ell)$, then $G_{E', \ell^{n}}$ is not conjugate to $G_{E, \ell^{n}}$ in $\GL(2, \Z/\ell^n \Z)$.

This motivates the following definition, which is in line with what other authors would consider the ``level of definition'' of an image.
\end{remark}

\begin{definition}\label{def-level-of-dif}
    Let $E/\Q(j_{K,f})$ be an elliptic curve with CM by $\Of$. We make a compatible choice of bases so that $G_{E} = \im \rho_{E} \subseteq \mathcal{N}_{\delta, \phi}$, that $G_{E,N} = \im \rho_{E,N} \subseteq \mathcal{N}_{\delta, \phi}(N)$, and that $G_{E, N} = \pi_{N}(G_E)$ for all $N \geq 2$. We say that $N$ is an \emph{adelic level of differentiation} for $E$ if 
    \begin{enumerate}
        \item $N$ is an adelic level of definition for $E$, so $G_E = \pi_{N}^{-1}(\pi_N(G_E))$, and
        \item If $E'/\Q(j_{K,f})$ is an elliptic curve with CM by $\mathcal{O}_{K, f}$ and $G_{E,N}$ is conjugate to $G_{E',N}$ in $\GL(2,\Z/N\Z)$, then $G_{E}$ is conjugate to $G_{E'}$ in $\GL(2, \widehat{\Z})$.
 \end{enumerate}
Similarly, we say that $\ell^n$ is an \emph{$\ell$-adic level of differentiation} for $E$ if the $\ell$-adic image of $E$ is defined at level $\ell^n$, and $G_{E, \ell^n}$ being conjugate to $G_{E', \ell^n}$ in $\GL(2, \Z/\ell^n \Z)$ implies  $G_{E,\ell^\infty}$ is conjugate to $G_{E', \ell^{\infty}}$ in $\GL(2, \Z_{\ell})$.
\end{definition}

We illustrate the distinction between a level of definition and a level of differentiation with the following example.

\begin{example}
    Let $E_1/\Q : y^2=x^3-x^2-3x-1$ and $E_2/\Q : y^2=x^3+x^2-13x-21$ be elliptic curves with LMFDB labels \eclabel{256.d2} and \eclabel{256.a1}, respectively. Both $E_1$ and $E_2$ have CM by the order $\Z[\sqrt{-2}]$ of discriminant $\Delta_K f^2 = -8$ in $K = \Q(\sqrt{-2})$. The $2$-adic images of these curves lie inside $\mathcal{N}_{-2,0}(2^{\infty})$, and we find their images to be
    \begin{align*}  G_{E_1, 2^{\infty}} &= \left\langle 
    \left( \begin{array}{cc} 3 & 0 \\ 0 & 3 \end{array} \right),
    \left( \begin{array}{cc} 1 & 1 \\ -2 & 1 \end{array} \right),
    \left( \begin{array}{cc} -1 & 0 \\ 0 & 1 \end{array} \right) \right\rangle \text{ and }\\
    G_{E_2, 2^{\infty}} &= \left\langle 
    \left( \begin{array}{cc} 3 & 0 \\ 0 & 3 \end{array} \right),
    \left( \begin{array}{cc} -1 & -1 \\ 2 & -1 \end{array} \right),
    \left( \begin{array}{cc} -1 & 0 \\ 0 & 1 \end{array} \right) \right\rangle.
    \end{align*}

    For $i = 1, 2$, we find that $[\mathcal{N}_{-2, 0}(2^{\infty}) : G_{E_i, 2^{\infty}}] = 2$, and reducing modulo $8$, we also have $[\mathcal{N}_{-2, 0}(8) : G_{E_i, 8}] = 2$. Thus, by Theorem \ref{thm-levelofdef}, the $2$-adic images of $E_1$ and $E_2$ are defined modulo $8$.

    However, $G_{E_1, 8}$ and $G_{E_2, 8}$ are conjugate in $\GL(2,\Z/8\Z)$, with conjugacy matrix $\left( \begin{array}{cc} 1 & 0 \\ 4 & 1 \end{array} \right)$, while $G_{E_1, 16}$ and $G_{E_2, 16}$ are not conjugate in $\GL(2, \Z/16\Z)$. Thus, the $2$-adic level of differentiation for $E_1$ and $E_2$ is $16$.
\end{example}

\begin{remark}\label{ell-level-of-diff}\label{rem-ell-level-of-diff}
    Let $E/\Q$ be an elliptic curve with CM. From Theorem \ref{thm-goodredn} and \cite[Proposition 12.1.4]{elladic}, it follows that $\ell^{n_{E, \ell}}$ is an $\ell$-adic level of differentiation for $E$, where
\[ n_{E, \ell} := \begin{cases}
1 & \text{ if } \ell > 3 \text{ or } \ell > 2 \text{ and } j(E) \neq 0, \\
3 & \text{ if } \ell = 3 \text{ and } j(E) = 0, \\
4 & \text{ if } \ell = 2.
\end{cases} \]
We will use the notation $n_{E,\ell}$ for the exponent of an $\ell$-adic level of differentiation of a CM elliptic curve over $\Q$ in later sections.
\end{remark}

% \begin{definition}
%     Let $E/\Q$ be an elliptic curve without CM, and let $\rho_{E} : G_{\Q} \to \GL(2, \widehat{\Z})$ be its adelic Galois representation, where we set $G_{E} := \im \rho_{E}$. Further, for an integer $N \geq 2$, define
%     \[ \widehat{\pi}_{N} : \GL(2, \widehat{\Z}) \to \GL(2, \Z/N\Z) \]
%     to be the natural projection map. Then we say that $N$ is the \emph{adelic level of definition for $E$} if 
%     \[ G_{E} = \widehat{\pi}_{N}^{-1}(G_{E, N}).
%     \]
% \end{definition}

% \textcolor{red}{Ben: I need something about indices here.}

% \begin{theorem}
%     Let $E_1/\Q$ and $E_2/\Q$ be elliptic curves without CM. Let $G_{E_i}$ be the adelic Galois representation for $E_i$, $i = 1, 2$, and suppose that $G_{E_i}$ is defined at level $N$. Then $G_{E_1}$ is conjugate to $G_{E_2}$ in $GL(2, \widehat{\Z})$ if and only if $G_{E_1, N}$ is conjugate to $G_{E_2, N}$ in $GL(2, \Z/N\Z)$.
% \end{theorem}

% \begin{proof}
%     The forward implication is immediate, so we prove the reverse implication.

%     Suppose there exists a matrix $A_N \in GL(2, \Z/N\Z)$ such that $G_{E_1, N} = A_N G_{E_2, N} A_N^{-1}$. Next, let $A \in \GL(2, \widehat{\Z})$ be a matrix such that $\widehat{\pi}_{N}(A) = A_N$. 

%     Then, since 
% \end{proof}

\section{Defining Simplest CM Curves}\label{sec-simplestCMcurves}

Let $E/\Q(j_{K,f})$ be an elliptic curve with CM by $\Of$, where $K$ is an imaginary quadratic field and $f \geq 1$ is the conductor of $\Of$. We let $G_{E}$ denote the adelic image of $E$, let $G_{E, \ell^{\infty}}$ denote the $\ell$-adic image of $E$ for prime $\ell$, and let $G_{E, N}$ denote the image of $E$ modulo $N$ for $N \geq 2$.

By Theorem \ref{thm-cmrep-intro-alvaro} (1), if we define
\[ d_E
:=
\begin{cases}
2 & \text{ if } j(E) \neq 0, 1728, \\
4 & \text{ if } j(E) = 1728, \\
6 & \text{ if } j(E) = 0,
\end{cases} \]
then $[\mathcal{N}_{\delta, \phi} : G_{E}] \mid d_{E}$. For this reason, we call $d_E$ the \emph{maximal adelic index of $E$}. Note that we might also denote the maximal adelic index as $d_{K,f}$ or $d_{j}$, since the index is dependent only on the CM class (equivalently the $j$-invariant) of the elliptic curve, not on the curve itself. We also remind the reader that all CM elliptic curves defined over $\Q$ achieve the maximal adelic index allowed by Theorem \ref{thm-adelic-indexQ}. 

We now prove a result for elliptic curves whose adelic image can immediately be determined by their $\ell$-adic image.

\begin{theorem}\label{thm-ell-simplest}
    Let $E/\Q(j_{K, f})$ be an elliptic curve with CM by $\Of$, and suppose that for some prime $\ell$, we have that $[\mathcal{N}_{\delta, \phi}(\ell^{\infty}) : G_{E, \ell^{\infty}}] = d_E$. Further, suppose that the $\ell$-adic image of $E$ is defined at level $\ell^n$ for $n \geq 1$. Then, the adelic image of $E$ is defined at level $\ell^n$. In particular, $G_{E} = \pi_{\ell^n}^{-1}(\pi_{\ell^n}(G_{E}))$.
\end{theorem}

\begin{proof}
    By assumption, $[\mathcal{N}_{\delta, \phi}(\ell^{\infty}) : G_{E, \ell^{\infty}}] = d_E$, and since $[\mathcal{N}_{\delta, \phi}(\ell^{\infty}) : G_{E, \ell^{\infty}}] \mid [\mathcal{N}_{\delta, \phi} : G_{E}]$ and $[\mathcal{N}_{\delta, \phi} : G_{E}] \mid  d_E$, we have $[\mathcal{N}_{\delta, \phi} : G_{E}] = d_E$.

    Since the $\ell$-adic image of $E$ is defined at level $\ell^n$, we have $[\mathcal{N}_{\delta, \phi}(\ell^n) : G_{E, \ell^n}] = d_E$, and so $[\mathcal{N}_{\delta, \phi}(\ell^n) : G_{E, \ell^n}] = [\mathcal{N}_{\delta, \phi} : G_{E}]$. By Theorem \ref{thm-levelofdef}, the adelic image is defined modulo $\ell^n$, as desired, and we have $G_{E} = \pi_{\ell^n}^{-1}(\pi_{\ell^n}(G_{E}))$.
\end{proof}

\begin{definition}\label{def-ell-simplest}
    Let $E/\Q(j_{K,f})$ be an elliptic curve with CM by $\Of$, and let $\ell$ be a prime number such that $[\mathcal{N}_{\delta, \phi}(\ell^{\infty}) : G_{E, \ell^{\infty}}] = d_E$. Then, we say that $E$ is an \emph{$\ell$-simplest CM curve}, or equivalently, a \emph{simplest CM curve for the prime $\ell$}. 
\end{definition}

Put another way, a \emph{simplest CM curve} is an elliptic curve $E/\Q(j_{K,f})$ with CM whose adelic image is completely determined by its $\ell$-adic data. In the following theorem, we use the results from \cite{modelspaper} to determine all the $\ell$-simplest CM curves defined over $\Q$, where  $E^N$ denotes a quadratic twist of $E$ by a non-zero, square-free integer $N$. We remark here that one could also use the results of \cite{modelspaper} to determine the simplest CM curves defined over quadratic number fields.

\begin{theorem}\label{thm-40simplestcurves}\label{lem-TwistsPreservingSimplestCM}\label{thm-TwistsPreservingSimplestCM}
    There are precisely $40$ simplest CM elliptic curves defined over $\Q$, and each such elliptic curve $E/\Q$ is given in Table \ref{simplesttable} together with the prime $\ell$ such that $E$ is an $\ell$-simplest curve. Moreover,
    \begin{enumerate}
        \item Let $\Of$ be a CM order such that $j_{K,f}\in \Q$. Then there are at least two simplest CM curves with CM by $\Of$ defined (and non-isomorphic) over $\Q$ and $j$-invariant equal to $j_{K,f}$.
        \item Let $E/\Q$ be an $\ell$-simplest curve with CM by $\Of$.
        \begin{enumerate} 
        \item $\ell$ is the unique prime dividing $\Delta_K$.
        \item The conductor of $E/\Q$ is a power of $\ell$, except when $\Delta_K f^2 = -12$ (and $\ell=3$), in which case the conductor is $36$.
        \item The quadratic twist $E^{-\ell}/\Q$ is another non-isomorphic $\ell$-simplest curve with CM by $\Of$.
        \item Let $E'/\Q$ be another CM curve with CM by $\Of$. Then $E'$ is an $\ell$-simplest curve if and only if $E'$ is isomorphic over $\Q$  to 
        \begin{itemize}
            \item ($j(E)\neq 0,1728$, $\ell>2$) the quadratic twist $E^{N}$ for $N\in \{1,-\ell\}$, or
            \item ($j(E)\neq 0,1728$, $\ell=2$) the quadratic twist $E^{N}$, where $N \in \{\pm 1, \pm 2\}$, and in this case $\Delta_Kf^2 = -8$ or $-16$, or
                       \item ($j(E)=0$, $\ell=3$) the sextic twists $y^2 = x^3 + 16N$ with $N\in \{1,-3,9,-27,81,-243 \}$, or
            \item ($j(E)=1728$, $\ell=2$) the quartic twists $y^2=x^3+Nx$ with $N\in \{\pm 1, \pm 2, \pm 4, \pm 8\}$.
 
        \end{itemize}
        \item If $E'/\Q$ is another $\ell$-simplest curve with CM by $\Of$, then $G_{E, \ell^{\infty}}$ is conjugate to $G_{E', \ell^{\infty}}$ in $\GL(2, \Z_{\ell})$ only when $E$ and $E'$ are $\Q$-isomorphic.
        \end{enumerate}
    \end{enumerate}
\end{theorem}
\begin{proof}
    In \cite{modelspaper}, the authors (together with Gonz\'alez-Jim\'enez) determined (i) every possible $\ell$-adic image $G_{E,\ell^\infty}$ of an $\ell$-adic Galois representation attached to an elliptic curve $E/\Q$ with CM, and (ii) the model of every curve $E/\Q$ with $\ell$-adic image conjugate to $G_{E,\ell^\infty}$. In particular, the possible images and twists with each image are listed in Tables 1, 2, 3, and 4 of that paper. We proceed by the maximal adelic index $d_E$ defined above, and let $E/\Q$ be an $\ell$-simplest CM curve.
    \begin{enumerate}
         \item[(i)] If $d_E=6$, then we must have $j(E)=0$ and $\ell=3$. There are six possible $3$-adic images of index $6$ in the $3$-adic normalizer, which are listed in \cite[Table 3]{modelspaper}. Each image occurs for exactly one elliptic curve $E/\Q$ which is a twist of $y^2 = x^3 + 16$ (\href{https://www.lmfdb.org/EllipticCurve/Q/27/a/4}{\texttt{27.a4}}). The fact that there are exactly $6$ such curves is shown in \cite[Theorem 3.17]{modelspaper}.
         
        \item[(ii)] If $d_E=4$, then we must have $j(E)=1728$ and $\ell=2$. There are eight possible $2$-adic images of index $4$ in the $2$-adic normalizer, which are listed in \cite[Table 2]{modelspaper}. Each image occurs for exactly one elliptic curve $E/\Q$ which is a twist of $y^2=x^3+x$ (\href{https://www.lmfdb.org/EllipticCurve/Q/64/a/4}{\texttt{64.a4}}). The fact that there are exactly $8$ such curves is shown in \cite[Theorem 3.5]{modelspaper}.
        \item[(iii)] If $d_E=2$ and $j(E)\neq 0,1728$, then the tables of \cite{modelspaper} show that there are $26$ $\ell$-adic images of index $2$ in the $\ell$-adic normalizer, as listed in Table \ref{simplesttable} in the current paper. Further, for those images realized within each CM class, each image is realized by exactly one simplest curve.
    \end{enumerate}
    Therefore, there are $40$ simplest CM elliptic curves over $\Q$, as claimed.

  Now, part (1) follows simply from inspection of Table \ref{simplesttable} since we list at least $2$ simplest curves for each of the $13$ CM $j$-invariants defined over $\Q$. Moreover, in all cases, either 
    \begin{itemize}
        \item 
    $\ell=2$ and $\Delta_K = -4$ or $-8$, or 
    \item $\ell>2$ and $\Delta_K=-\ell$,
    \end{itemize}
    so $\ell$ is the unique prime dividing $\Delta_K$. This shows (2a). Part (2b) follows from inspection of the conductors of each curve in Table \ref{simplesttable}. Further, for (2c), in all cases, if $E/\Q$ is $\ell$-simplest, then $E^{-\ell}/\Q$ is non-isomorphic over $\Q$ and it is also an $\ell$-simplest curve with the same CM (see Tables 1-4 in \cite{modelspaper}). 

Part (2d) follows, for $j(E)\neq 0,1728$, from the fact that when $\ell>2$ there are exactly two $\ell$-simplest curves, which are precisely $E$ and $E^{-\ell}$ (a fact that is shown in \cite[Table 4]{modelspaper}), and if $\ell=2$ (so $\Delta_Kf^2=-8$ or $-16$), then there are exactly four $2$-simplest curves, which are $E$, $E^{-1}$, $E^2$, and $E^{-2}$ (a fact that is shown in \cite[Table 2]{modelspaper}). When $j(E)=0$ then $\ell=3$ and the simplest CM curves correspond to the sextic twists $y^2=x^3 + 16N$ with  $N\in \{1,-3,9,-27,81,-243 \}$ (\cite[Table 4]{modelspaper}), and when  $j(E)=1728$ then $\ell=2$ and the simplest CM curves are the quartic twists $y^2=x^3+Nx$ with $N\in \{\pm 1, \pm 2, \pm 4, \pm 8\}$ (\cite[Table 2]{modelspaper}).

Finally, part (2e) follows from inspection of (\cite[Tables 2 and 4]{modelspaper}), which show that, within each CM class, each $\ell$-adic image with index $d_E$ is realized by exactly one $\ell$-simplest curve up to $\Q$-isomorphism.
\end{proof}

\begin{center}
\begin{table}
\caption{A complete list of $\ell$-simplest CM curves defined over $\Q$. For each $\ell$-simplest curve, we provide its CM discriminant $\Delta_K f^2$, the prime $\ell$, the maximal adelic index $d_E$, and its LMFDB label.}\label{simplesttable}
\begin{tabular}{|c|c|c|c|}
\hline
$\Delta_K f^2$ & $\ell$ & $d_E$ & LMFDB Labels \\
\hline
$-3$ & $3$ & $6$ & \eclabel{27.a3}, \eclabel{27.a4}, \eclabel{243.a1}, \eclabel{243.a2}, \eclabel{243.b1}, \eclabel{243.b2}  \\
\hline
$-12$ & $3$ & $2$ & \eclabel{36.a1}, \eclabel{36.a2}  \\
\hline
$-27$ & $3$ & $2$ & \eclabel{27.a1}, \eclabel{27.a2}  \\
\hline
$-4$ & $2$ & $4$ & \eclabel{32.a3}, \eclabel{32.a4}, \eclabel{64.a3}, \eclabel{64.a4}, \eclabel{256.b1}, \eclabel{256.b2}, \eclabel{256.c1}, \eclabel{256.c2}  \\
\hline
$-16$ & $2$ & $2$ & \eclabel{32.a1}, \eclabel{32.a2}, \eclabel{64.a1}, \eclabel{64.a2}    \\
\hline
$-7$ & $7$ & $2$ & \eclabel{49.a2}, \eclabel{49.a4}  \\
\hline
$-28$ & $7$ & $2$ & \eclabel{49.a1}, \eclabel{49.a3}   \\
\hline
$-8$ & $2$ & $2$ & \eclabel{256.a1}, \eclabel{256.a2}, \eclabel{256.d1}, \eclabel{256.d2}  \\
\hline
$-11$ & $11$ & $2$ & \eclabel{121.b1}, \eclabel{121.b2}  \\
\hline
$-19$ & $19$ & $2$ & \eclabel{361.a1}, \eclabel{361.a2}  \\
\hline
$-43$ & $43$ & $2$ & \eclabel{1849.b1}, \eclabel{1849.b2}   \\
\hline
$-67$ & $67$ & $2$ & \eclabel{4489.b1}, \eclabel{4489.b2}   \\
\hline
$-163$ & $163$ & $2$ & \eclabel{26569.a1}, \eclabel{26569.a2}   \\
\hline
\end{tabular}
\end{table}
\end{center}

\begin{remark}
    Let $E/\Q$ be an elliptic curve with CM by $\Of$ and with $j_{K, f} \neq 0, 1728$, and let $K(E[p^{\infty}])$ denote the compositum of the family of division fields $\{ K(E[p^n])\}_{n \in \N}$. We note here that the curves we define here as a simplest CM curves correspond precisely to those curves such that the family $\{ K(E[q^{\infty}])\}_{q}$ is linearly disjoint over $K$, where the family runs over all rational primes $q$ (which have been previously studied in the work of Campagna and Pengo; see for example \cite[Theorem 6.3]{CampagnaPengo1}).
\end{remark}

As a corollary of Theorem \ref{thm-ell-simplest} we obtain an easy criterion to decide what twists are not simplest CM curves.

\begin{corollary}\label{cor-notsimplest} Let $E/\Q$ be an $\ell$-simplest CM curve and let $E'/\Q$ be a twist of $E$. Further, suppose the following:
\begin{enumerate}
    \item  $j(E)\neq 0,1728$, and  $E'/\Q$ is a quadratic  twist of $E$ by a non-zero square-free $N\neq 1$ with $\gcd(\ell,N^\dagger)=1$ (with $N^\dagger$ as defined in Theorem \ref{cyc-division-thm}), or
    \item $j(E)=1728$ and $\ell=2$, and $E'/\Q$ is a quartic twist $y^2=x^3+Nx$ with $N$ a non-zero fourth-power-free integer such that $N\neq \pm 1$ and $\gcd(2,N)=1$, or
    \item $j(E)=0$, and $E'/\Q$ is a sextic twist $y^2=x^3 + 16N$ with $N$ a non-zero sixth-power-free integer such that  $N\neq \pm 1$ and $\gcd(3,N)=1$,
\end{enumerate}
then $E'/\Q$ is not a simplest CM curve.
\end{corollary}
\begin{proof}
Let $E/\Q$ be an $\ell$-simplest CM curve, and suppose $E'/\Q$ is a quadratic, quartic, or sextic twist by $N\neq 1$ (non-zero, and square-, fourth-power-, or sixth-power-free, respectively) where $N$ satisfies the conditions of the corollary. Suppose for a contradiction that $E'$ also a simplest CM curve. Then, by Theorem \ref{thm-ell-simplest}, part (2d), it follows that $E'$ is isomorphic to  a quadratic, quartic, or sextic twist by $N'$, with $N'$ taking values as in the statement of (2d). Since $E^N \cong_\Q E' \cong_\Q E^{N'}$ and both $N$ and $N'$ are non-zero, and square, fourth-power, or sixth-power-free, respectively, then $N=N'$ (by \cite[Ch. X, Prop. 5.4, and Cor. 5.4.1]{silverman1}). Now, the conditions on $N$ in the statement of the corollary, together with the conditions on the possible values of $N'$ lead to a contradiction, and concludes the proof.
\end{proof}

Now we can put together Theorems \ref{thm-ell-simplest} and \ref{thm-40simplestcurves}  to explicitly describe the adelic image of each simplest CM curve over $\Q$.

\begin{theorem}\label{thm-adelicimageofsimplestcurves}
    Let $E/\Q$ be any of the $\ell$-simplest curves described in Theorem \ref{thm-40simplestcurves}, let $G_{E,\ell^\infty}$ be the $\ell$-adic image of $E$ as described in \cite{modelspaper}, and let $n=n_{E,\ell}$ be the number defined in Remark \ref{rem-ell-level-of-diff}. Let $G_{E,\ell^n}=\pi_{\ell^\infty,\ell^{n}}(G_{E,\ell^\infty})$, where $\pi_{\ell^\infty,\ell^n}\colon  \mathcal{N}
_{\delta, \phi}(\Z_\ell) \to \mathcal{N}_{\delta, \phi}(\ell^n)$ is given by reduction modulo $\ell^n$. Then, the adelic image is given by $G_E=\pi_{\ell^{n}}^{-1}(G_{E,\ell^n})$.
\end{theorem}
\begin{proof}
    Let $E/\Q$ be an $\ell$-simplest curve, let $G_E$ be its adelic image, and let $G_{E,\ell^\infty}$ be its $\ell$-adic image. Let $\pi_{\ell^\infty}\colon  \mathcal{N}
_{\delta, \phi} \to \mathcal{N}_{\delta, \phi}(\Z_{\ell})$ be the map induced by the projection $\tau_{\ell^\infty}\colon \widehat{\Z}\cong \prod_\ell \Z_\ell\to \Z_{\ell}$ to the $\ell$-adic component. Since $\tau_{\ell^\infty}\circ \rho_E = \rho_{E,\ell^\infty}$, it follows that $\pi_{\ell^\infty}(G_E)=G_{E,\ell^\infty}$. Let $\pi_{\ell^\infty,\ell^n}\colon  \mathcal{N}
_{\delta, \phi}(\Z_\ell) \to \mathcal{N}_{\delta, \phi}(\ell^n)$ be given by reduction modulo $\ell^n$, and let $G_{E,\ell^n}=\pi_{\ell^\infty,\ell^{n}}(G_{E,\ell^\infty})$.  Since $\pi_{\ell^\infty,\ell^n} \circ \pi_{\ell^\infty} = \pi_{\ell^n}$, it follows that 
$$G_{E,\ell^n}=\pi_{\ell^\infty,\ell^n}(G_{E,\ell^\infty}) = \pi_{\ell^n}(G_E).$$
Since $E/\Q$ is an $\ell$-simplest curve, now Theorem \ref{thm-ell-simplest} implies that $$G_E=\pi_{\ell^{n}}^{-1}(\pi_{\ell^n}(G_{E}))=\pi_{\ell^{n}}^{-1}(G_{E,\ell^n})$$ as claimed.
\end{proof}

\begin{example}\label{ex-intro1-details}
    Let $E/\Q$ be the elliptic curve given by $y^2+xy=x^3-x^2-107x+552$ (with LMFDB label \href{https://www.lmfdb.org/EllipticCurve/Q/49/a/2}{49.a2}). This elliptic curve has CM by $\Z[(1+\sqrt{-7})/2]$, the maximal order of $\Q(\sqrt{-7})$, and it is a $7$-simplest CM curve.  According to \cite[Table 4]{modelspaper}, the $7$-adic image of $E/\Q$ is given by $\langle J_{-7/4, 0},c_1 \rangle \subseteq \GL(2,\Z_7)$, where the group $J_{-7/4,0}$ is described in Theorem \ref{thm-oddprimedividingdisc} and $c_1= \left(\begin{array}{cc} 1 & 0\\  0 & -1 \\ \end{array}\right)$. However, for an adelic image of $E$, we have quantities $\delta=-2$ and $\phi=1$, so by Remark \ref{rem-changebasis}, the $7$-adic image of $\rho_{E,7^\infty}$ is conjugate to
$$G_{E,7^\infty}=\left\langle \left\{ \left(\begin{array}{cc} s+r/2 & r\\ -2r & s-r/2 \\\end{array}\right) : s \in ((\Z_7)^\times)^2, r\in \Z_7 \right\}, \left(\begin{array}{cc} 1 & 0\\ -1 & -1 \\\end{array}\right)\right\rangle$$
in $\GL(2,\Z_7)$, and it is of index $2$
 in the maximal possible image $\mathcal{N}_{-2,1}(\Z_7)$ allowed by the CM order. Moreover, the $7$-adic image is defined modulo $7$ (i.e., $n_{E,7}=1$ in the notation of Theorem \ref{thm-adelicimageofsimplestcurves}). Therefore, the minimal adelic level of definition of $\rho_E$ is $7$ and the adelic image is the full inverse image of $G_{E,7}=\pi_{7^\infty,7}(G_{E,7^\infty})$ in $\mathcal{N}_{-2,1}(\widehat{\Z})$ via the natural projection map  $\mathcal{N}_{-2,1}(\widehat{\Z})\to \mathcal{N}_{-2,1}(\Z/7\Z)$ given by reduction modulo $7$. Therefore, the adelic image is given by
 $$G_{E}=\left\langle \left\{ \left(\begin{array}{cc} a+b & b\\ -2b & a \\\end{array}\right) : a,b \in \widehat{\Z}, a^2+ab+2b^2 \in \widehat{\Z}^\times, a+b/2 \in ((\Z/7\Z)^\times)^2\right\}, \left(\begin{array}{cc} 1 & 0\\ -1 & -1 \\\end{array}\right)\right\rangle$$
 in $\GL(2,\widehat{\Z})$.
 \end{example}

While this paper concerns mostly those elliptic curve $E/\Q$ with CM and $j(E)\neq 0,1728$, we are also able to compute the adelic image of those curves with $j(E)=0,1728$ that are simplest, thanks to Theorem \ref{thm-adelicimageofsimplestcurves}.

 \begin{example}\label{ex-intro3-details} Let $E/\Q$ be given by $y^2=x^3+x$ (\href{https://www.lmfdb.org/EllipticCurve/Q/64/a/4}{\texttt{64.a4}}). This elliptic curve has CM by $\Z[i]$, the maximal order of $\Q(i)$, and it is a $2$-simplest CM curve. According to \cite[Table 4]{modelspaper}, the $2$-adic image $G_{E,2^\infty}$ of $E/\Q$ is the closure of  
     \[ G=\left\langle 5\cdot \operatorname{Id},\left(\begin{array}{cc} 1 & 2\\ -2 & 1\\\end{array}\right), \left(\begin{array}{cc} 0 & -1\\ -1 & 0\\\end{array}\right) \right\rangle\]
     as a subgroup of $\GL(2,\Z_2)$. The image $G$ is defined modulo $2^{n_{E,2}}=16$ (it is in fact defined modulo $4$ also), so let $G_{E,16}\equiv G \bmod 16$. Therefore, by Theorem \ref{thm-adelicimageofsimplestcurves}, the adelic image $G_E$ of $E$ is given by
   $$\left\langle \left\{ c_{-1,0}(a,b)=\left(\begin{array}{cc} a & b\\ -b & a \\\end{array}\right) : a,b \in \widehat{\Z}, a^2+b^2 \in \widehat{\Z}^\times, (c_{-1,0}(a,b) \bmod 16) \in G_{E,16}\right\}, \left(\begin{array}{cc} 0 & -1\\ -1 & 0 \\\end{array}\right)\right\rangle$$
   as a subgroup of $\GL(2,\widehat{\Z})$.
 \end{example}

To conclude this section, we prove that every CM elliptic curve $E/\Q$ (with $j(E)\neq 0,1728$) is a quadratic twist (possibly trivial) of an $\ell$-simplest CM curve, where the twist itself satisfies additional properties that will be used in the proof of the main theorem.

\begin{remark}\label{rem-notation}
    Let $E/\Q$ be an elliptic curve with CM by $\Of$, let $\ell$ be a prime, and let $N$ be a non-zero square-free integer. Throughout the rest of the paper, we fix the notation 
    \[ n_{E, \ell} := \begin{cases}
    1 & \text{ if } \ell > 3 \text{ or } \ell > 2 \text{ and } j(E) \neq 0, \\
    3 & \text{ if } \ell = 3 \text{ and } j(E) = 0, \\
    4 & \text{ if } \ell = 2,
    \end{cases} \]
    for the exponent of an $\ell$-adic level of differentiation as in Remark \ref{ell-level-of-diff}, and 
    \[ N^{\dagger} :=   
\begin{cases}
|N| & \text{ if } N \equiv 1 \bmod 4, \\
|4N| & \text{ if } N \equiv 2,3 \bmod 4, \\
\end{cases} \]
for the absolute value of the field discriminant of $\Q(\sqrt{N})$ as in Theorem \ref{kronecker-weber}.
\end{remark}

\begin{prop}\label{prop-correcttwist}
    Let $E/\Q$ be an elliptic curve with CM by $\Of$ such that $j(E)\neq 0,1728$, and let $\ell$ be the unique prime dividing $\Delta_K$. Then, there is an $\ell$-simplest CM curve $E'/\Q$ and a non-zero square-free integer $N\in\Z$ with $\gcd(\ell,N^\dagger)=1$, such that $E^N$ is isomorphic to $E'$ over $\Q$.
\end{prop}
\begin{proof}
    Let $E/\Q$ be an elliptic curve with CM by $\Of$ such that $j(E)=j_{K,f}\neq 0,1728$, and let $\ell$ be the unique prime dividing $\Delta_K$. By part (1) of Theorem \ref{thm-40simplestcurves}, there is an elliptic curve $E'/\Q$ with CM by $\Of$ that is $\ell$-simplest, and $j(E')=j_{K,f}$, such that its quadratic twist $E''=(E')^{-\ell}$ is also an $\ell$-simplest curve defined over $\Q$.

    Since $j(E)=j(E')=j_{K,f}\neq 0,1728$, it follows that $E'$ is isomorphic to a quadratic twist $E^d$ of $E$ by a non-zero square-free integer $d$. Let $d=N\ell^e$, where  $e=\nu_\ell(d)$ is the $\ell$-adic valuation of $d$, and $N=d\cdot \ell^{-e}$ is another square-free integer relatively prime to $\ell$. Note that $e=0$ or $1$ because $d$ is square-free.

    Let us first consider the case of $\ell\neq 2$. If $e=0$, then $d=N$ and since $\ell$ is odd, we have $\gcd(\ell,N)=\gcd(\ell,N^\dagger)=1$, and $E^N$ is isomorphic to $E'$, as desired. If $e=1$, then $E^{N\ell} \cong E'$ and therefore 
    $$E^{-N}\cong (E^{N\ell})^{-\ell}\cong (E')^{-\ell} =: E''.$$
    Hence if we put $N'=-N$, it follows that $E^{N'}\cong E''$ which is another $\ell$-simplest curve, and $\gcd(\ell,N)=\gcd(\ell,-N) = \gcd(\ell,(-N)^\dagger)=\gcd(\ell,(N')^\dagger)=1$ as needed.

    Now suppose that $\ell=2$. In this case, by Theorem \ref{thm-40simplestcurves} we have that $\Delta_K f^2 = -8$ or $-16$, and if $E'$ is $2$-simplest, then $(E')^M$ with $M=\pm 1$ or $\pm 2$ is also $2$-simplest. As before, suppose $d=N2^e$ with $N$ an odd integer, $e=\nu_2(d)$, and let 
    $$N^\ast = (-1)^{(N-1)/2}N=\begin{cases}
        N &\text{ if } N\equiv 1 \bmod 4,\\
        -N &\text{ if } N\equiv 3 \bmod 4.
    \end{cases}$$ 
    Then, we may write $d=N^\ast\cdot d'$ with $d'=(-1)^{(N-1)/2}\cdot 2^{e}$, and since $E'\cong E^d$ we have
    $$E^{N^\ast}\cong (E^{N^\ast \cdot d'})^{d'} \cong (E')^{d'} =: E'''$$
    where $E''' = (E')^{d'}$ is another $2$-simplest CM curve because $d'=\pm 1$ or $\pm 2$. Moreover, $$\gcd(2,(N^\ast)^\dagger)=1$$ because $(N^\ast)^\dagger = |N^\ast|=|N|$, since $N^\ast \equiv 1 \bmod 4$, and $\gcd(N,2)=1$. Hence, in all cases, there is a quadratic twist of $E$ isomorphic to an $\ell$-simplest curve, as claimed.
\end{proof}

\section{Determining Level of Definition}\label{sec-determinelevelofdef}

In this section, we will determine a level of definition for the adelic image of an elliptic curve $E/\Q$ with CM and $j(E) \neq 0, 1728$. We want to remark that our methods here build on previous work of the authors with Gonz\'alez-Jim\'enez (see \cite{modelspaper}) and that our goal is an explicit computational approach. Some of the results below could  in fact be retrieved from previous work of Campagna and Pengo (\cite{CampagnaPengo1}) but their approach is more theoretical and it would have to be modified anyway to achieve our more explicit computational goals.

To compute adelic levels of definition, we first need the following result.

\begin{lemma}[\cite{modelspaper}, Lemma 3.4]\label{modelspaper-twistlemma}\label{lem-twistimage}
    Let $E/F : y^2=x^3+Ax+B$ be an elliptic curve defined over a number field $F$, let $N>2$, and let $G_{E,N}$ be the image of $\rho_{E,N}\colon \Gal(\overline{F}/F)\to \GL(2,\Z/N\Z)$. Let $\alpha \in F$ and let $E^\alpha$ be the quadratic twist of $E$ by $\alpha$, i.e., $E^\alpha : \alpha y^2 = x^3+Ax+B$. Then, 
    \begin{enumerate}
        \item $F(E^\alpha[N])\subseteq F(\sqrt{\alpha},E[N])$. In particular, if $\sqrt{\alpha}\in F(E[N])$, then $F(E^\alpha[N])\subseteq F(E[N])$.
        \item If $[F(E^\alpha[N]):F] > [F(E[N]):F]$ and $\sqrt{\alpha}$ does not belong to $F(E[N])$, then $F(E^\alpha[N])=F(\sqrt{\alpha},E[N])$.
        \item If $\sqrt{\alpha}$ does not belong to $F(E[N])$, then $G_{E^\alpha,N}$ is conjugate to $\langle -\operatorname{Id},G_{E,N}\rangle \subseteq \GL(2,\Z/N\Z)$.
    \end{enumerate}
\end{lemma}

Next we prove a corollary of the previous lemma.

\begin{cor}\label{cor-twistimage}
    Under the assumptions of Lemma \ref{lem-twistimage}, if
    \begin{enumerate}
        \item $\sqrt{\alpha}$ does not belong to $F(E[N])$, and
        \item $-\operatorname{Id}$ does not belong to $G_{E,N}$, 
    \end{enumerate}
    then $G_{E^\alpha,N}$  is conjugate to $\langle -\operatorname{Id},G_{E,N}\rangle \subseteq \GL(2,\Z/N\Z)$, and $F(E^\alpha[N])=F(\sqrt{\alpha},E[N])$.
\end{cor}
\begin{proof}
    By Lemma \ref{lem-twistimage}, part (3), we have $G_{E^\alpha,N}$ is conjugate to $\langle -\operatorname{Id},G_{E,N}\rangle \subseteq \GL(2,\Z/N\Z)$. Thus, $\Gal(F(E[N])/F)\cong G_{E,N}$ and $\Gal(F(E^\alpha[N])/F)\cong \langle -\operatorname{Id},G_{E,N}\rangle$. Hence, 
    $$[F(E^\alpha[N]):F]=2\cdot [F(E[N]):F] > [F(E[N]):F]$$
    and by Lemma \ref{lem-twistimage}, part (2), we have $F(E^\alpha[N])=F(\sqrt{\alpha},E[N])$ as desired.
\end{proof}

In the following result we use the notation $n_{E,\ell}$ and $N^\dagger$ from Remark \ref{rem-notation}.

\begin{lemma}
\label{lem-twistofsimplest}
    Let $E/\Q$ be an $\ell$-simplest elliptic curve with CM by $\Of$ and $j(E) \neq 0, 1728$. For $N$ a non-zero, square-free integer, let $E^N$ denote the quadratic twist of $E$ by $N$.    If $\gcd(\ell, N^{\dagger}) = 1$, then $\Q(E^N[\ell^n]) = \Q(\sqrt{N}, E[\ell^n])$ for all $n\geq n_{E,\ell}$.
\end{lemma}

\begin{proof}
Let $E/\Q$ be as in the statement. If $N = 1$, then $E = E^N$ and the result is trivial. Otherwise, if $N\neq 1$ is non-zero, square-free,  and $\gcd(\ell,N^\dagger)=1$, then Corollary \ref{cor-notsimplest} shows that $E^N$ is not an $\ell$-simplest curve. Since $j(E^N)=j(E)\neq 0,1728$, it follows that  $[\mathcal{N}_{\delta, \phi}(\ell^{\infty}) : G_{E^N, \ell^{\infty}}] < d_E=2$ and so the index of the  $\ell$-adic image of $E^N$ in the $\ell$-adic normalizer must be $1$. Thus $[\mathcal{N}_{\delta, \phi}(\ell^n) : G_{E^N, \ell^n}]=1$ for all $n\geq 1$. 

Now consider $F=\Q(\sqrt{N})$. We will show that $F\cap \Q(E[\ell^n])=\Q$ for all $n\geq 1$. Since $N\neq 1$ and $N$ is square-free, the extension $F/\Q$ is quadratic and ramified at all primes dividing $N^\dagger$. Let $p$ be a prime that ramifies in $F/\Q$. We have assumed that $\gcd(\ell,N^\dagger)=1$, and therefore $\ell$ does not ramify in $F/\Q$ and so $p\neq \ell$. If $E/\Q$ is any of the curves in Table \ref{simplesttable}, except for \eclabel{36.a1} or \eclabel{36.a2}, it follows that $F\cap \Q(E[\ell^n])=\Q$, for any $n\geq 1$, because the conductor of $E/\Q$ is a power of $\ell$, and only $\ell$ can ramify in $\Q(E[\ell^n])$ by the criterion of N\'eron--Ogg--Shafarevich. 

If $E/\Q$ is \eclabel{36.a1} or \eclabel{36.a1}, then $\ell=3$, and only $2$ and $3$ may ramify in $\Q(E[3^n])$. If $N\neq -1,\pm 2$, then there is a prime $p>3$ that ramifies in $F$ and the same argument as above shows $F\cap \Q(E[\ell^n])=\Q$ for all $n\geq 1$. If $N= -1$, or $\pm 2$, then $F=\Q(i)$, $\Q(\sqrt{2})$, or $\Q(\sqrt{- 2})$, and if $F\cap \Q(E[3^n])\neq \Q$ for some $n\geq 1$, then we must have $F\cap \Q(E[3])=F$ because $[\Q(E[3^{n}]):\Q]=2\cdot 3^n$ for $n\geq 1$ (from the description of the image given by \cite[Table 4]{modelspaper}). However, $\Q(E[3])/\Q$ is a Galois $S_3$-extension and the unique quadratic field in $\Q(E[3])$ is $K=\Q(\sqrt{-3})$ as determined by Theorem \ref{cyc-division-thm}, because $E$ has CM by an order of $K$.

Now let $n_{E,\ell}$ be as in Remark \ref{rem-notation}, and recall from Remark \ref{rem-ell-level-of-diff} that $\ell^{n_{E,\ell}}$ is an $\ell$-adic level of differentiation (and thus also an adelic level of definition) for $E$. In particular, the index $[\mathcal{N}_{\delta, \phi}(\ell^n) : G_{E, \ell^n}]=2$ for all $n\geq n_{E,\ell}$. By \cite[Cor. 4.2]{lozano-galoiscm}, it follows that $-\operatorname{Id}$ does not belong to $G_{E,N}$. Since our work above shows that $\Q(\sqrt{N})\cap \Q(E[\ell^n])=\Q$, then Cororllary \ref{cor-twistimage} implest that $\Q(E^N[\ell^n])=\Q(\sqrt{N},E[\ell^n])$ for all $n\geq n_{E,\ell}$, as desired to complete the proof of the lemma.
\end{proof}

\begin{theorem}
\label{thm-entanglement}
Let $E/\Q$ be an elliptic curve with CM by $\Of$ and $j(E) \neq 0, 1728$, and suppose that $E$ is a simplest CM curve for the prime $\ell$. For prime $\ell$ and for $N$ a square-free integer with $N \notin \{0, 1 \}$, we define $n=n_{E,\ell}$, and $N^{\dagger}$ as in Remark \ref{rem-notation}.
If $\gcd(\ell, N^\dagger) = 1$, then $E^{N}/\Q$, the quadratic twist of $E$ by $N$, has a non-trivial entanglement between its $\ell^n$-division field and its $N^\dagger$-division field. In particular, $K(\sqrt{N}) = \Q(E^{N}[\ell^n]) \cap \Q(E^{N}[N^\dagger])$.

\end{theorem}

\begin{proof}

Let $E, N, \ell$ and $E^N$ be as described above. Note that in all cases, $N^\dagger$ and $\ell^n$ are greater than $2$, so $K \subseteq \Q(E^{N}[\ell^n]) \cap \Q(E^{N}[N^\dagger])$ by Theorem \ref{lem-clark}, part (3). Also, by Theorem \ref{sqrt-general} part (2), we have that $\Q(\sqrt{N}) \subseteq \Q(E^N[N^\dagger])$. Since $\gcd(\ell, N^\dagger) = 1$, we have that $\Q(E^N[\ell^n]) = \Q(\sqrt{N}, E[\ell^n])$ by Lemma \ref{lem-twistofsimplest}, so $K(\sqrt{N})$ is contained in the intersection $\Q(E^{N}[\ell^n]) \cap \Q(E^{N}[N^\dagger])$. We also note here that by Theorem \ref{thm-ell-simplest}, the prime $\ell$ is the unique prime ramifying in $K/\Q$, and therefore $\Q(\sqrt{N})\cap K=\Q$, and so $K(\sqrt{N})/K$ is a non-trivial quadratic extension.

Let $D=\ell^n N^\dagger$. Since $j(E) \neq 0, 1728$, it follows from Theorem \ref{thm-cmrep-intro-alvaro} that \[ [\mathcal{N}_{\delta, \phi}(D) : G_{E^N, D}] = [\mathcal{C}_{\delta, \phi}(D) : G_{E^N/K, D}] = 2,\] where $G_{E^N, D} := \im \rho_{E^N, D}$ and $G_{E^N/K, D} := \im \rho_{E^N/K, D}$ as in Section \ref{sec-galoisrepnotation}. By Corollary \ref{cor-galoistheoryofcompositum}, it follows that the degree of $[\Q(E^N[\ell^n])\cap \Q(E^N[N^\dagger]):K]$ divides $[\mathcal{C}_{\delta, \phi}(D) : G_{E^N/K, D}] = 2$.
Since $\Q(E^N[\ell^n])\cap \Q(E^N[N^\dagger])$ contains the quadratic extension $K(\sqrt{N})/K$, it follows that $K(\sqrt{N}) = \Q(E^{N}[\ell^n]) \cap \Q(E^{N}[N^\dagger])$, as claimed.
\end{proof}

We note here again that the previous result could also be derived from \cite[Theorem 6.3]{CampagnaPengo1}.

\begin{corollary}\label{cor-levelofdef}
    Let $E/\Q$ be an elliptic curve with CM by $\Of$ and $j(E) \neq 0, 1728$, and suppose that $E$ is an $\ell$-simplest CM curve for the prime $\ell$. Let $N$ be a square-free integer with $N \notin \{0, 1\}$, and let $n$ and $N^\dagger$ be as in Theorem \ref{thm-entanglement}. If $\gcd(\ell, N^\dagger) = 1$, then the adelic image of $E^N$ is defined at level $D=\ell^n N^\dagger$. Thus, $[\mathcal{N}_{\delta, \phi}(D) : G_{E^N, D}] = 2$.
\end{corollary}

\begin{proof}
    By Theorem \ref{thm-entanglement}, we have that $K(\sqrt{N}) = \Q(E^{N}[\ell^n]) \cap \Q(E^{N}[N^\dagger])$. Let $D=\ell^n N^\dagger$. Then, we have by Proposition \ref{prop-sameindex} that $[\mathcal{C}_{\delta, \phi}(D) : G_{E^N/K, D}] = [\mathcal{N}_{\delta, \phi}(D) : G_{E^N, D}] = 2$. The result now follows from Theorem \ref{thm-levelofdef}. 
\end{proof}

\begin{corollary}\label{cor-indexisalways2}
    Let $E/\Q$ be an elliptic curve with CM by $\Of$ and $j(E) \neq 0, 1728$. Then, all levels of definition for $G_E$ are divisible by $\ell$, the unique prime dividing $\Delta_K$.
\end{corollary}
\begin{proof}

   Suppose for a contradiction that $M\geq 2$ is another level of definition for $G_E$, such that $\gcd(M,\ell)=1$. We have just shown that $D=\ell^n N^\dagger$ is a level of definition, where $\gcd(\ell^n,N^\dagger)=1$. Suppose first that $N=1$, and so $E$ is an $\ell$-simplest curve, and $\ell^n$ is a level of definition. Then, by Prop. \ref{prop-gcdoneimpliesindex1}, the adelic index would be $1$, which contradicts Theorem \ref{thm-adelic-indexQ}. Hence such a level $M$ relatively prime to $\ell$ cannot exist.    
    
    Now suppose that $E$ is not $\ell$-simplest, so that $N\neq 1$, and $E=(E')^N$. By Theorem \ref{thm-entanglement}, we know that $K(\sqrt{N}) = \Q(E[\ell^n]) \cap \Q(E[N^\dagger])$. Then, by Cor. \ref{cor-galoistheoryofcompositum}, we have that $[\mathcal{N}_{\delta, \phi}(\ell^n) : G_{E, \ell^n}]=[\mathcal{N}_{\delta, \phi}(N^\dagger) : G_{E, N^\dagger}] = 1$, that is, the image is not defined at level $\ell^n$ or $N^\dagger$. Then, by Prop. \ref{prop-levelofdefngcd}, the number $d=\gcd(D,M)=\gcd(\ell^n N^\dagger,M)=\gcd(N^\dagger,M)$ (or $2\cdot \gcd(N^\dagger,M)$ if $d=\gcd(N^\dagger,M)=2$) is another level of definition. But then, $N^\dagger$ itself would also be a level of definition because $d\mid  N^\dagger$. However we have just shown that the image is not defined at level $N^\dagger$, which is a contradiction.

    Thus, in all cases, $\gcd(\ell,M)=1$ leads to a contradiction, and therefore all levels of definition are divisible by $\ell$ as claimed.
\end{proof}

\section{Image of Complex Conjugation}\label{sec-complexconj}

Let $E/\Q$ be an elliptic curve given by a short Weierstrass equation $y^2=x^3+Ax+B$ for some $A,B\in\Z$. Let $d$ be a non-zero square-free integer, and let $\chi=\chi_d\colon G_\Q \to \{\pm 1\}$ be the quadratic character attached to the field $\Q(\sqrt{d})$, defined by $\chi(\sigma)=\sigma(\sqrt{d})/\sqrt{d}$. Let $E^\chi$ be the twist of $E$ by $\chi$, given by the Weierstrass equation $dy^2=x^3+Ax+B$ (see \cite[Ch. X, Example 2.4]{silverman1}). 

The adelic Tate module $T(E)$ is naturally a $\widehat{\Z}[G_\Q]$-module, such that 
$$(n\cdot \sigma)(P) = [n](\sigma(P))$$
for any $n\in \widehat{\Z}$, $\sigma\in G_\Q$, and $P\in T(E)$. Moreover, following the construction of $E^\chi$ as in \cite[Ch. X, Example 2.4]{silverman1}, one sees that the adelic Tate module of $E^\chi$ is isomorphic as a $\widehat{\Z}[G_\Q]$-module to $T(E)$ together with the action defined by
$$(n\cdot \sigma)(P) = [n]([\chi(\sigma)]\cdot\sigma(P)) = [\chi(\sigma)\cdot n](\sigma(P)).$$
In particular, if we define $\widehat{\Z}^\chi$ to be the $\widehat{\Z}$-module $\widehat{\Z}$ together with the action $n\cdot \alpha = (\chi(\sigma)\cdot n) \cdot \alpha$, for any $n,\alpha\in\widehat{\Z}$, then it follows that
\begin{align} \label{eq-tatemodules}
    \widehat{\Z}^\chi \otimes_{\widehat{\Z}}  T(E) \cong \widehat{\Z}\otimes_{\widehat{\Z}}  T(E^\chi) \cong T(E^\chi)
    \end{align}
as $\widehat{\Z}[G_\Q]$-modules, where the copy of $\widehat{\Z}$ in the middle is supposed to be understood as a one-dimensional $\widehat{\Z}$-module with trivial Galois action. Notice that one can write an analogous isomorphism $T(E)\cong \widehat{\Z}\otimes_{\widehat{\Z}} T(E)$ of adelic Galois modules.

By rephrasing Eq.~(\ref{eq-tatemodules}) in terms of Galois representations, with the identification of the adelic Galois modules $T(E)$, $T(E^\chi)$, and $\widehat{\Z}^\chi$ with the adelic Galois representations $\rho_E$, $\rho_{E^\chi}$, and $\chi$, we obtain the following result. 

\begin{prop}
\label{cpx_conj_lemma}\label{cpx_conj_prop}
Let $E/\Q$ be an elliptic curve, and let $E^{\chi}/\Q$ be a quadratic twist of $E$ by a quadratic character $\chi$. Let $\rho_{E} \colon G_{\Q} \to \Aut(T(E))$, and $\rho_{E^\chi}$, be the adelic Galois representations attached to $E$ and $E^\chi$, respectively, and let $\chi\colon G_\Q \to \Aut(\widehat{\Z})$ be the quadratic adelic character given by $\chi(\sigma)(n) = (\sigma(\sqrt{d})/\sqrt{d})\cdot n$, for any $n\in \widehat{\Z}$. Then, there is an isomorphism of adelic Galois representations:
\[ \rho_{E^{\chi}} \cong \chi \otimes_{\widehat{\Z}} \rho_{E}. \]
\end{prop}

Now, we fix a $\widehat{\Z}$-basis $\{P,Q\}$ of $T(E)$ and the corresponding $\widehat{\Z}$-basis $\{\phi(P),\phi(Q)\}$ of $T(E^\chi)$ induced by the isomorphism $\phi\colon E\to E^\chi$ that sends $\phi((x_0,y_0))=(x_0,y_0/\sqrt{d})$, and let $\rho_E$ and $\rho_{E^\chi}$ be the adelic Galois representations $G_\Q\to \GL(2,\widehat{\Z})$ with respect to such $\widehat{\Z}$-bases. Noticing that for $R=(x_0,y_0)\in T(E)$ and $\sigma\in G_\Q$ we have
\begin{align*}
    \sigma(\phi(R)) &=\sigma((x_0,y_0/\sqrt{d}))\\
    &=(\sigma(x_0),\sigma(y_0/\sqrt{d}))\\
    &=(\sigma(x_0),\chi(\sigma)\cdot \sigma(y_0)/\sqrt{d})\\
    &=[\chi(\sigma)](\phi(\sigma(R))),
\end{align*}
and applying this for $R\in \{P,Q\}$, the elements of the adelic basis of $T(E)$, it follows that 
$$\rho_{E^\chi}(\sigma) = (\chi(\sigma)\cdot \operatorname{Id})\cdot \rho_{E}(\sigma),$$
as matrices in $\GL(2,\widehat{\Z})$. Thus, we have shown the following result.

\begin{prop}\label{cpx_conj_lemma2}\label{cpx_conj_prop2}
    Let $E/\Q$ be an elliptic curve, and let $E^{\chi}/\Q$ be a quadratic twist of $E$ by a quadratic character $\chi$. Let $\rho_{E} \colon G_{\Q} \to \GL(2,\widehat{\Z})$ and $\rho_{E^\chi}$ be the adelic Galois representations attached to $E$ and $E^\chi$, respectively, with respect to the adelic bases specified above. Then: 
$$\rho_{E^\chi}(\sigma) = (\chi(\sigma)\cdot \operatorname{Id})\cdot \rho_{E}(\sigma),$$
as matrices in $\GL(2,\widehat{\Z})$, for any $\sigma \in G_\Q$.
\end{prop}

While the proposition above works for any elliptic curve over $\Q$, we need a slightly more precise statement for our purposes of CM elliptic curves, to ensure that the adelic bases we specified above are compatible with the choices we made in Section \ref{sec-notation}. We clarify this point in the following result.

\begin{prop}\label{prop-basis-preserving}
    Let  $E/\Q$ be an elliptic curve with complex multiplication, and let $E^{\chi}/\Q$ be a quadratic twist of $E$ by a quadratic character $\chi$. Let $\rho_{E} \colon G_{\Q} \to \GL(2,\widehat{\Z})$  be the adelic Galois representation attached to $E$ as specified in Section \ref{sec-notation}, with respect to a $\widehat{\Z}$-basis $\mathcal{B}=\{P,Q\}$ such that $\rho_E(G_{\Q})\subseteq \mathcal{N}_{\delta,\phi}\subseteq \GL(2,\widehat{\Z})$. Let $\mathcal{B}_\chi=\{ \phi(P),\phi(Q)\}$ be the corresponding adelic basis of $T(E^\chi)$. Then, $\rho_{E^\chi/K}(G_{K})\subseteq \mathcal{C}_{\delta,\phi}$ and $\rho_{E^\chi}(G_{\Q})\subseteq \mathcal{N}_{\delta,\phi}$ as well. 
\end{prop}
\begin{proof}
    Suppose $E/\Q$ is as in the statement and we have chosen the basis $\mathcal{B}=\{P,Q\}$ such that $\rho_E(G_{\Q})\subseteq \mathcal{N}_{\delta,\phi}\subseteq \GL(2,\widehat{\Z})$ as in Section \ref{sec-notation}. With respect to this basis, for $\sigma\in G_K=\Gal(\overline{K}/K)$, the matrix $\rho_E(\sigma)\in \mathcal{C}_{\delta,\phi}$ corresponds to the multiplication-by-$\alpha$ matrix of an element $\alpha\in \OK\otimes \widehat{\Z}$ such that $\sigma$ induces the automorphism of $T(E)$ that is induced by $[\alpha]\in \End(E)$, i.e.,
    $$\sigma(P) = [\alpha](P), \ \text{ and } \ \sigma(Q) = [\alpha](Q),$$
    and so $\sigma(R)=[\alpha](R)$ for any $R\in T(E)$. Let $\phi\colon E\to E^\chi$ be the isomorphism defined above. Our work in this section shows that $\sigma(\phi(R))=[\chi(\sigma)](\phi(\sigma(R))$ for any $R\in T(E)$. Thus,
     $$\sigma(\phi(R))=[\chi(\sigma)](\phi(\sigma(R))=[\chi(\sigma)](\phi([\alpha](R)))=[\chi(\sigma)][\alpha](\phi(R))=[\chi(\sigma)\alpha](\phi(R))$$
     and therefore $\sigma$ acts as $[\chi(\sigma)\alpha]$ on $T(E^\chi)$, which corresponds to the equality 
     $$\rho_{E^\chi}(\sigma) = (\chi(\sigma)\cdot \operatorname{Id})\cdot \rho_{E}(\sigma).$$
     In particular, $\rho_{E^\chi}(\sigma)$ is the matrix of $\mathcal{C}_{\delta,\phi}$ that represents multiplication by $\chi(\sigma)\alpha \in \OK\otimes \widehat{\Z}$. This shows that $\rho_{E^\chi/K}(G_K)\subseteq \mathcal{C}_{\delta,\phi}$.

     Alternatively, for $\sigma\in G_K$, we have $\rho_E(\sigma) = c_{\delta,\phi}(a,b)\in \mathcal{C}_{\delta,\phi}$ for some $a,b\in \widehat{\Z}$, and 
$$\rho_{E^\chi}(\sigma) = (\chi(\sigma)\cdot \operatorname{Id})\cdot \rho_{E}(\sigma) = (\chi(\sigma)\cdot \operatorname{Id})\cdot c_{\delta,\phi}(a,b) = c_{\delta,\phi}(\chi(\sigma)\cdot a,\chi(\sigma)\cdot b),$$
which belongs to $\mathcal{C}_{\delta,\phi}$. Further, if $c\in G_\Q$ such that $\rho_E(c)=n \in \mathcal{N}_{\delta,\phi}$, then by the definition of the normalizer, we have $n=c_{\delta,\phi}(a,b)\cdot c_1$, for some $a,b\in\widehat{\Z}$ and $c_1$ as defined in Section \ref{sec-notation}. Thus,
$$\rho_{E^\chi}(c) = (\chi(c)\cdot \operatorname{Id})\cdot \rho_{E}(c) = (\chi(c)\cdot \operatorname{Id})\cdot c_{\delta,\phi}(a,b)\cdot c_1 = c_{\delta,\phi}(\chi(\sigma)\cdot a,\chi(\sigma)\cdot b)\cdot c_1=c_{\delta,\phi}(a,b)\cdot c_{\chi(c)},$$
which belongs to $\mathcal{N}_{\delta,\phi}$. (Note that $-\operatorname{Id}\in \mathcal{N}_{\delta,\phi}$, so $c_1$ and $c_{-1}\in \mathcal{N}_{\delta,\phi}$.) Hence, $\rho_E(G_\Q)\subseteq \mathcal{N}_{\delta,\phi}$, as desired.
\end{proof}

\begin{lemma}\label{lem-conjugatesandtwists}
    Let $M>1$ be an integer, let $E,E'$ be elliptic curves over $\Q$ and suppose that their Galois representations $\rho_{E,M}$ and $\rho_{E,M'}$ are conjugate. Let $\chi\colon G_\Q\to \mu_2$ be a quadratic character. Then, the Galois representations $\chi\otimes \rho_{E,M}$ and $\chi\otimes \rho_{E',M}$ (with tensor products being over $\Z/M\Z$) are also conjugate.
\end{lemma}
\begin{proof}
Suppose that $\rho_{E,M}$ and $\rho_{E',M}$ are conjugate Galois representations. This means that there is a matrix $B\in \GL(2,\Z/M\Z)$ such that $B\cdot \rho_{E,M}(\sigma)\cdot B^{-1} = \rho_{E',M}(\sigma)$, for all $\sigma\in G_\Q$. Then,
\begin{align*} B\cdot (\chi \otimes \rho_{E,M})(\sigma) \cdot B^{-1} &= B\cdot ((\chi(\sigma)\cdot \operatorname{Id})\cdot \rho_{E,M}(\sigma)) \cdot B^{-1}\\
&= (\chi(\sigma)\cdot \operatorname{Id})\cdot B\cdot \rho_{E,M}(\sigma) \cdot B^{-1}\\
&= (\chi(\sigma)\cdot \operatorname{Id})\cdot \rho_{E',M}(\sigma) = (\chi \otimes \rho_{E',M})(\sigma).
\end{align*}
Therefore, $\chi \otimes \rho_{E,M}$ and $\chi \otimes \rho_{E',M}$ are conjugate Galois representations as well, as claimed.
\end{proof}

\subsection{The cases of \texorpdfstring{$j=0,1728$}{j=0,1728}.}

Now, let $E/\Q$ be an elliptic curve with $j=1728$ or $0$ and therefore with CM by $\mathcal{O}=\OK$, the maximal order of $K=\Q(i)$ or $\Q(\sqrt{-3})$, and $n=4$ or $6$, respectively. Let $d$ be an $n$-th power free integers and let $\chi\colon G_\Q\to \mu_n$ be the cocycle defined by $\chi(\sigma)=\sigma(\sqrt[n]{d})/\sqrt[n]{d}$ (notice that $\chi\colon G_\Q \to \mu_d$ is only a cocycle, but when restricted to $G_K$ it becomes a character). Then, 
$$\widehat{\mathcal{O}},\widehat{\mathcal{O}^\chi}, T(E), T(E^\chi)$$
are $\widehat{\mathcal{O}}[G_K]$-modules as in the beginning of Section \ref{sec-complexconj}. In addition, we remind the reader that $\Aut(E)\cong \mu_n$ in this case, and $\mu_n\subseteq \mathcal{O}$. After choosing compatible adelic bases of $T(E)$ and $T(E^\chi)$ so that $\Aut(T(E))\cong \Aut(T(E^\chi))\cong \GL(2,\widehat{\Z})$, we can also write $\chi(\sigma)\in \Aut(T(E))$ in terms of the same basis, so that we can consider $\chi\colon G_K \to \GL(2,\widehat{\Z})$ (note that before, when $n=2$, we could have done similarly and $\chi(\sigma) = \pm \operatorname{Id}\in \GL(2,\widehat{\Z})$). 

Moreover, there is an isomorphism between $E$ and $E^\chi$ (both given in short Weierstrass form) given by 
$$(x_0,y_0) \mapsto \left(\frac{x_0}{d^{2/n}},\frac{y_0}{d^{3/n}}\right),$$
and if $\zeta\in \mu_n$, then $[\zeta]\in \Aut(E)$ is given by
$$(x_0,y_0) \mapsto \left(\frac{x_0}{\zeta^{2}},\frac{y_0}{\zeta^{3}}\right),$$
and therefore, it follows that 
\begin{align*}
    \sigma(\phi(R)) &=\sigma\left(\left(\frac{x_0}{d^{2/n}},\frac{y_0}{d^{3/n}}\right)\right)\\
    &= \left(\frac{\sigma(x_0)}{\chi(\sigma)^2d^{2/n}},\frac{\sigma(y_0)}{\chi(\sigma)^3d^{3/n}}\right)\\
    &=[\chi(\sigma)](\phi(\sigma(R))),
\end{align*}
for any $R\in T(E)$ and any $\sigma\in G_K$. Hence, we obtain a result that is analogous to Prop. \ref{cpx_conj_prop} and \ref{cpx_conj_prop2} which we record below.

\begin{prop}
\label{cpx_conj_lemma3}\label{cpx_conj_prop3}
Let $E/\Q$ be an elliptic curve with CM by $K=\Q(i)$ or $\Q(\sqrt{-3})$, and let $E^{\chi}/\Q$ be a quadratic twist of $E$ by a cocycle $\chi$ defined as above. Let $\rho_{E} \colon G_K \to \Aut(T(E))$, and $\rho_{E^\chi}$, be the adelic Galois representations attached to $E/K$ and $E^\chi/K$, respectively, and let $\chi\colon G_K \to \Aut(\widehat{\mathcal{O}})$ be the adelic character given by $\chi(\sigma)(n) = (\sigma(\sqrt[n]{d})/\sqrt[n]{d})\cdot n$, for any $n\in \widehat{\mathcal{O}}$. Then, 
\begin{enumerate}
    \item there is an isomorphism of adelic Galois representations:
\[ \rho_{E^{\chi}} \cong \chi \otimes_{\widehat{\mathcal{O}}} \rho_{E}. \]
    \item with a compatible choice of adelic bases for $T(E)$ and $T(E^\chi)$, it follows that 
    $$\rho_{E^\chi}(\sigma) = \chi(\sigma)\cdot \rho_{E}(\sigma),$$
as matrices in $\GL(2,\widehat{\Z})$, for any $\sigma \in G_K$.
\end{enumerate}
\end{prop}

% {\bf ALR: Leaving the old lemma below here for now for reference, just so you can compare with the new result above. Then, let us delete it and use the one above.}
% \begin{lemma}
%     Let $E/\Q$ be an elliptic curve with CM by $\Of$, let $M \geq 2$ be a positive integer, let $N \in \Z$ be a non-zero, square-free integer, and let $c \in G_{\Q}$ be a complex conjugation. If $E^N$ is the quadratic twist of $E$ by $N$, then 
%     \[ \rho_{E^N, M}(c) = \chi_{N}(c) \cdot \rho_{E, M}(c). \]
% \end{lemma}

\subsection{Elements not in the Cartan}

In order to compute a mod-$M$ image, we will need to find both the elements of the image in the Cartan subgroup (see next section), and one image element that is not in the Cartan. The next result provides a way to construct such non-Cartan elements.

\begin{prop}\label{prop-reconstructimage}
    Let $E/\Q$ be an  elliptic curve with CM by $\Of$, with $j(E)\neq 0,1728$, and let $N\geq 1$ be a level of definition for $G_E$. Suppose an adelic basis of $T(E)$ has been chosen, as in Thm. \ref{thm-cmrep-intro-alvaro}, so that $G_E=\rho_E(G_\Q)\subseteq \mathcal{N}_{\delta, \phi}$, and let $G_{E/K}=\rho_E(G_K)$. Then:
    \begin{enumerate}
        \item $G_{E/K}=\rho_E(G_K)$ is a subgroup of $\mathcal{C}_{\delta, \phi}$. Moreover, $G_{E/K}=G_E\cap \mathcal{C}_{\delta, \phi}$.
        \item Let $c\in G_\Q$ be an element that does not fix $K/\Q$, and let $C=\rho_E(c)\in \GL(2,\widehat{\Z}).$ Then, $C$ belongs to $\mathcal{N}_{\delta, \phi}\setminus \mathcal{C}_{\delta, \phi}$, and $G_E=\langle C, G_{E/K} \rangle$.
        \item Let $G_{E,N}=\pi_{N}(G_E)$, let $G_{E/K,N}=\pi_{N}(G_K)$ and let $C_N \in G_{E,N}\setminus G_{E/K,N}$. Then, $G_{E,N}=\langle C_N, G_{E/K,N}\rangle$ and $G_E=\pi^{-1}_{N}(\langle C_N, G_{E/K,N}\rangle)$.
    \end{enumerate}
\end{prop}
\begin{proof}
    Let $E$ and $N$, and a choice of adelic basis of $T(E)$, be as in the statement. Then Theorem \ref{thm-cmrep-intro-alvaro}, part (2), shows that $G_{E/K,M}=G_{E,M}\cap \mathcal{C}_{\delta, \phi}(M)$ for any $M\geq 1$, and therefore the equality also holds adelically, i.e., $G_{E/K}=G_E\cap \mathcal{C}_{\delta, \phi}$. In particular, $G_{E/K}=\rho_E(G_K)$ is a subgroup of $\mathcal{C}_{\delta, \phi}$. This shows (1).

    Now, let $c\in G_\Q$ be an element that does not fix $K/\Q$, and let $C=\rho_E(c)\in \GL(2,\widehat{\Z}).$ Since $c$ does not fix $K$, it follows that $c\not\in G_{K}$, and since $G_{E/K}=G_E\cap \mathcal{C}_{\delta, \phi}$, it follows that $C$ is not in $\mathcal{C}_{\delta, \phi}$. Hence, $C\in \mathcal{N}_{\delta, \phi} \setminus \mathcal{C}_{\delta, \phi}$, as claimed.

    Moreover, since $j(E)\neq 0,1728$, we have $[\mathcal{N}_{\delta, \phi} : G_E]=2$ by Theorem \ref{thm-adelic-indexQ} and, by Prop, \ref{prop-sameindex},
    $$2=[\mathcal{N}_{\delta, \phi}:G_{E}] = [\mathcal{C}_{\delta, \phi}:G_{E}\cap \mathcal{C}_{\delta, \phi}]=[\mathcal{C}_{\delta, \phi}:G_{E/K}].$$
    Since we also have $[\mathcal{N}_{\delta, \phi}:\mathcal{C}_{\delta, \phi}]=2$ by \cite[Thm. 1.1(1)]{lozano-galoiscm}, it follows that $[G_E:G_{E/K}]=2$ as well (alternatively, $[G_\Q:G_K]=2$, so $[G_E:G_{E/K}]=1$ or $2$, but $G_E\neq G_{E,K}$ because $G_E$ is not contained in the Cartan subgroup by \cite[Thm. 6.7]{lozano-galoiscm}). And since $C\in G_E \setminus G_{E/K}$, we conclude that $G_E=\langle C, G_{E/K} \rangle$. This shows (2).

    Finally, let $N\geq 1$ be a level of definition for $G_E$. It follows that $2=[\mathcal{N}_{\delta, \phi}:G_{E}]=[\mathcal{N}_{\delta, \phi}(N):G_{E,N}]$, and replacing adelic groups by groups modulo $N$, the argument we used to show (2) also proves that $G_{E,N}=\langle C_N, G_{E/K,N}\rangle$. Thus, since $N$ is a level of definition, we have
    $$G_E=\pi^{-1}_N(G_{E,N})=\pi^{-1}_{N}(\langle C_N, G_{E/K,N}\rangle),$$
    as desired. This concludes the proof of (3).
\end{proof}

\section{Finding the Image of Cartan}\label{sec-imageofcartan}

Now that we are able to determine the image of complex conjugation, we can turn to understanding the image of Cartan at a level where the adelic image is defined. To do this, we need the following group theoretic result.

\begin{prop}\label{prop-grouptheoryentanglement}
    For each $i=1,2$, let $G_i$ be an abelian group, let $H_i\subseteq G_i$ be a subgroup of index $2$, let $G=G_1\times G_2$, and let $\pi_i\colon G\to G_i$ be the natural projection to the $i$-th coordinate. Let $g_i\in G_i$ be elements such that $G_i=\langle g_i, H_i\rangle$. Then, the subgroup $\cC=\langle (g_1,g_2),H_1\times H_2\rangle$ is the unique subgroup of $G$ with the following properties:
    \begin{enumerate}
        \item $H_1\times H_2 \subseteq \cC$,
        \item $\cC$ is of index $2$ in $G$,
        \item $\pi_i(\cC)=G_i$ for $i=1,2$, and
        \item $\cC\cap (H_1\times G_2) = \cC\cap (G_1\times H_2)$ is of index $4$ in $G$.
    \end{enumerate}
\end{prop}
\begin{proof}
    Let $G$, $\cC$, $G_i$, $H_i$, $g_i$, and $\pi_i$ be as in the statement of the proposition. We shall prove that $\cC$ verifies conditions (1), (2), (3), and (4), and $\cC$ is the unique such subgroup satisfying all four conditions.

    First note that $H_1\times H_2$ is of index $4$ in $G=G_1\times G_2$, and in fact $G/(H_1\times H_2)\cong \Z/2\Z \times \Z/2\Z$. By the fourth (or lattice) isomorphism theorem, there are precisely $5$ subgroups of $G$ containing $H_1\times H_2$, namely:
$$H_1\times H_2, \langle (g_1,e_2),H_1\times H_2\rangle, \langle (e_1,g_2),H_1\times H_2\rangle, \langle (g_1,g_2),H_1\times H_2\rangle, G.$$
Of these, $G$ itself is of index $1$, $H_1\times H_2$ is of index $4$, and the other $3$ are of index $2$. Since $\pi_i((g_1,g_2))=g_i$, and $H_1\times H_2\subseteq \cC$, it follows that $\pi_i(\cC)=G_i$. Finally,
$$\cC\cap (H_1\times G_2) = \cC\cap (G_1\times H_2)=H_1\times H_2,$$
and so $\cC$ as defined in the statement satisfies (1), (2), (3), and (4).

Now let $\cC'$ be another subgroup of $G$ that satisfies (1), (2), (3), and (4). By (1), we have $H_1\times H_2 \subseteq \cC'$. By (2) and our arguments above, we have 
$$\cC' = \langle (g_1,e_2),H_1\times H_2\rangle,\ \langle (e_1,g_2),H_1\times H_2\rangle, \text{ or } \langle (g_1,g_2),H_1\times H_2\rangle$$
but of these three possibilities there is only one subgroup that satisfies conditions (3) and (4), namely $\langle (g_1,g_2),H_1\times H_2\rangle$.
\end{proof}

In the following result we verify that the Chinese remainder theorem ``commutes'' with the connecting projection maps among Cartan subgroups.

\begin{prop}\label{prop-CRTcompatible}
    Let $M,N$ be relatively prime integers, and let $\delta,\phi$ be fixed. Then, the isomorphism
    $$\phi\colon \mathcal{C}_{\delta, \phi}(M)\times \mathcal{C}_{\delta, \phi}(N) \to \mathcal{C}_{\delta, \phi}(MN)$$
    given by the Chinese remainder theorem preserves our choices of bases. In other words, 
    $$\phi((\pi_{M}(g),\pi_N(g)))=\pi_{MN}(g)$$
    for any $g\in \mathcal{C}_{\delta, \phi}$.
\end{prop}
\begin{proof}
    Let $g\in \mathcal{C}_{\delta, \phi}$ be an adelic element of the Cartan subgroup, and consider $g_{MN}=\pi_{MN}(g)$, $g_M=\pi_M(g)$, and $g_N=\pi_N(g)$. Since our choices of bases are compatible under reduction maps, it follows that 
    $$\pi_M(\pi_{MN}(g)) = \pi_M(g) \ \text{ and } \ \pi_N(\pi_{MN}(g)) = \pi_N(g),$$
    i.e.,
    $$\pi_{MN,M}(g_{MN})=g_M \ \text{ and } \ \pi_{MN,N}(g_{MN})=g_N.$$
    Now, by definition of the Chinese remainder isomorphism, $\phi((g_M,g_N))$ is the unique matrix $h$ modulo $MN$ such that $h\equiv g_M \bmod M$ and $h\equiv g_N \bmod N$. Since $g_{MN}$ satisfies these properties, and by the uniqueness of $h$, we conclude that $h=g_{MN}$, as desired.
\end{proof}

\begin{figure}
    \centering 
\begin{tikzcd}
&  & {\mathbb{Q}(j_{K,f},E[\ell^nM])} \arrow[rrd, no head] \arrow[dd, "H", no head] \arrow[ddd, "\mathcal{C}", no head, bend right] & & \\
{\mathbb{Q}(j_{K,f},E[\ell^n])} \arrow[rru, no head] \arrow[rrd, "H_{\ell^n}", no head] \arrow[rrdd, "\mathcal{C}({\ell^n})"', no head] \arrow[rrddd, "\mathcal{N}({\ell^n})"', no head, bend right] &  &  
&  & {\mathbb{Q}(j_{K,f},E[M])} \\
&  & F \arrow[d, no head] \arrow[rru, "H_M", no head]  &  &  \\
&  & {K(j_{K,f})} \arrow[rruu, "\mathcal{C}(M)"', no head] \arrow[d, no head] &  & \\
&  & {\mathbb{Q}(j_{K,f})} \arrow[d, no head] \arrow[rruuu, "\mathcal{N}(M)"', no head, bend right] &  &  \\
&  & \mathbb{Q} &  &                           
\end{tikzcd}
\caption{Field diagram of division fields from Theorem \ref{thm-cartan-image}}
\label{fig-fielddiagram}
\end{figure}
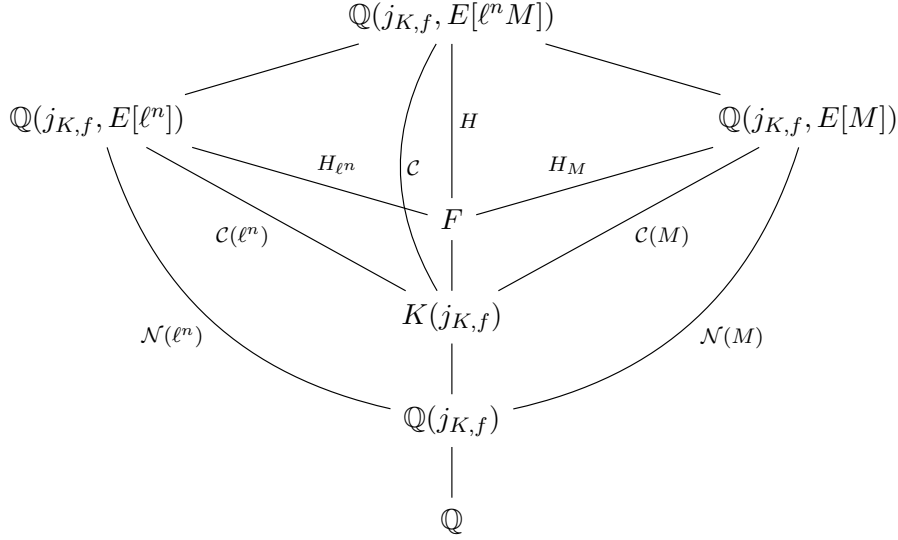

% \begin{lemma}
%     Let $E/\Q$ be an elliptic curve with CM by $\Of$ and $j(E) \neq 0, 1728$. Let $M, N > 2$ be relatively prime integers, let $L := MN$, and let $\rho_{E/K, L} \colon G_{K} \to \GL(2, \Z/L\Z)$ be the mod-$L$ Galois representation attached to $E/K$. 
% \end{lemma}

\begin{theorem}\label{thm-cartan-image} Let $E/\Q$ be an elliptic curve with CM by $\Of$ and $j(E) \neq 0, 1728$. Let $\ell$ be the unique prime dividing $\Delta_K$, and suppose the following:
\begin{enumerate}
    \item There is a non-zero square-free integer $N$ such that $\gcd(\ell,N^\dagger)=1$, and $K(\sqrt{N}) = \Q(E[\ell^n]) \cap \Q(E[N^\dagger])$, for some $n \geq 1$. 
    \item For $A\in \{ \ell^n, N^\dagger\}$ we have $\Gal(\Q(E[A]))/K)\cong \cC_{\delta,\phi}(A)$.
    \item For $A\in \{\ell^n,N^\dagger\}$, let $H_{A} \subseteq \mathcal{C}_{\delta, \phi}(A)$ with $H_{A} \cong \Gal(\Q(E[A])/K(\sqrt{N}))$, and let $g_{A}$ be a fixed element in $\mathcal{C}_{\delta, \phi}(A) \setminus H_{A}$. 
\end{enumerate}
Then, under the isomorphism between $\mathcal{C}_{\delta, \phi}(\ell^n N^\dagger)$ and $\mathcal{C}_{\delta, \phi}(\ell^n) \times \mathcal{C}_{\delta, \phi}(N^\dagger)$ induced by the Chinese remainder theorem, we have $G_{E/K, \ell^n N^\dagger} = \langle (g_{\ell^n}, g_{N^\dagger}), H_{\ell^n} \times H_{N^\dagger} \rangle.$

\end{theorem}

\begin{proof}

First, since $\ell$ is the unique prime dividing $\Delta_K$ and $\gcd(\ell,N^\dagger)=1$, the extension $K(\sqrt{N})/K$ is non-trivial quadratic, as in the proof of Theorem \ref{thm-entanglement}.

By assumption (2), we have $\Gal(\Q(E[\ell^n]))/K)\cong \cC_{\delta,\phi}(\ell^n)$, and  $\Gal(\Q(E[N^\dagger])/K) \cong \cC_{\delta,\phi}(N^\dagger)$. Moreover, the index of $\Gal(\Q(E[\ell^n N^\dagger])/K)$ in $\cC_{\delta,\phi}(\ell^n N^\dagger)$ is given precisely by the degree of the extension $\Q(E[\ell^n])\cap \Q(E[N^\dagger])/K$ (see Figure \ref{fig-fielddiagram}) which is $2$ by assumption (1). Hence, the Galois group $\Gal(\Q(E[\ell^n N^\dagger])/K)$ is a subgroup $\cC$ of index $2$ in $\cC_{\delta,\phi}(\ell^n N^\dagger)$. Moving forward, we use the notation $\cC(m) = \cC_{\delta, \phi}(m)$ for simplicity.

Since $\gcd(\ell,N^\dagger)=1$, we have an isomorphism $\cC(\ell^n N^\dagger)\cong \cC(\ell^n)\times \cC(N^\dagger)$ by the Chinese remainder theorem, so it follows that $\cC$ is a subgroup of index $2$ of $\cC(\ell^n)\times \cC(N^\dagger)$. Moreover, let $F := K(\sqrt{N})$ and consider $H = \Gal(\Q(E[\ell^n N^\dagger])/F)$ as a subgroup of $\cC \subseteq \cC(\ell^n)\times \cC(N^\dagger)$, and let $H_{\ell^n}$ and $H_{N^\dagger}$ be as in assumption (3). Notice that $H_{\ell^n}$, $H_{N^\dagger}$, and $H$ are index $2$ subgroups of $\cC(\ell^n)$, $\cC(N^\dagger)$, and $\cC$, respectively. By Galois theory and assumption (1) (see Figure \ref{fig-fielddiagram}), we must have $H_{\ell^n}\times H_{N^\dagger} = H \subseteq \cC$. Thus, $H_{\ell^n} \times H_{N^\dagger} \subseteq \cC(\ell^n)\times \cC(N^\dagger)\cong \cC(\ell^n N^\dagger)$ is a subgroup of index $4$. Moreover, notice that comparing indices inside $\cC(\ell^n N^\dagger)$, it follows that $[\cC: H_{\ell^n}\times H_{N^\dagger}]=2$. Thus:
\begin{enumerate}
    \item[(i)] $H_{\ell^n} \times H_{N^\dagger} \subseteq \cC$,
    \item[(ii)] $\cC$ is of index $2$ in $\cC(\ell^n N^\dagger)$, and
    \item[(iii)] $\pi_{\ell^n N^\dagger, \ell^n}(\cC)\cong \cC(\ell^n)$ and $\pi_{\ell^n N^\dagger, N^\dagger}(\cC)\cong \cC(N^\dagger)$ because these projections correspond to restrictions to the corresponding sub-division fields and, by assumption (2), $\Gal(\Q(E[\ell^n])/K)\cong \cC_{\delta,\phi}(\ell^n)$, and  $\Gal(\Q(E[N^\dagger])/K)\cong \cC_{\delta,\phi}(N^\dagger)$.
\end{enumerate}
As we have seen, we have $[\cC(\ell^n)\times \cC(N^\dagger):\cC]=2$ and $[\cC: H_{\ell^n}\times H_{N^\dagger}]=2$ but note that $\cC$ cannot be either of the index-$2$ subgroups $H_{\ell^n}\times \cC(N^\dagger)$ or  $ \cC(\ell^n) \times H_{N^\dagger}$ of $\cC(\ell^n N^\dagger)$ because neither one satisfies (iii) above. Thus we obtain:
\begin{enumerate} 
    \item[(iv)] $\cC \cap (H_{\ell^n}\times \cC(N^\dagger)) = \cC \cap (\cC(\ell^n) \times H_{N^\dagger})=H_{\ell^n}\times H_{N^\dagger}$ is of index $4$ in $\cC(\ell^n N^\dagger)$.
\end{enumerate}

Hence, we are in the situation described by Prop. \ref{prop-grouptheoryentanglement}. Let $g_{\ell^n}$ and $g_{N^\dagger}$ be as in the statement of the theorem, assumption (3). Then, $\cC \cong \langle (g_{\ell^n},g_{N^\dagger}),H_{\ell^n}\times H_{N^\dagger} \rangle$ by Prop. \ref{prop-grouptheoryentanglement}. We conclude with the observation that $\mathcal{C}$ is the unique subgroup with the properties above by Prop. \ref{prop-grouptheoryentanglement}, that the subgroup $G_{E/K, \ell^n N^\dagger} := \im \rho_{E/K, \ell^n N^\dagger}$ also satisfies properties (i)-(iv) as a subgroup of $\cC(\ell^n)\times \cC(N^\dagger)$, and that the isomorphism induced by the Chinese remainder theorem preserves the choice of basis such that $G_{E/K, \ell^n N^\dagger} \subseteq \mathcal{C}(\ell^n N^\dagger)$ by Prop. \ref{prop-CRTcompatible}. Thus, $G_{E/K, \ell^n N^\dagger} = \mathcal{C}$ as a subgroup of $\cC_{\delta, \phi}(\ell^n N^\dagger)$, as claimed.
\end{proof}

\section{Proof of the Main Result}\label{sec-proofofmaintheorem}

We now begin the proof of the main result of the paper, Theorem \ref{main-result}. First, we prove parts (1) and (2).

\begin{thm}\label{thm-main}
    Let $E/\Q$ be an elliptic curve with complex multiplication by an order $\Of$ of an imaginary quadratic field $K$, such that $j(E)\neq 0,1728$. Let $G_E$ be the image of $\rho_E\colon \GQ\to \GL(2,\widehat{\Z})$. Then,
    \begin{enumerate}
        \item There is a group $\mathcal{N}_{\delta,\phi}$ of $\GL(2,\widehat{\Z})$ such that $G_E\subseteq\mathcal{N}_{\delta,\phi}$ with index $[\mathcal{N}_{\delta,\phi}:G_E]=2$. 
        \item There is an explicitly computable integer $M=M_E\geq 2$ such that $G_E=\pi_M^{-1}(\pi_M(G_E))$, where $\pi_M\colon \mathcal{N}_{\delta,\phi}\to \mathcal{N}_{\delta,\phi}(\Z/M\Z)$. In other words, $G_E$ is defined modulo $M$.
    \end{enumerate}
\end{thm}

\begin{proof}
     Let $E/\Q$ be an elliptic curve with complex multiplication by an order $\Of$ of an imaginary quadratic field $K$, such that $j(E)\neq 0,1728$. We may make a compatible choice of bases, as in Section \ref{sec-notation}, so that $G_E := \im \rho_{E} \subseteq \mathcal{N}_{\delta, \phi}$. Part (1) then follows from Theorem \ref{thm-adelic-indexQ}, so we prove part (2).
     
     First, suppose that $E$ is an $\ell$-simplest CM curve for prime $\ell$. Then for $n := n_{E, \ell}$, the $\ell$-adic image of $E$ is defined at level $\ell^{n}$ (by Definition \ref{def-level-of-dif} and Remark \ref{ell-level-of-diff}), and so $G_E$ is defined modulo $\ell^{n}$ by Theorem \ref{thm-ell-simplest}. In this case, $M = M_E = \ell^n$.

    Now, suppose that $E$ is not a simplest CM curve, and let $\ell$ be the unique prime dividing $\Delta_K$. Then, Prop. \ref{prop-correcttwist} shows that there exists an $\ell$-simplest curve $E'/\Q$ and a non-zero, square-free integer $N \in \Z$ such that $E^N$, the quadratic twist of $E$ by $N$ is isomorphic to $E'$ and, moreover, $\gcd(\ell, N^{\dagger}) = 1$, where $N^{\dagger}$ is defined as in Remark \ref{rem-notation}. Since $E^N \cong E'$ and $j(E)\neq 0,1728$, it follows that $(E')^N \cong E$ as well, so we may assume that $E=(E')^N$ with $\gcd(\ell,N^\dagger)=1$. Then, Cor. \ref{cor-levelofdef} shows that the adelic image of $E$ is defined at level $M = M_E = \ell^n N^\dagger$, where $n=n_{E,\ell}$.
\end{proof}

It remains to prove part (3) of Theorem \ref{main-result}, i.e.,  the image $\pi_M(G_E)$ is explicitly computable. First, we note that for an $\ell$-simplest CM curve we may compute $\pi_{\ell^n}(G_E)$ explicitly using techniques from \cite[Section 12]{elladic}, or with the classification results proved in \cite{modelspaper}. Before we explain Algorithm \ref{alg-main} to compute $\pi_M(G_E)$ in general, which proves part (3) of Theorem \ref{main-result}, we require the following propositions. For the following proposition, we note that the $N$ that appears in the result guaranteed to exist by Prop. \ref{prop-correcttwist}.

\begin{proposition}\label{prop-ellfixingN}
    Let $E/\Q$ be an elliptic curve with CM by $\Of$. Let $N$ be a non-zero square-free integer and suppose that $E$ is not an $\ell$-simplest CM curve, while $E^N$ is an $\ell$-simplest CM curve such that $\gcd(\ell, N^\dagger) = 1$, and let $n := n_{E, \ell}$ (see Rem. \ref{rem-notation}). Then $K(\sqrt{N}) \subseteq K(E[\ell^n])$ and $\Gal(K(E[\ell^n])/K(\sqrt{N})) \cong \Gal(K(E^N[\ell^n])/K)$, and so $\Gal(K(E[\ell^n])/K(\sqrt{N})) \cong G_{E^N/K, \ell^n}$ is conjugate to an index $2$ subgroup of $G_{E/K, \ell^n} = \mathcal{C}_{\delta, \phi}(\ell^n)$, where $G_{E/K, \ell^n} := \im \rho_{E/K, \ell^n}$.
\end{proposition}

\begin{proof}    
    The fact that $K(\sqrt{N}) \subseteq K(E[\ell^n])$ follows from Theorem \ref{thm-entanglement}. Further, we have $\Q(E[\ell^n]) = \Q(\sqrt{N}, E^N[\ell^n])$ by Lemma \ref{lem-twistofsimplest}. Thus, we have
    \[ \Gal(K(E[\ell^n])/K(\sqrt{N})) \cong \Gal(K(\sqrt{N}, E^N[\ell^n])/K(\sqrt{N})) \cong \Gal(K(E^N[\ell^n])/K),\]
    where we are implicitly using that $\gcd(\ell,N^\dagger)=1$ and  $K(E^N[\ell^n])\cap K(\sqrt{N})=K$ (a fact that was shown in the course of the proof of Lemma \ref{lem-twistofsimplest}).

    Fix a compatible choice of basis so that $G_{E/K, \ell^n} \subseteq \mathcal{C}_{\delta, \phi}(\ell^n)$. Since $E$ is not an $\ell$-simplest CM curve, we in fact have $G_{E/K, \ell^n} = \mathcal{C}_{\delta, \phi}(\ell^n)$, and there is an isomorphism $\mathcal{C}_{\delta, \phi}(\ell^n) \cong \Gal(K(E[\ell^n])/K)$. Under this isomorphism, we conclude that $\Gal(K(E[\ell^n])/K(\sqrt{N})) \cong G_{E^N/K, \ell^n}$, which is conjugate to an index $2$ subgroup of $G_{E/K, \ell^n}$, as claimed.
\end{proof}

Let $m > 2$ be an integer and let $\zeta_m$ be a primitive $m$-th root of unity. Fix the isomorphism $(\Z/m\Z)^{\times} \cong \Gal(\Q(\zeta_m)/\Q)$ where $a \mapsto \sigma_a$ and $\sigma_a(\zeta_m) = \zeta_m^a$. If $N$ is a non-zero square-free integer with $\sqrt{N} \in \Q(\zeta_m)$, then we let $\mathcal{A}_{N}^{m}$ denote the subgroup of $(\Z/m\Z)^{\times}$ corresponding to automorphisms fixing $\sqrt{N}$. That is,
\begin{equation}\label{eq_An}
    \mathcal{A}_{N}^{m} := \left\{ a \in (\Z/m\Z)^{\times} : \sigma_a(\sqrt{N}) = \sqrt{N} \right\}.
\end{equation}

With this definition in mind, we state the following proposition.

\begin{proposition}\label{prop-MfixingN}
    Let $E/\Q$ be an elliptic curve with CM by $\Of$, and for non-zero square-free integer $N$ and integer $M > 2$, suppose that $\Q(\sqrt{N})$ is a degree $2$ extension of $\Q$ contained in $\Q(\zeta_M) \subseteq \Q(E[M]) = K(E[M])$, and suppose that $K \cap \Q(\zeta_M) = \Q$. Fix a compatible choice of basis so that $\im \rho_{E/K, M} \subseteq \mathcal{C}_{\delta, \phi}(M)$.
    Suppose further that $G_{E/K, M} := \im \rho_{E/K, M} = \mathcal{C}_{\delta, \phi}(M)$, so $\Gal(K(E[M])/K) \cong \mathcal{C}_{\delta, \phi}(M)$. Then there is a subgroup $H_{M} \subseteq \mathcal{C}_{\delta, \phi}(M)$ corresponding to the Galois group $\Gal(K(E[M])/K(\sqrt{N}))$ that is given by $H_{M} := \{ g \in \mathcal{C}_{\delta, \phi}(M) : \det(g) \in \mathcal{A}_{N}^{M} \}$.
    \end{proposition}

\begin{proof}
    Let $\rho_{E,M}$ be the mod-$M$ Galois representation attached to $E/\Q$ with respect to our usual choice of $\Z/M\Z$-basis such that its image $G_{E, M}$ is contained in $\mathcal{N}_{\delta, \phi}(M)$. In this case, since we have $G_{E/K, M} = \mathcal{C}_{\delta, \phi}(M)$, we also have $G_{E,M} = \mathcal{N}_{\delta, \phi}(M)$. It follows that $\Gal(\Q(E[M])/\Q) \cong \mathcal{N}_{\delta, \phi}(M)$ and \[ \Gal(\Q(E[M])/\Q(\zeta_M)) \cong \mathcal{N}_{\delta, \phi}(M) \cap \SL(2, \Z/M\Z) = \{ g \in \mathcal{N}_{\delta, \phi}(M) : \det(g) = 1 \}, \] 
    and hence $\Gal(\Q(E[M])/\Q(\sqrt{N})) \cong  \{ g \in \mathcal{N}_{\delta, \phi}(M) : \det(g) \in \mathcal{A}_{N}^{M}  \}$ by Galois theory. Further, since $\Q(E[M]) = K(E[M])$ for all $M>2$, and $K \cap \Q(\zeta_M) = \Q$ by assumption, we also have 
    \[ \Gal(K(E[M])/K(\sqrt{N})) \cong H_M = \{ g \in \mathcal{C}_{\delta, \phi}(M) : \det(g) \in \mathcal{A}_{N}^{M}  \}, \]
    as desired.
\end{proof}

\subsection{Algorithm to compute the adelic level of definition and adelic image}\label{subsec-alg-main} In this section we prove part (3) of Theorem \ref{main-result}. In order to do so, we will explain how to explicitly compute a Galois image at a level for which the adelic image is defined, using Algorithm \ref{alg-main}. First we state the justifications and details for the algorithm, which outlines a way to explicitly compute $\pi_{M}(G_E)$ as above. The computation proceeds as follows.

\begin{enumerate}
    \item Let $E/\Q$ be an elliptic curve with CM and $j(E)\neq 0,1728$. Given $j(E)$, the $j$-invariant of $E$, we can identify the CM field $K$, the CM order $\Of$, the unique prime $\ell$ that divides the discriminant $\Delta_K$, and a minimal short Weierstrass integral model $y^2 = x^3 + Ax + B$, such that $A$ and $B$ are minimal in the sense that if $d$ is a non-zero integer such that $d^4\mid A$ and $d^6\mid B$, then $d=\pm 1$.
    \item Consider the set of $\ell$-simplest CM curves with CM by $\Of$, as described in Table \ref{simplesttable}. Let $\{ A_i, B_i \}_{i \in I }$ be the set of coefficients of these simplest CM curves in short Weierstrass form, where $I$ is an indexing set of size $2$ or $4$.
    \item  If $E/\Q$ is $\ell$-simplest, i.e., it is isomorphic (over $\Q$) to some curve in Table \ref{simplesttable}, then its image can be determined as per Theorem \ref{thm-ell-simplest}. An explicit algorithm to determine this image is described in \cite[Section 12]{elladic}. Thus, we may assume $E$ is not $\ell$-simplest from now on. 
    \item By Lemma \ref{prop-correcttwist}, there is a non-zero square-free integer $N$ such that $E'=E^N$ is an $\ell$-simplest curve such that $\gcd(\ell,N^\dagger)=1$ (note that $N\neq 1$ since $E$ is assumed to be non-simplest). In order to determine $N$, we first find all indices $i \in I$ such that $A\equiv A_i \bmod (\Q^\times)^2$ and $B\equiv B_i \bmod (\Q^\times)^3$. Then we find the index $i\in I$, and a non-zero square-free $N$ such that $\gcd(\ell,N^\dagger)=1$ and the twist $E^N$ given by $y^2 = x^3 + N^2 A x + N^3 B$ is isomorphic to the $\ell$-simplest CM curve $E': y^2 = x^3 + A_i x + B_i$.

    \item Let $n = n_{E, \ell}$ be defined as in Remark \ref{rem-notation}. By Corollary \ref{cor-levelofdef}, the integer $M=\ell^n N^\dagger$ is an adelic level of definition for the image of $E$. (Note: Prop. \ref{prop-minimalleveldivideslevelsofdef} shows that the minimal level of definition $M_E$ is a divisor of $M$, and by Cor. \ref{cor-indexisalways2} the level $M_E$ is divisible by $\ell$.) 
    \item Let $M := \ell^n N^\dagger$. Proposition \ref{prop-reconstructimage} shows that $G_{E, M} = \langle C, G_{E/K, M} \rangle$, where $C := \rho_{E, M}(c)$ for any element $c \in G_{\Q}$ that does not fix $K$ and $G_{E/K, M}$ is the image of $\rho_{E/K, M} \colon G_{K} \to \GL(2, \Z/M\Z)$. 
    \item By Thm. \ref{thm-oddprimedividingdisc}(a) or Thm. \ref{thm-m8and16alvaro}, depending on whether $\ell > 2$ or $\ell = 2$, the image $\rho_{E^N,\ell^n}(G_\Q)$ contains an element of the form $c_\varepsilon$ (see Section \ref{sec-notation}). We remark here that \cite{modelspaper} identifies a precise element $c_\varepsilon$ in the mod-$\ell^n$ image. In other words, there is $c\in G_\Q$ such that $c_\varepsilon = \rho_{E^N,\ell^n}(c)$. Moreover, $c$ does not fix $K$, i.e., $c_\varepsilon \in \mathcal{N}_{\delta, \phi}(\ell^n)\setminus \mathcal{C}_{\delta, \phi}(\ell^n)$. Since the adelic image of $E^N$ is defined modulo $\ell^n$, we can choose any lift $c_{\varepsilon,M} \in \mathcal{N}_{\delta, \phi}(\Z/M\Z)$ of $c_{\varepsilon}\in \mathcal{N}_{\delta, \phi}(\Z/\ell^n\Z)$ and so there is an element $c'\in G_\Q$ such that $\rho_{E^N,M}(c')=c_{\varepsilon,M}$, and $c'$ does not fix $K$ (because $c$ does not, and $c'$ is equal to $c$ when we restrict to the action on $\ell^n$-torsion). Finally, it follows from Prop. \ref{cpx_conj_prop2} that $\rho_{E, M}(c') = (\chi_{N}(c') \cdot \Id) \cdot \rho_{E^N, M}(c')$, where $\chi_{N}$ is the quadratic character such that $E^N = E^{\chi_N}$. We let $C=\rho_{E,M}(c')$.
    
    \item From Theorem \ref{thm-cartan-image}, if $H_{A}$ is a subgroup of $\mathcal{C}_{\delta, \phi}(A)$ corresponding to $\Gal(\Q(E[A])/K(\sqrt{N}))$ and $g_{A} \in \mathcal{C}_{\delta, \phi}(A) \setminus H_{A}$ for $A \in \{ \ell^n, N^\dagger \}$, then $G_{E/K, M} = \langle (g_{\ell^n}, g_{N^\dagger}), H_{\ell^n} \times H_{N^\dagger} \rangle$ as a subgroup of $\mathcal{C}_{\delta, \phi}(M) \cong \mathcal{C}_{\delta, \phi}(\ell^n) \times \mathcal{C}_{\delta, \phi}(N^\dagger)$.
    \item By Prop. \ref{prop-ellfixingN}, we have $H_{\ell^n}$ is conjugate to $G_{E^N/K, \ell^n}$ as a subgroup of $\mathcal{C}_{\delta, \phi}(\ell^n)$, and by Prop. \ref{prop-MfixingN}, we have 
    \[ H_{N^\dagger} = \{ g \in \mathcal{C}_{\delta, \phi}(N^\dagger) : \det(g) \in \mathcal{A}_{N}^{N^\dagger}  \}, \] with $\mathcal{A}_{N}^{N^\dagger} $ as in (\ref{eq_An}).
    \item Note, from the proof of Theorem \ref{thm-cartan-image}, that $G_{E/K, M} = \langle (g_{\ell^n}, g_{N^\dagger}), H_{\ell^n} \times H_{N^\dagger} \rangle$ is the unique subgroup of $\mathcal{C}_{\delta, \phi}(M)$ satisfying 
    \begin{enumerate}
        \item $H_{\ell^n} \times H_{N^\dagger} \subseteq G_{E/K, M}$,
        \item $[\mathcal{C}_{\delta, \phi}(M) : G_{E/K, M}] = 2$,
        \item $G_{E/K, M} \bmod{\ell^n} = \mathcal{C}_{\delta, \phi}(\ell^n)$ and $G_{E/K, M} \bmod{N^\dagger} = \mathcal{C}_{\delta, \phi}(N^\dagger)$, and 
        \item $G_{E/K, M} \cap (H_{\ell^n} \times \cC(N^\dagger)) = G_{E/K, M} \cap (\cC(\ell^n) \times H_{N^\dagger})=H_{\ell^n}\times H_{N^\dagger}$
    \end{enumerate}
    \item It is thus sufficient to compute $G_{E/K, M}$ by realizing $H_{\ell^n} \times H_{N^\dagger}$, $H_{\ell^n} \times \cC(N^\dagger)$, and $\cC(\ell^n) \times H_{N^\dagger}$ as subgroups of $\mathcal{C}_{\delta, \phi}(M)$ via the Chinese remainder theorem (see Prop. \ref{prop-CRTcompatible}) and then finding the unique subgroup $G_{E/K,M}$ satisfying conditions (a)-(d) above.
    \item Together, we finally compute $G_{E, M} = \langle C, G_{E/K, M} \rangle$, with $C=\rho_{E,M}(c')$ as above.
\end{enumerate}

This work justifies the following algorithm, whose implementation completes the proof of Theorem \ref{thm-main}.

\begin{algorithm} \label{alg-main}
    \mbox{}
        \begin{enumerate}[leftmargin=6em]
        \item[\textsc{Input:}] An elliptic curve $E/\Q$ with CM by $\Of$ and $j(E) \neq 0, 1728$.
        \item[\textsc{Output:}] A Galois image $G_{E,M} := \im \rho_{E,M}$ where $M$ is an adelic level of definition for $E$.
    \end{enumerate}

    \begin{enumerate}
        \item Compute a non-zero square-free integer $N$ such that $E^N$ is an $\ell$-simplest CM curve for a prime $\ell$.
        \item If $N \in \{1, - \ell\}$ and $\ell$ is odd, or if $N \in \{\pm 1, \pm 2 \}$ and $\ell = 2$, return $G_{E, \ell^n}$ where $n := n_{E, \ell}$ as in Remark \ref{rem-notation}. Otherwise, continue to (3).
        \item Compute $N^{\dagger}$ and $n := n_{E, \ell}$ as in Remark \ref{rem-notation}, and let $M := \ell^{n} N^\dagger$.
        \item Compute $H_{\ell^n} \subseteq \mathcal{C}_{\delta, \phi}(\ell^n)$ and $H_{N^\dagger} \subseteq \mathcal{C}_{\delta, \phi}(N^\dagger)$ as in Theorem \ref{thm-cartan-image}, and compute $\mathcal{C}_{\delta, \phi}(M).$ 
        \item Compute $H_{\ell^n} \times H_{N^\dagger}$, $H_{\ell^n}\times \mathcal{C}_{\delta, \phi}(N^\dagger)$, and  $\mathcal{C}_{\delta, \phi}(\ell^n) \times H_{N^\dagger}$ as subgroups of $\mathcal{C}_{\delta, \phi}(M)$ via the Chinese remainder theorem.
        \item For $g_{\ell^{n}} \in \mathcal{C}_{\delta, \phi}(\ell^n) \setminus H_{\ell^n}$ and $g_{N^\dagger} \in \mathcal{C}_{\delta, \phi}(N^\dagger) \setminus H_{N^\dagger}$ do:
        \begin{enumerate}
            \item Compute $\mathcal{C} := \langle (g_{\ell^n}, g_{N^\dagger} ), H_{\ell^n} \times H_{N^\dagger} \rangle$ as a subgroup of $\mathcal{C}_{\delta, \phi}(M)$ via CRT.
            \item Check if $[\mathcal{C}_{\delta, \phi}(M) : \mathcal{C}] = 2$.
            \item Check if $\mathcal{C} \bmod \ell^n \cong \mathcal{C}_{\delta, \phi}(\ell^n)$ and $\mathcal{C} \bmod N^\dagger \cong \mathcal{C}_{\delta, \phi}(N^\dagger)$.
            \item Check if $\mathcal{C} \cap (H_{\ell^n}\times \mathcal{C}_{\delta, \phi}(N^\dagger)) = \cC \cap (\mathcal{C}_{\delta, \phi}(\ell^n) \times H_{N^\dagger})$.
        \end{enumerate}
        \item If all properties are satisfied, define $G_{E/K, M} := \mathcal{C}$. Otherwise return to (6).
        \item For $c \in G_{\Q}$ with $c_{\varepsilon} = \rho_{E^N, \ell^n}(c)$, compute $C := (\chi_{N}(c) \cdot \Id) \cdot c_{\varepsilon}$, and compute $C_M$ as a lift of $C$ modulo $M$.
        \item Return $G_{E, M} := \langle C_M, G_{E/K, M} \rangle$. 
    \end{enumerate}
\end{algorithm}

An implementation of this algorithm in \texttt{Magma} can be found in the GitHub repository \cite{githubrepo}.

\begin{remark}
    For an elliptic curve $E/\Q$ with CM by $\Of$ and $j(E) \neq 0, 1728$, let $M = \ell^{n_{E, \ell}} N^{\dagger}$ be the adelic level of definition for $E$ found using Algorithm \ref{alg-main}. We note that $M$ is not always the minimal level of definition, and in particular, if $\ell$ or $N^{\dagger}$ is even, we may find that the minimal adelic level of definition is a proper divisor of $M$. We illustrate these possibilities as follows.
\end{remark}

\begin{example}
    Let $E/\Q : y^2 = x^3-99x-378$ be the elliptic curve with LMFDB label \href{https://www.lmfdb.org/EllipticCurve/Q/288/d/1}{288.d1}. This curve has CM by $\Z[2i]$, the order of conductor $2$ in $K = \Q(i)$, which has discriminant $\Delta_K f^2 = -16$. The elliptic curve $E^{-3}/\Q : y^2=x^3-11x+14$ is a $2$-simplest CM curve with label \href{https://www.lmfdb.org/EllipticCurve/Q/32/a/2}{32.a2}. Thus,  since $(-3)^{\dagger} = 3$, we find that the adelic image of $E$ is defined at level $2^4 \cdot 3 = 48$.

    We can then compute $G_{E, 48}$ via Algorithm \ref{alg-main}. However, reducing this image modulo $12$, we find that $[\mathcal{N}_{-4, 0}(12) : G_{E, 12}] = 2$, so by Theorem \ref{thm-levelofdef}, the adelic image of $E$ is defined at level $12 = 2^2 \cdot 3$. In this case, the adelic level of definition produced by our algorithm is not minimal because $[\mathcal{N}_{-4, 0}(4) : G_{E^{-3}, 4}] = 2$, so the $2$-adic image of $E^{-3}$ is defined at level $4$.
\end{example}

\begin{example}
    Let $E/\Q : y^2=x^3-595x+5586$ be the elliptic curve with LMFDB label \href{https://www.lmfdb.org/EllipticCurve/Q/784/f/3}{784.f3}. This curve has CM by $\Z[\sqrt{-7}]$, the order of conductor $2$ in $K = \Q(\sqrt{-7})$, which has discriminant $\Delta_K f^2 = -28$. The elliptic curve $E^{-1}/\Q : y^2+xy=x^3-x^2-37x-78$ is a $7$-simplest CM curve with label  \href{https://www.lmfdb.org/EllipticCurve/Q/49/a/3}{49.a3}. Thus, since $(-1)^{\dagger} = 4$, the adelic image of $E$ is defined at level $7 \cdot 4 = 28$ by Algorithm \ref{alg-main}. 
    
    We can compute $G_{E, 28}$ via Algorithm \ref{alg-main}. However, reducing this image modulo $14$, we find that $[\mathcal{N}_{-7, 0}(14) : G_{E, 14}] = 2$, so by Theorem \ref{thm-levelofdef}, the adelic image of $E$ is defined at level $14 = 7 \cdot 2$. In this case, the adelic level of definition produced by our algorithm is not minimal because there is an entanglement between the $2$- and $7$-division fields of $E$, rather than just between the $4$- and $7$-division fields of $E$. This occurs because $\Q(E[2]) = \Q(\sqrt{7})$, and since $\sqrt{-1}$ and $\sqrt{-7}$ are contained in $\Q(E[7])$ by Lemma \ref{lem-twistofsimplest} and Thm. \ref{sqrt-general}, respectively, we also have $\Q(\sqrt{7}) \subseteq \Q(E[7])$.
\end{example}

\begin{remark}
    Let $E/\Q$ be an elliptic curve with CM by $\Of$ and $j(E) \neq 0, 1728$, and let $M :=\ell^{n_{E, \ell}} N^{\dagger}$ be the adelic level of definition for $E$ found using Algorithm \ref{alg-main}. From the examples above, it is clear that we can determine the minimal level of definition for $E$ by finding the smallest divisor $d$ of $M$ for which $[\mathcal{N}_{\delta, \phi}(d) : G_{E,d}] = 2$, which is not a computationally intensive process.
\end{remark}

\section{Levels of Differentiation}\label{sec-levelsofdiff}

In Definition \ref{def-level-of-dif}, we define an \emph{adelic level of differentiation} for an elliptic curve $E/\Q(j_{K, f})$ with CM by $\Of$ to be a positive integer $M \geq 2$ such that $M$ is an adelic level of definition for $E$, and if $E'/\Q(j_{K,f})$ is another elliptic curve with CM by $\Of$, then $G_E$ is conjugate to $G_{E'}$ in $\GL(2, \widehat{\Z})$ only if $G_{E, M}$ is conjugate to $G_{E', M}$ in $\GL(2, \Z/M\Z)$.

For $E/\Q$ with CM by $\Of$ and $j(E) \neq 0, 1728$, we have shown that if $N$ is a non-zero square-free integer such that $E^N$ is an $\ell$-simplest CM curve where $\gcd(\ell, N^\dagger) = 1$, then $M := \ell^{n_{E, \ell}} N^{\dagger}$ is an adelic level of definition for $E$ (see Rem. \ref{rem-notation} for notation). We wish to show that $M$ is an adelic level of differentiation as well.

% We begin by proving that conjugation at the level of the normalizer implies conjugation at the level of the Cartan.

% \begin{lemma}\label{lem-grptheoryconj}
%     Let $\mathcal{G}$ be a group, let $G_1, G_2$ be subgroups of $\mathcal{G}$, and let $H_i \subseteq G_i$ be subgroups of index $2$ for $i = 1, 2$. If $G_1$ and $G_2$ are conjugate in $\mathcal{G}$, then $H_1$ and $H_2$ are conjugate in $\mathcal{G}$.
% \end{lemma}

% \begin{proof}
%     Suppose that $G_2 = g G_1 g^{-1}$ for $g \in \mathcal{G}$. This conjugation induces an isomorphism $G_1/H_1 \to G_2/H_2$ defined by $g_1 \bmod H_1 \mapsto g g_1 g^{-1} \bmod H_2$ for $g_1 \in G_1$. For $g_1 \in G_1 \setminus H_1$, $g_1 \bmod H_1$ is a generator for $G_1/H_1$, and is sent to a generator $g g_1 g^{-1} \bmod H_2$, so $g g_1 g^{-1} \in G_2 \setminus H_2$. Therefore $H_2 = g H_1 g^{-1}$.
% \end{proof}

% \begin{corollary}\label{cor-conj}
%     Let $E/\Q(j_{K,f})$ and $E'/\Q(j_{K,f})$ be elliptic curves with CM by $\Of$. For positive integer $M \geq 3$, if $G_{E, M}$ is conjugate to $G_{E',M}$ in $\GL(2, \Z/M\Z)$, then $G_{E/K, M}$ is conjugate to $G_{E'/K, M}$ in $\GL(2, \Z/M\Z)$.
% \end{corollary}

% \begin{proof}
%     If $M \geq 3$, then $G_{E/K, M}$ and $G_{E'/K, M}$ are index $2$ subgroups of $G_{E,M}$ and $G_{E', M}$, respectively. The result follows from \ref{lem-grptheoryconj}.
% \end{proof}

We first characterize $N^{\dagger}$ for $N$ non-zero and square-free.

\begin{lemma}\label{lem-structureofN}
    Let $N$ be a non-zero, square-free integer, and consider $N^{\dagger}$ as in Remark \ref{rem-notation}. Then, $N^{\dagger} = 2^k \cdot N'$ where $N' := (N^{\dagger})^{\text{sf}}$ is the odd, square-free part of $N^{\dagger}$, and $k \in \{0, 2, 3\}$. More specifically,
    \begin{itemize}
        \item $N^{\dagger} = N'$ if $N \equiv 1 \bmod 4$, or
        \item $N^{\dagger} = 4 \cdot N'$ if $N \equiv 3 \bmod 4$, or
        \item $N^{\dagger} = 8 \cdot N'$ if $N \equiv 2 \bmod 4$.
    \end{itemize}
\end{lemma}

\begin{proof}
    In Theorem \ref{kronecker-weber}, we defined 
    \[ N^{\dagger} = \disc(\mathcal{O}_{\Q(\sqrt{N})}) =\begin{cases} |N| & \text{ if } N \equiv 1 \bmod 4, \\
|4N| & \text{ if } N \equiv 2,3 \bmod 4. \\
\end{cases}\]
The properties follow by inspection of the different cases for $N$.

% If $N \equiv 1 \bmod 4$, then $N^{\dagger} = |N|$. Since $N$ is square-free, we have that $N' := |N|$ is also square-free in this case, so $N^{\dagger} = N' = (N^{\dagger})^{\text{sf}}$.

% If $N \equiv 3 \bmod 4$, then $N^{\dagger} = 4 \cdot |N|$, where $|N|$ is a odd, positive integer. Again, since $N$ is square-free, so is $N' := |N|$, in which case 

\end{proof}

We now prove that the level of definition produced by Algorithm \ref{alg-main} shares all prime divisors with the minimal adelic level of definition.

\begin{proposition}\label{lem-primediv}
    Let $E/\Q$ be an elliptic curve with CM by $\Of$ with $j(E) \neq 0, 1728$. Let $N$ be a non-zero square-free integer such that $E^N$ is a simplest CM curve for the prime $\ell$ where $\gcd(\ell, N^{\dagger}) = 1$, and let $n := n_{E, \ell}$, so $M := \ell^n N^{\dagger}$ is an adelic level of definition for $E$. If $M_E$ is the minimal level of definition for $E$, then a prime $p$  is a divisor of $M_E$ if and only if $p$ is a divisor of $M$.
\end{proposition}

\begin{proof}

    Since $M_E \mid M$ by Prop. \ref{prop-minimalleveldivideslevelsofdef}, we suppose towards a contradiction that there is a prime $p$ such that $p \mid M$ but $p \nmid M_E$. Since $\ell$ divides all levels of definition by Cor. \ref{cor-indexisalways2}, we know that $p \neq \ell$. In particular, $p \mid N^{\dagger}$ since $\gcd(\ell, N^{\dagger}) = 1$.
    
    Let $M_E = \ell^m N'$ with $1 \leq m \leq n$ and $N' \mid N^{\dagger}$. If $N' = 1$, then $M_E = \ell^m$, so $E$ is a simplest CM curve for the prime $\ell$. Since $E^N$ is also an $\ell$-simplest CM curve, and $\gcd(\ell, N^{\dagger}) = 1$, we have $E^N \cong E$ and $N = 1$ by Cor. \ref{cor-notsimplest}. Thus $M = \ell^n$, a contradiction.

    Assume $N' > 1$. Define $d := \nu_{p}(N^{\dagger})$, the $p$-adic valuation of $N^{\dagger}$, and define $P := N^{\dagger}/p^d$. Note that we must have $N' \mid P$ and $P > 1$. Also, by Lemma \ref{lem-structureofN}, $p^d \geq 3$. Since $M_E = \ell^m N'$ is a level of definition for $E$ and $M_E \mid \ell^n P$, it follows that $\ell^n P$ is a level of definition for $E$ as well. Therefore $[\mathcal{N}_{\delta, \phi}(\ell^n P) : G_{E, \ell^n P}] = 2$.
    
    By Prop. \ref{prop-sameindex}, we have that $[\mathcal{C}_{\delta, \phi}(\ell^n P) : G_{E/K, \ell^n P}] = 2$. However, since $P \mid N^{\dagger}$ and $[\mathcal{N}_{\delta, \phi}(N^{\dagger}) : G_{E, N^{\dagger}}] = 1$, we have $[\mathcal{C}_{\delta, \phi}(P) : G_{E/K, P}] = [\mathcal{N}_{\delta, \phi}(P) : G_{E, P}] = 1$. Since we also have $[\mathcal{C}_{\delta, \phi}(\ell^n) : G_{E/K, \ell^n}] = 1$, that means that $G_{E/K, \ell^n P}$ is an index $2$ subgroup of $G_{E/K, \ell^n} \times G_{E/K, P}$. By Cor. \ref{cor-galoistheoryofcompositum}, $[\Q(E[\ell^n]) \cap \Q(E[P]) : K] = 2$. Further, by Theorem \ref{thm-entanglement}, we have $K(\sqrt{N}) = \Q(E[\ell^n]) \cap \Q(E[N^{\dagger}])$, and since $\Q(E[P]) \subseteq \Q(E[N^{\dagger}])$, it must be that $K(\sqrt{N}) = \Q(E[\ell^n]) \cap \Q(E[P])$.

    We now wish to show that $\Q(E[p^d])$ and $\Q(E[P])$ have a non-trivial entanglement. If $p$ is odd, then $d = 1$, and we may choose $p^{\ast} := (-1)^{(p-1)/2} p$, so $\Q(\sqrt{p^{\ast}}) \subseteq \Q(\zeta_{p}) \subseteq \Q(E[p]) = \Q(E[p^d])$, and $(N/p^{\ast})^{\dagger} = P$. If $p = 2$, then $d \geq 2$, and $N \equiv 2, 3 \bmod 4$. If $N \equiv 3 \bmod 4$, we choose $p^{\ast} := -1$, and if $N \equiv 2 \bmod 4$,  there is some choice of $p^{\ast} \in \{\pm 2\}$ so that, in all cases, $\Q(\sqrt{p^{\ast}}) \subseteq \Q(\zeta_{p^d}) \subseteq \Q(E[p^d])$ and $(N/p^{\ast})^{\dagger} = P$.

    Since $(N/p^{\ast})^{\dagger} = P$ in all cases, it must be that $\Q(\sqrt{N/p^{\ast}}) \subseteq \Q(\zeta_{P}) \subseteq \Q(E[P])$, and therefore $\Q(\sqrt{N}), \Q(\sqrt{N/p^{\ast}}),$ and $\Q(\sqrt{p^{\ast}})$ are contained in $\Q(E[P])$. Further, since $p^d \geq 3$, we have that $K \subseteq \Q(E[p^d]) \cap \Q(E[P])$, and so $K(\sqrt{p^{\ast}}) \subseteq \Q(E[p^d]) \cap \Q(E[P])$. The fact that $K(\sqrt{p^{\ast}})/K$ is a quadratic extension follows because $\ell$ is the unique prime dividing $\Delta_K = \disc (\mathcal{O}_K)$, and $\ell \neq p$.

    This means that $[G_{E/K, p^d} \times G_{E/K, P} : G_{E/K, N^{\dagger}}] = [\Q(E[p^d]) \cap \Q(E[P]) : K] \geq 2$. However, $G_{E/K, N^{\dagger}} \cong \mathcal{C}_{\delta, \phi}(N^{\dagger})$, so $[G_{E/K, p^d} \times G_{E/K, P} : G_{E/K, N^{\dagger}}] = 1$, a contradiction.    
\end{proof}

The following two results together prove that, for the cases we wish to consider, conjugation at the level of the normalizer implies conjugation at the level of the Cartan.

\begin{proposition}\label{prop-MfixingN-redux}
    Let $E/\Q$ be an elliptic curve with CM by $\Of$, let $m \in \Z$ be the square-free integer such that $K = \Q(\sqrt{m})$, and suppose $M > 2$ is a positive integer divisible by $|\Delta_K| = m^{\dagger}$. Let $G_{E, M} := \im \rho_{E, M}$ be the mod-$M$ Galois image of $E$ and let $G_{E/K, M} := \im \rho_{E/K, M}$. Using notation as in Prop. \ref{prop-MfixingN}, we have that
    \[ G_{E/K, M} = \{ g \in G_{E, M} : \det(g) \in \mathcal{A}^M_{m}\}. \]
    \end{proposition}

\begin{proof}
    Note that $K \subseteq \Q(\zeta_M) \subseteq \Q(E[M])$ by Theorem \ref{kronecker-weber}. By hypothesis, $G_{E, M} \cong \Gal(\Q(E[M])/\Q)$ and $G_{E/K, M} \cong \Gal(\Q(E[M])/K)$. Further, it follows that 
    \[ \Gal(\Q(E[M])/\Q(\zeta_M)) \cong G_{E, M} \cap \SL(2, \Z/M\Z) = \{ g \in G_{E, M} : \det(g) = 1 \}, \] 
    and hence $\Gal(\Q(E[M])/K) \cong  \{ g \in G_{E, M} : \det(g) \in \mathcal{A}_{m}^{M}  \}$ by Galois theory. Therefore $G_{E/K, M} = \{ g \in G_{E, M} : \det(g) \in \mathcal{A}_{m}^{M}  \}$ as desired.
\end{proof}

\begin{corollary}\label{cor-conjugate-at-cartan}
    Let $E/\Q$ and $E'/\Q$ be elliptic curves with CM by $\Of$, and let $M > 2$ be a positive integer divisible by, but not equal to, $|\Delta_K|$. If $G_{E, M}$ is conjugate to $G_{E', M}$ in $\GL(2, \Z/M\Z)$, then $G_{E/K, M}$ is conjugate to $G_{E'/K, M}$ in $\GL(2, \Z/M\Z)$.
\end{corollary}

\begin{proof}
    Let $m$ be the square-free integer such that $K = \Q(\sqrt{m})$. Since $K/\Q$ is a quadratic extension and $M \neq |\Delta_K|$, it follows that $\mathcal{A}_{m}^{M}$ is an index $2$ subgroup of $\det(G_{E, M}) = \det(G_{E', M})$ by Galois theory. By Proposition \ref{prop-reconstructimage}, we can choose $C_M \in G_{E, M} \setminus G_{E/K, M}$ and $C'_M \in G_{E', M} \setminus G_{E'/K, M}$ such that $G_{E, M} = \langle C_M, G_{E/K, M} \rangle$ and $G_{E', M} = \langle C'_M, G_{E'/K, M} \rangle$, and by Prop. \ref{prop-MfixingN-redux}, we can make our choice so that $\det(C_M)$ and $\det(C'_M)$ are not contained in $\mathcal{A}_{m}^{M}$. Therefore any conjugation sending $G_{E, M}$ to $G_{E', M}$ must also send $G_{E/K, M}$ to $G_{E'/K, M}$, since conjugation preserves determinants.
\end{proof}

In the following two propositions we show that images resulting from certain twists of simplest CM curves do not result in conjugate images.

\begin{proposition}\label{prop-notconjugate-level8}
    Let $E/\Q$ and $E'/\Q$ be elliptic curves with CM by $\Of$ and $j_{K, f} \neq 0, 1728$. Let $\ell$ be the unique prime dividing $\Delta_K$, and suppose that $\ell \neq 2$.  Let $P := -2 \cdot N$ and $P' := 2 \cdot N$ for $N$ a square-free integer satisfying $N \equiv 1 \bmod 4$.     If the twists $E^P$ and $(E')^{P'}$ are $\ell$-simplest CM curves such that $\ell$ is relatively prime to both $P^{\dagger}$ and $(P')^{\dagger}$, then for $M := \ell P^{\dagger}$, $G_{E, M}$ is not conjugate to $G_{E', M}$ in $\GL(2, \Z/ M\Z)$.
\end{proposition}

\begin{proof}
    Note that $P^{\dagger} = (P')^{\dagger} = 8 N^{\dagger}$ by Lemma \ref{lem-structureofN}, $n_{E, \ell} = 1$ by Rem. \ref{rem-notation}, and by Cor. \ref{cor-levelofdef}, $M = 8 \ell N^{\dagger}$ is an adelic level of definition for both $E$ and $E'$. By Theorem \ref{thm-entanglement}, we have that 
    \[ \Q(E[\ell]) \cap \Q(E[8 N^{\dagger}]) = K(\sqrt{-2 N}) \quad \text{ and } \quad  \Q(E'[\ell]) \cap \Q(E'[8 N^{\dagger}]) = K(\sqrt{2 N}).\]
    % By Theorem \ref{thm-oddprimedividingdisc}(a), and without loss of generality, we can represent $G_{E^P, \ell} = \langle C, G_{\delta, 0}^{2, 1}(\ell) \rangle$ and $G_{E', \ell} = \langle C', G_{\delta, 0}^{2, 1}(\ell) \rangle$ where $C, C' \in \{c_{1}, c_{-1}\}$ and $G_{\delta, 0}^{2, 1}(\ell)$ is the group defined as in Thm. \ref{thm-oddprimedividingdisc}(a). Thus, 
     We first note that, since $\ell \neq 2$ and $E^{P}$ and $(E')^{P'}$ have non-maximal $\ell$-adic images by hypothesis, the classification in Theorem \ref{thm-j0ell3alvaro}(a) shows that
    \[ G_{E^P/K, \ell} \cong G_{(E')^{P'}/K, \ell} \cong J_{\delta, 0} = \left\langle \left\{ \left(\begin{array}{cc} a^2 & b\\ \delta b & a^2 \\ \end{array}\right): a \in (\Z/\ell\Z)^\times, b \in \Z/\ell\Z \right\} \right\rangle.\]
    
    By Thm. \ref{thm-cartan-image} and from the discussion in Section \ref{subsec-alg-main}, we have that
    \[ G_{E/K, M} \cong \langle (g_{\ell}, g_{P^{\dagger}}), H_{\ell} \times H_{P^{\dagger}}^{-2N} \rangle \]
    and 
    \[ G_{E'/K, M} \cong \langle (g_{\ell}, g_{P^{\dagger}}), H_{\ell} \times H_{P^{\dagger}}^{2N} \rangle, \]
    where $H_{\ell} = J_{\delta, 0}$, $H_{P^{\dagger}}^{m} = \{ g \in \mathcal{C}_{\delta, \phi}(P^{\dagger}) : \det(g) \in \mathcal{A}_m^{P^{\dagger}} \}$
    for $m = \pm 2N$, and where we choose $g_{\ell} \in \mathcal{C}_{\delta, \phi}(\ell) \setminus H_{\ell}$ and $g_{P^{\dagger}} \in \mathcal{C}_{\delta, \phi}(P^{\dagger}) \setminus (H_{P^{\dagger}}^{-2N} \cup H_{P^{\dagger}}^{2N})$. %Also, the identifications of the modulo $M$ images above explicitly come from the isomorphism between $\mathcal{C}_{\delta, \phi}(M)$ and $\mathcal{C}_{\delta, \phi}(\ell) \times \mathcal{C}_{\delta, \phi}( P^{\dagger})$ induced by the Chinese remainder theorem.    
    %Note that $g_{M}$ can be chosen in such a way since $H_{M}^{-2N}$ and $H_{M}^{2N}$ are index $2$ subgroups of $\mathcal{C}_{\delta, \phi}(M)$, and therefore their union is not all of $\mathcal{C}_{\delta, \phi}(M)$.
    Note also that $\mathcal{A}_{2N}^{P^{\dagger}} \neq \mathcal{A}_{-2N}^{P^{\dagger}}$, and since conjugation preserves determinants, we cannot have $H_{P^{\dagger}}^{2N}$ conjugate to $H_{P^{\dagger}}^{-2N}$ in $\GL(2, \Z/P^{\dagger}\Z)$. 

    Suppose, towards a contradiction, that $G_{E/K, M}$ is conjugate to $G_{E'/K, M}$. Since we have chosen to represent the two groups so that they are identical in their mod-$\ell$ components, we may choose a matrix $B_{P^{\dagger}} \in \GL(2, \Z/P^{\dagger}\Z)$ such that 
    \[ (\Id, B_{P^{\dagger}})  \langle (g_{\ell}, g_{P^{\dagger}}), H_{\ell} \times H_{P^{\dagger}}^{-2N} \rangle (\Id, B_{P^{\dagger}})^{-1} = \langle (g_{\ell}, g_{P^{\dagger}}), H_{\ell} \times H_{P^{\dagger}}^{2N} \rangle. \]
    Since $g_{P^{\dagger}} \in \mathcal{C}_{\delta, \phi}(P^{\dagger}) \setminus (H_{P^{\dagger}}^{-2N} \cup H_{P^{\dagger}}^{2N})$, we can choose $g_{P^{+}}$ such that $\det(g_{P^{\dagger}}) \notin  \mathcal{A}_{2N}^{P^{\dagger}} \cup \mathcal{A}_{-2N}^{P^{\dagger}}$. Hence $B_{P^{\dagger}} g_{P^{\dagger}} B_{P^{\dagger}}^{-1}$ is not in $H_{P^{\dagger}}^{2N}$, so $B_{\dagger}H_{P^{\dagger}}^{-2N} B_{\dagger}^{-1} = H_{P^{\dagger}}^{2N}$, a contradiction.

    Therefore $G_{E/K, M}$ is not conjugate to $G_{E'/K, M}$ in $\GL(2, \Z/M\Z)$ and by Corollary \ref{cor-conjugate-at-cartan}, $G_{E, M}$ is not conjugate to $G_{E', M}$ in $\GL(2, \Z/M\Z)$.
    \end{proof}
    
    % Now, from Theorem \ref{thm-oddprimedividingdisc} and Algorithm \ref{alg-main}, we can write $G_{E, M} = \langle C, G_{E/K, M} \rangle$ and $G_{E', M} = \langle C', G_{E'/K, M} \rangle$, where $C, C' \in \{c_{1}, c_{-1} \}$ (see Section \ref{sec-notation} for the definition of $c_{\varepsilon}$ where $\varepsilon \in \{\pm 1\}$).

    % Since $\ell \neq 2$ and $j(E)$ is rational, it must be that $\ell \in \{3, 7, 11, 19, 43, 67, 163 \}$, so $\ell \equiv 3 \bmod 4$. Therefore, by Lemma \ref{lem-det-image}, no element in $\mathcal{C}_{\delta, \phi}(\ell)$ has determinant $-1 \bmod \ell$. Further, since $\ell \mid M$, no element in $\mathcal{C}_{\delta, \phi}(M)$ has determinant $-1 \bmod M$, including all elements in $G_{E/K, M}$ and $G_{E'/K, M}$. 

    % However, $\det(C) = \det(C') = -1$, and since conjugation preserves determinants, any conjugation between $G_{E, M}$ and $G_{E',M}$ would also be a conjugation between $G_{E/K, M}$ and $G_{E'/K, M}$. Therefore $G_{E,M}$ is not conjugate to $G_{E',M}$ in $\GL(2, \Z/M\Z)$.

\begin{proposition}\label{prop-notconjugate-level4and8}
    Let $E/\Q$ and $E'/\Q$ be elliptic curves with CM by $\Of$ and assume $j_{K, f} \neq 0, 1728$. Denote by $\ell$ the unique prime dividing $\Delta_K$, and assume $\ell \neq 2$.    Let $P := -N$ and $P' \in \{\pm 2N\}$ for $N$ a square-free integer satisfying $N \equiv 1 \bmod 4$, and suppose that 
    \begin{itemize}
        \item $E^P$ and $(E')^{P'}$ are $\ell$-simplest CM curves, and 
        \item $\gcd(\ell, P^{\dagger}) = 1$.
    \end{itemize}
Then for $M := \ell P^{\dagger}$, $G_{E, M }$ and $G_{E', M}$ are not conjugate in $\GL(2, \Z/ M\Z)$.
\end{proposition}

\begin{proof}
    From Lemma \ref{lem-structureofN}, we have that $M = \ell P^{\dagger} = 4 \ell N^{\dagger}$ and $(P')^{\dagger} = 8 N^{\dagger}$ so $\gcd(\ell, (P')^{\dagger}) = 1$. By Rem. \ref{rem-notation}, $n_{E, \ell} = 1$, and by Cor. \ref{cor-levelofdef}, $M$ and $2M$ are adelic levels of definition for $E$ and $E'$, respectively. If $G_{E', M} \cong \mathcal{N}_{\delta, \phi}(M)$ then we are done, so we assume that $M$ is also a level of definition for $E'$, hence $[\mathcal{N}_{\delta, \phi}(M) : G_{E', M}] = 2$.

    Since $\ell \neq 2$ and $E^P$ and $(E')^{P'}$ have non-maximal $\ell$-adic images by hypothesis, it follows from Theorem \ref{thm-oddprimedividingdisc}(a) that 
    \[ G_{E^P/K, \ell} \cong G_{(E')^{P'}/K, \ell} \cong J_{\delta, 0} = \left\langle \left\{ \left(\begin{array}{cc} a^2 & b\\ \delta b & a^2 \\ \end{array}\right): a \in (\Z/\ell\Z)^\times, b \in \Z/\ell\Z \right\} \right\rangle. \]
 Then by Thm. \ref{thm-cartan-image} and from the discussion in Section \ref{subsec-alg-main}, we have that
    \[ G_{E/K, M} \cong \langle (g_{\ell}, g_{P^{\dagger}}), H_{\ell} \times H_{P^{\dagger}} \rangle \quad \text{ and } \quad
 G_{E'/K, 2M} \cong \langle (g_{\ell}, g_{2P^{\dagger}}), H_{\ell} \times H_{2P^{\dagger}} \rangle, \]
    where $H_{\ell} = J_{\delta, 0}$, $H_{P^{\dagger}} = \{ g \in \mathcal{C}_{\delta, \phi}(P^{\dagger}) : \det(g) \in \mathcal{A}_{P}^{P^{\dagger}} \}$, and $H_{2P^{\dagger}} = \{ g \in \mathcal{C}_{\delta, \phi}(2P^{\dagger}) : \det(g) \in \mathcal{A}_{P'}^{2P^{\dagger}} \}$ for some $g_{\ell} \in \mathcal{C}_{\delta, \phi}(\ell) \setminus H_{\ell}$, $g_{P^{\dagger}} \in \mathcal{C}_{\delta, \phi}(P^{\dagger}) \setminus H_{P^{\dagger}}$, and $g_{2P^{\dagger}} \in \mathcal{C}_{\delta, \phi}(2P^{\dagger}) \setminus H_{2P^{\dagger}}$.

   For ease of notation, we let $\pi \colon \mathcal{C}_{\delta, \phi}(2M) \to \mathcal{C}_{\delta, \phi}(M)$ and $\pi' \colon \mathcal{C}_{\delta, \phi}(2P^{\dagger}) \to \mathcal{C}_{\delta, \phi}(P^\dagger)$ denote the natural projection maps between their respective Cartan subgroups.  By Proposition \ref{prop-CRTcompatible}, we have
   \[ \pi(G_{E'/K, 2M}) = G_{E'/K, M} \cong \langle (g_{\ell}, \pi'(g_{2P^{\dagger}})), H_{\ell} \times \pi'(H_{2P^{\dagger}}) \rangle. \]

   Since $\Q(\zeta_{8}), \Q(\zeta_{N^{\dagger}}) \subseteq \Q(\zeta_{8 N^{\dagger}})$, by Theorem \ref{kronecker-weber}, $\Q(\sqrt{P}, \sqrt{P'}) \subsetneq \Q(\zeta_{8 N^{\dagger}})$, so there exists a choice of $\sigma \in \Gal(\Q(\zeta_{8 N^{\dagger}})/\Q)$ such that $\sigma(\sqrt{P'}) = \sqrt{P'}$ but $\sigma(\sqrt{P}) \neq \sqrt{P}$. Thus, there is an element $B \in H_{2P^{\dagger}}$ such that $\det(\pi'(B)) \notin \det(H_{P^{\dagger}}) = \mathcal{A}_{P}^{P^{\dagger}}$. Therefore $\mathcal{A}_{P}^{P^{\dagger}} \neq \det(\pi'(H_{2P^{\dagger}}))$, and so $H_{P^{\dagger}}$ and $\pi'(H_{2 P^{\dagger}})$ are not conjugate.

   Suppose, towards a contradiction, that $G_{E/K, M}$ is conjugate to $G_{E'/K, M}$. We first note that $\det(\pi'(H_{2P^{\dagger}}))$ must be an index $2$ subgroup of $\det(\mathcal{C}_{\delta, \phi}(P^{\dagger}))$, otherwise $G_{E/K, M}$ and $G_{E'/K, M}$ would be of different sizes. Since we have chosen to represent the two groups so that they are identical in their mod-$\ell$ components, our hypothesis implies there is a choice of matrix $B_{P^{\dagger}} \in \GL(2, \Z/P^{\dagger}\Z)$ such that 
    \[ (\Id, B_{P^{\dagger}})  \langle (g_{\ell}, g_{P^{\dagger}}), H_{\ell} \times H_{P^{\dagger}} \rangle (\Id, B_{P^{\dagger}})^{-1} = \langle (g_{\ell}, \pi'(g_{2P^{\dagger}})), H_{\ell} \times \pi'(H_{2P^{\dagger}}) \rangle. \]

   Since $\det(H_{P^{\dagger}}) = \mathcal{A}_{P}^{P^{\dagger}}$ and $\det(\pi'(H_{2P^{\dagger}}))$ are both of index $2$ but not equal, we can choose $g_{P^{\dagger}}$ such that $\det(g_{P^{\dagger}}) \notin \mathcal{A}_{P}^{P^{\dagger}} \cup \det(\pi'(H_{2P^{\dagger}}))$. Thus, $B_{P^{\dagger}} g_{P^{\dagger}} B_{P^{\dagger}}$ cannot be contained in $\pi'(H_{2P^{\dagger}})$, so conjugation by $B_{P^{\dagger}}$ must send $H_{P^{\dagger}}$ to $\pi'(H_{2P^{\dagger}})$, a contradiction. It follows that $G_{E/K, M}$ is not conjugate to $G_{E'/K, M}$ in $\GL(2, \Z/M\Z)$, and therefore $G_{E, M}$ is not conjugate to $G_{E', M}$ in $\GL(2, \Z/M\Z)$ by Cor. \ref{cor-conjugate-at-cartan}.
\end{proof}

We conclude with the main result of this section.

\begin{theorem}
    Let $E/\Q$ be an elliptic curve with CM by $\Of$ with $j(E) \neq 0, 1728$. Let $N$ be a non-zero square-free integer such that $E^N$ is a simplest CM curve for the prime $\ell$ where $\gcd(\ell, N^{\dagger}) = 1$, and let $n := n_{E, \ell}$. Then $\ell^n N^{\dagger}$ is an adelic level of differentiation for $E$.
\end{theorem}

\begin{proof}

    Let $M := \ell^n N^{\dagger}$ be defined as above, which is the level of definition for $E$ as in Algoritgm \ref{alg-main}. Suppose that $E'$ is an elliptic curve with CM by $\Of$ and that $G_{E, M}$ and $G_{E',M}$ are conjugate as subgroups of $\GL(2, \Z/M\Z)$. We wish to prove that $E \cong_{\Q} E'$, which implies $G_{E}$ is conjugate to $G_{E'}$ in $\GL(2, \widehat{\Z})$.

    Since $M$ is an adelic level of definition for $E$, by Thm \ref{thm-main} part (1), we have that $[\mathcal{N}_{\delta, \phi}(M) : G_{E,M}] = 2$, and hence $[\mathcal{N}_{\delta, \phi}(M) : G_{E',M}] = 2$. So $M$ is also an adelic level of definition for $E'$ by Thm. \ref{thm-levelofdef}. By a similar argument, if $m$ is an adelic level of definition for either $E$ or $E'$, and $m \mid M$, then $m$ is an adelic level of definition for both $E$ and $E'$. In particular, if $M_E$ and $M_{E'}$ denote the minimal levels of definition for $E$ and $E'$, respectively, then $M_E = M_{E'}$.
    
    Suppose we have $M' := \ell^n P^{\dagger}$ for a non-zero square-free integer $P$, where $\gcd(\ell, P^{\dagger}) = 1$ and $(E')^P$ is an $\ell$-simplest CM curve. Then $M'$ is a level of definition for $E'$ by Cor. \ref{cor-levelofdef}.

    From Lemma \ref{lem-primediv}, it follows that $M$, $M'$, and $M_{E} = M_{E'}$ share all prime divisors. In particular, by Lemma \ref{lem-structureofN}, if $N'$ and $P'$ are the odd, square-free parts of $N^{\dagger}$ and $P^{\dagger}$, respectively, then  $N' = P'$ and $N' \equiv 1 \bmod 4$. Further, since $P^{\dagger}$ and $N^{\dagger}$ share all prime divisors, we are in one of three cases. Without loss of generality, we have that
    \begin{enumerate}
        \item $N = P$ and $N^{\dagger} = P^{\dagger}$, or
        \item $N \neq P$ and $N^{\dagger} = P^{\dagger} = 8 N'$, or
        \item $N \neq P$, $N^{\dagger} = 4 N'$, and $P^{\dagger} = 8 N'$.
    \end{enumerate}
    
    In the first case, if $N = P = 1$, then $N^{\dagger} = P^{\dagger} = 1$, $E = E^N$ and $E' = (E')^P$ are both $\ell$-simplest CM curves, and $M = \ell^n$. Then $G_{E, \ell^n}$ is conjugate to $G_{E', \ell^n}$, so it follows from Theorem \ref{thm-40simplestcurves}(2e) that $E \cong_{\Q} E'$, as desired.
    
    In the second case, we can assume, without loss of generality, that $N = 2N'$ and $P = -2N'$, and in the third case, we can show that $N = -N''$ and $P = \pm 2 N''$ for some positive, square-free integer $N'' \equiv 1 \bmod 4$. Note that in both the second and third cases, $N^{\dagger}$ and $P^{\dagger}$ are even, so $\ell \neq 2$ since $\gcd(\ell, N^{\dagger}) = \gcd(\ell, P^{\dagger}) = 1$.

    Thus, the second and third cases are those described in Propositions \ref{prop-notconjugate-level8} and \ref{prop-notconjugate-level4and8}, respectively. In either case, we have shown that $G_{E, M}$ is not conjugate to $G_{E', M}$, which contradicts our hypothesis.

    Thus, we assume $N = P \neq 1$ and $N^{\dagger} = P^{\dagger}$, so $M = M'$ is a level of definition for both curves, and the steps of Algorithm \ref{alg-main} can be applied to construct the Galois image modulo $M$ of both curves.

    Recall that we assume $G_{E, M}$ is conjugate to $G_{E', M}$. For $C \in \{E, E'\}$, we define $G_{C, M, N} := \{ g \in G_{C, M} : \det(g) \in \mathcal{A}_{N}^{M} \}$, the subgroup of $G_{C, M}$ isomorphic to $\Gal(\Q(C[M])/\Q(\sqrt{N}))$. We also define $G_{C/K, M, N} := G_{C, M, N} \cap G_{C/K, M}$, which is isomorphic to $\Gal(\Q(C[M])/K(\sqrt{N}))$. Note that $G_{E, M, N}$ is conjugate to $G_{E', M, N}$ and $G_{E/K, M, N}$ is conjugate to $G_{E'/K, M, N}$ in $\GL(2, \Z/M\Z)$ since determinants are preserved by conjugation (see Proposition \ref{prop-MfixingN-redux}).
    
    Proposition \ref{prop-ellfixingN} shows that $G_{E/K, M, N} \bmod{\ell^n} \cong G_{E^N/K, \ell^n}$ and $G_{E'/K, M, N} \bmod{\ell^n} \cong G_{(E')^N/K, \ell^n}$, which are conjugate in $\GL(2,\Z/\ell^n\Z)$ since conjugation is preserved by reduction.

    Now, let $\sigma_E \in \Gal(\Q(E[M])/\Q)$ be an automorphism that does not fix $K$ and does not fix $\Q(\sqrt{N})$. Under an appropriate choice of basis, this corresponds to an element $g_E := \rho_{E,M}(\sigma_E) \in G_{E, M}$ that is not in $G_{E, M, N}$ or in $\mathcal{C}_{\delta, \phi}(M)$. Note that, since $\sigma_E$ does not fix $\Q(\sqrt{N})$, $\chi_N(\sigma_E) = -1$, where $\chi_N$ is the quadratic character  corresponding to the twist $E^N = E^{\chi_N}$. Thus, by Proposition \ref{cpx_conj_lemma2}, we have that $(\chi_{N}(\sigma_E) \cdot \Id) \cdot g_E = -\Id \cdot g_E \in G_{E^N, M}$. 
    
    Now, let $g_{E, \ell^n} := \pi_{M, \ell^n}(-\Id \cdot g_E)$. Note that, since $g_E \notin \mathcal{C}_{\delta, \phi}(M)$, $g_{E, \ell^n} \notin \mathcal{C}_{\delta, \phi}(\ell^n)$, and therefore $G_{E^N, \ell^n} = \langle g_{E, \ell}, G_{E^N, \ell^n} \rangle$ by Proposition \ref{prop-reconstructimage}.

    Let $B \in \GL(2, \Z/M\Z)$ be a conjugation matrix sending $G_{E, M}$ to $G_{E', M}$. We define $g_{E'} := B g_E B^{-1} \in G_{E', M}$, and note that $g_{E'}$ must also satisfy $g_{E'} \notin G_{E', M, N}$ and $g_{E'} \notin \mathcal{C}_{\delta, \phi}(M)$. Thus, for $g_{E', \ell^n} := \pi_{M, \ell^n}(g_{E'})$, we have that $g_{E, \ell^n} \notin \mathcal{C}_{\delta, \phi}(\ell^n)$, so $G_{(E')^N, \ell^n} = \langle g_{E', \ell^n}, G_{(E')^N/K, \ell^n} \rangle$. Finally, by construction, we have $G_{(E')^N, \ell^n}$ conjugate to $G_{E^N, \ell^n}$ in $\GL(2, \Z/\ell^n \Z)$. Since $E^N$ and $(E')^N$ are $\ell$-simplest, and $\ell^n$ is an $\ell$-adic level of differentiation, it follows by Theorem \ref{thm-40simplestcurves}(2e) that $E^N \cong_{\Q} (E')^N$, so $E \cong_{\Q} E'$ by the theory of twists.
    \end{proof}

\bibliography{bibliography}
\bibliographystyle{plain}

\end{document}